\documentclass[a4paper,twoside,11pt]{amsart}
\usepackage{amssymb, amsfonts,amsmath}
\usepackage{graphicx}
\usepackage[enableskew,vcentermath]{youngtab}
\usepackage{ytableau}
\usepackage{etex}
\usepackage[all]{xy}
\usepackage[small,nohug,heads=vee]{diagrams}
\newtheorem{theorem}{Theorem}
\newtheorem{corollary}{Corollary}
\newtheorem{lemma}{Lemma}
\newtheorem{proposition}{Proposition}
\theoremstyle{definition}
\newtheorem{definition}{Definition}
\newtheorem{example}{Example}
\theoremstyle{remark}
\newtheorem{remark}{Remark}

\newcommand{\gl}{\mathfrak{gl}}

\newcommand{\gr}{\operatorname{gr}}

\newcommand{\Diag}{\operatorname{Diag}}
\newcommand{\ad}{\operatorname{ad}}
\newcommand{\Tab}{\operatorname{Tab}}
\newcommand{\rank}{\operatorname{rank}}
\newcommand{\Span}{\operatorname{span}}
\newcommand{\Id}{\operatorname{Id}}
\newcommand{\pf}{\operatorname{pf}}
\renewcommand{\mod}{\operatorname{mod}}
\newcommand{\sll}{\mathfrak{sl}}

\renewcommand{\P}{\mathbb P}

\newcommand{\R}{\mathbb R}

\newcommand{\mg}{\mathfrak g}
\newcommand{\vf}{\varphi}
\newcommand{\Hom}{\operatorname{Hom}}
\newcommand{\pg}{\mathfrak p}

\renewcommand{\S}{\mathcal S}
\begin{document}

\title{On local geometry of vector distributions with given Jacobi symbols}
\author
{Boris Doubrov
\address{Belarussian State University, Nezavisimosti Ave.~4, Minsk 220030, Belarus;
 E-mail: doubrov@bsu.by}
 \and Igor Zelenko
\address{Department of Mathematics, Texas A$\&$M University,
   College Station, TX 77843-3368, USA; E-mail: zelenko@math.tamu.edu}
}
\thanks{Research of I. Zelenko is supported by NSF grant DMS-1406193.}
\begin{abstract}
Based on the ideas of Optimal Control, we introduce the new basic characteristic of a bracket generating distribution, the Jacobi symbol.
 In contrast to the classical Tanaka symbol, the set of Jacobi symbols is discrete and classifiable. We give an explicit and unified algebraic procedure for the construction of the canonical frames (the structure of absolute parallelism) for all distribution with given Jacobi symbol.
 %It extends significantly the scope of distributions for which  the canonical frames  can be constructed explicitly and in an unified way.
 We describe all Jacobi symbols for which this procedure ends up in a finite number of steps (i.e. all Jacobi symbols of finite type) and distinguish the symplectically flat distributions with given Jacobi symbol, which in the case of finite type are the maximally symmetric distribution among all distribution with given Jacobi symbol. Also, for the most of Jacobi symbols we relate the prolongation procedure and the resulting prolongation algebra, which is the algebra of infinitesimal symmetries of the corresponding symplectically flat distribution, to the standard (in the sense of Kobayashi and Sternberg) prolongation of certain subspace in the algebra of infinitesimal symmetries of the flat (unparametrized) curve of flags associated with the Jacobi symbol. In particular, we show
 that as such subspace one can take the algebra of infinitesimal symmetries of the flat curve considered as parametrized curve (i.e. only symmetries preserving the parametrization are considered). Finally, we give an upper bounds for the size of the resulting prolongation algebra in terms of certain spaces of polynomials vanishing on certain projective varieties.
For some class of Jacobi symbols we are able to identify the resulting prolongation algebra with such spaces of polynomials. In this way we get the description of the prolongation algebra for all Jacobi symbols of finite type appearing in rank $2$ and rank $3$ distributions and also for most of such Jacobi symbols appearing in rank $4$ distributions. In particular, in the case of rank $3$ distributions this description is given in terms of the tangential developable and the secant varieties of a rational normal curve.
\end{abstract}
 %in the case of distributions of rank $2$ and $3$, we explicitly compute the resulting algebra, which is the algebra of infinitesimal symmetries of the most symmetric distribution with the given Jacobi symbol.

\maketitle\markboth{Boris Doubrov and Igor Zelenko}{On local geometry of vector distributions with given Jacobi symbols}

\section{Introduction}
\setcounter{equation}{0}
\setcounter{theorem}{0}
\setcounter{lemma}{0}
\setcounter{proposition}{0}
\setcounter{definition}{0}
\setcounter{corollary}{0}
\setcounter{remark}{0}
\setcounter{example}{0}
This paper is a culmination  of a long-standing program for study of the geometry
of non-holonomic vector distributions via a novel approach called the \emph{symplectification/linearization procedure}. This approach takes its origin in Optimal Control and  in essence consists of the reduction
%in essence
of the equivalence problem for distributions
%for
%geometric
%structures
to the (extrinsic) differential geometry of curves of isotropic and coisotropic subspaces  (or, shortly, of curves of symplectic flags) in a linear symplectic space.
%(or, more generally, submanifolds).
The extrinsic geometry of
 %submanifolds
 such curves
 %in flag varieties
is simpler in many respects than the original equivalence problem.
%in certain flag varieties
%(curves in Lagrangian Grassmannians \cite{agrgam1, agrachev, princjac, jac1,jac2, doubzel1}, curves of flags of isotropic and coisotropic subspaces in a %linear symplectic space {\cite{doubzel2, doubzel3}  etc)

%. In contrast to the classical approach of nilpotent differential geometry of Tanaka

In the earlier articles~\cite{doubzel1,doubzel2}  and the preprint \cite{doubzel3} we used this approach to construct the canonical
frames (the canonical structures of absolute parallelism)  for distributions of rank 2 and rank 3 on manifolds of arbitrarily large dimension under very mild genericity assumptions. Later on, in order to interpret these constructions in more conceptual way and generalize them to distributions of higher rank, we developed several new prolongation techniques such as Tanaka like theory for curves of flags \cite{flag1, flag2} and
%frame bundles of so-called quasi-principal frame bundles
for flag structures on manifolds
\cite{quasi}.
In the course of our work we realized that we can distinguish a new basic invariant of distributions, the \emph{Jacobi symbol} \cite{quasi}, which is discrete and classifiable in contrast to the Tanaka symbol.
%It
%allows one to overcome all difficulties of the Tanaka theory mentioned in section 2.3:
%gives an explicit unified construction of the canonical frames for huge classes of distributions of arbitrary rank on manifolds of arbitrary dimensions
In the present paper we combine all this together to describe the uniform (i.e. without any additional branching) construction of canonical frames for distribution with given Jacobi symbol. Our approach significantly extends the scope of the distributions for which the canonical frames can be constructed in explicit algebraic terms,  avoiding the problems of the classification of Tanaka symbols (graded nilpotent Lie algebras) and the possible presence of moduli in the set of Tanaka symbols.

\subsection{Equivalence problem for vector distributions: statement of the problem.}
A rank $l$ vector distribution on an $n$-dimensional  manifold $M$ (or, shortly, an $(l,n)$-distribution) is a rank $l$ subbundle of the tangent bundle $TM$.
The group of
germs of diffeomorphisms of $M$ acts naturally on the set of germs of
$(l,n)$-distributions and defines the equivalence relation there.

Vector distributions appear naturally in
Geometric Control Theory as the sets of admissible velocities for control systems linear with respect to
control parameters  and Geometric Theory of Differential Equations as
natural distributions on submanifolds of jet spaces.

A simple estimation shows that at least $l(n-l)-n$ functions of $n$ variables
are required to describe generic germs of $(l,n)$-distribution, up to the
equivalence.
% (see \cite{versh1} and \cite {zhit0} for precise statements).
 There are only three cases, where $l(n-l)-n$ is not positive: $l=1$ (line distributions),
$l=n-1$, and $(l,n)=(2,4)$. Moreover, it is well known that in these cases
generic germs of distributions are equivalent. For $l=1$ it is just a
particular case of Frobenius theorem. For $l=n-1$ all generic germs are
equivalent to Darboux's model, while for $(l,n)=(2,4)$ they are equivalent to
so-called Engel's model (see, for example, \cite{bryantbook}). In all other
cases generic $(l,n)$-distributions have functional, and, thus, non-trivial
differential invariants. Often it is rather difficult to compute such invariants and to
interpret them geometrically.

%\subsection{Brief history of equivalence problem for distributions}
The general way to solve equivalence problems was developed by E. Cartan by mean of his method of ``moving frames''. The main goal of this method is
to assign to a geometric structure on $M$ the (co)frame  or the structure of absolute parallelism on some (fiber) bundle $P(M)$ over $M$
%over the ambient manifold
in a canonical way.
First, it gives the sharp upper bounds for the group of local symmetries of geometric structures and the explicit description of the group of local symmetries for the maximally symmetric models. Second, it reduces the construction of full system of invariants for the original geometric structure
to the construction of the full system of invariants for the absolute parallelism and the latter can be done easily \cite{stern}.

For example,
in his classical paper~\cite{cart10} Cartan
associates a canonical coframe to any maximally nonholonomic $(2,5)$-distribution.
%to a system of partial differential equations
%of second order and constructs a canonical coframe for non-degenerate
%distributions of this type.
%Note that
This is the case of distributions of smallest rank in the manifold of smallest dimensions where generic distributions have functional invariants.
%This was the first example of an explicit solution for
%the equivalence problem of vector distributions with non-trivial functional
%invariants.
%Remarkably, the most symmetric $(2,5)$-distributions form one
%equivalence class and have an exceptional Lie algebra $G_2$ as their symmetry
%algebra.
Cartan's method of equivalence was successfully applied to other classes
of low-dimensional distributions, for example, for $(3,6)$-distributions by R. Bryant
 (\cite{Bryant,BryantPhd}). However, in the original Cartan prolongation procedure one cannot predict the number
 of prolongation steps and the dimension of the bundle, where the canonical coframe lives, without making the concrete normalizations on each step. The latter often is accompanied with extremely tedious calculations.

 Cartan's prolongation procedure has been reformulated by S. Sternberg in the framework of $G$-structures. However, geometric structures involving nonholonomic distributions do not properly fit into this framework, because the information about the nontrivial commutators of vector fields tangent to  these distributions is not encoded in this framework.

\subsection{Tanaka approach.}
\label{Tanakasub}
N.~Tanaka developed a deep generalization of the theory of $G$-structures which is perfectly adopted to nonholonomic structures, i.e. to nonholonomic distributions themselves or to nonholonomic distribution together with additional structures on them \cite{tan1, mori,aleks, zeltan}. Here we shortly describe the main notions and results of the Tanaka theory in context of distributions.

%\subsubsection{Tanaka symbol}
%yamag,yatsui.
First, passing from the natural filtered structure generated by a distribution to the corresponding graded objects, Tanaka assigned to the distribution at a point a special graded nilpotent Lie algebra, called the \emph{symbol of a distribution $D$ at this point}, which contains the information about the principal parts of all commutators of vector fields tangent to $D$.

More precisely, taking Lie brackets of vector fields tangent to a distribution $D$ (i.e. sections of $D$)   one can define a filtration
\begin{equation}
\label{wdf}
D^{-1}\subset D^{-2}\subset\ldots
\end{equation}
of the tangent bundle, called a \emph{weak derived flag}  or a \emph{small flag (of $D$)}. More precisely, set $D=D^{-1}$ and define recursively
%The obvious (but very rough in the most cases) discrete invariant of a
%distribution $D$ at $q$ is so-called \emph{ the small growth vectors} at $q$.
%It is the tuple
%%$\bigl(\dim D(q),\dim D^2(q),\dim D^3(q),\ldots\bigr)$,
%$\{\dim D^j(q)\}_{j\in{\mathbb N}}$, where $D^j$ is the $j$-th power
%of the distribution $D$, i.e.,
$D^{-j}=D^{-j+1}+[D,D^{-j+1}]$, $j>1$.
Let  $X_1,\ldots X_l$ be $l$ vector fields constituting a \emph{local basis} of a distribution $D$, i.e. $D= \Span \{X_1, \ldots, X_l\}$
in some open set in $M$. Then $D^{-j}(x)$ is the linear span of all iterated Lie brackets of these vector fields, of length not greater than  $j$,  evaluated at a point $x$.
A distribution $D$ is called \emph{bracket-generating} (or \emph{completely nonholonomic}) if for any $x$ there exists $\mu\in\mathbb N$ such that $D^{-\mu}(x)=T_x M$. The number $\mu$ is called the \emph{degree of nonholonomy} of $D$ at a point $x$.
%A distribution $D$ is called \emph{regular} if for all $j<0$, the dimensions of subspaces $D^j(x)$ are independent of the point $x$.
%From now on we assume that $D$ is regular bracket-generating distribution with degree of nonholonomy $\mu$.
Let $\mg^{-1}(x)\stackrel{\text{def}}{=}D^{-1}(x)$ and $\mg^{j}(x)\stackrel{\text{def}}{=}D^{j}(x)/D^{j+1}(x)$ for $j<-1$ . Consider the graded space
\begin{equation}
\label{symbdef}
\mathfrak{t}(x)=\bigoplus_{j=-\mu}^{-1}\mg^j(x),
\end{equation}
corresponding to the filtration
\begin{equation*}
D(x)=D^{-1}(x)\subset D^{-2}(x)\subset\ldots\subset D^{-\mu+1}(x)
\subset D^{-\mu}(x)=T_xM.
\end{equation*}
 This space is endowed naturally with the structure of a graded nilpotent Lie algebra, generated by
$\mg^{-1}(x)$. Indeed, let $\mathfrak p_j:D^j(x)\mapsto \mg^j(x)$ be the canonical projection to a factor space. Take $Y_1\in\mg^i(x)$ and $Y_2\in \mg^j(x)$. To define the Lie bracket $[Y_1,Y_2]$ take a local section $\widetilde Y_1$ of the distribution $D^i$ and
a local section $\widetilde Y_2$ of the distribution $D^j$ such that $\mathfrak p_i\bigl(\widetilde Y_1(x)\bigr) =Y_1$
and $\mathfrak p_j\bigl(\widetilde Y_2(x)\bigr)=Y_2$. It is clear that $[Y_1,Y_2]\in\mg^{i+j}(x)$. Put
\begin{equation}
\label{Liebrackets}
[Y_1,Y_2]\stackrel{\text{def}}{=}\mathfrak p_{i+j}\bigl([\widetilde Y_1,\widetilde Y_2](x)\bigr).
\end{equation}
It is easy to see that the right-hand side  of \eqref{Liebrackets} does not depend on the choice of sections $\widetilde Y_1$ and
$\widetilde Y_2$. Besides, $\mg^{-1}(x)$ generates the whole algebra $\mathfrak{t}(x)$.
A graded Lie algebra satisfying the last property is called \emph{fundamental}.
The graded nilpotent Lie algebra $\mathfrak{m}(x)$ is called the \emph {(Tanaka) symbol of the distribution $D$ at the point $x$}.

Given a fundamental graded nilpotent Lie algebra $\mathfrak t=\displaystyle{\bigoplus_{j\le -1}\mathfrak g^j}$ Tanaka considered \emph{distributions of constant type $\mathfrak t$}, i.e. distributions $D$ with symbols at any point isomorphic to %given graded nilpotent Lie algebra
$\mathfrak t$.  Note that in general the assumption of constancy of Tanaka symbol is quite restrictive, see discussions in subsection \label{tanakalim} below.
%To any nilpotent graded Lie algebra $\mathfrak m=\displaystyle{\bigoplus_{j\le -1}
Among all such distributions he distinguish the most simple distribution  called the \emph{flat (or standard) distribution} of type $\mathfrak t$, namely,
%The flat distribution $D_\mathfrak t$ of type $\mathfrak t$ is
the left-invariant distribution on the connected simply-connected Lie group with Lie algebra $\mathfrak t$  such that $D_{\mathfrak t}(e)=\mg^{-1}$, where $e$ is the identty of the group.
%\mathfrak g_j(q)}$ one can assign the \emph{standard} or \emph{flat} distribution $D_{\mathfrak t}$, which is the left invariant distribution on the %corresponding simply-connected Lie group equal to $\mathfrak g_{-1}$ at the identity.

Remarkably, the prolongation procedure for the construction of canonical frames can be described uniformly, i.e. without any additional branching, for all distribution of constant Tanaka type $\mathfrak t$. This construction
%of the canonical (co)frame for all distributions of constant type $\mathfrak m$
can be imitated by the construction of such canonical frame for the flat distribution of type $\mathfrak t$.
%, which is the most simple homogeneous model among all structures of this type.
The latter construction can be described purely algebraically in terms of the so-called universal algebraic prolongation of the symbol $\mathfrak t$,
which is in essence the maximal (nondegenerate) graded Lie algebra, containing the graded algebra $\mathfrak t$ as its negative part.

More precisely, the \emph{universal algebraic prolongation of the symbol $\mathfrak t$} is the graded Lie algebra $\mathfrak u(\mathfrak t)=\displaystyle{\bigoplus_{i\in\mathbb{Z}}}\mg^i(\mathfrak t)$ determined  by the following three conditions:
\begin{enumerate}
  \item the graded subalgebra $\displaystyle{\bigoplus_{i<0}}\mg^i(\mathfrak t)$ of $\mathfrak u(\mathfrak t)$, i.e. the negative part of $\mathfrak u(\mathfrak t)$, coincides with $\mathfrak t$;

\item (\emph{non-degenericity assumption}) for any  nozero  $x\in {\mathfrak g}^i(\mathfrak t)$ with  $i\geq 0$
        %\tb{$\mathrm{ad}\, x|_{\mathfrak m}\neq 0$}
      there exists $y\in \mathfrak t$ such that $[x,y]\neq 0$ (i.e. $\ad x|_{\mathfrak m}\neq 0$);

\item $\mathfrak{u}(\mathfrak{m})$ is the maximal graded algebra satisfying conditions (1) and (2) above.
\end{enumerate}

Note  that if
$\dim \mathfrak u(\mathfrak t)<\infty$, then  $\mathfrak u(\mathfrak t)$ is isomorphic
to the algebra of infinitesimal symmetries of the flat distribution of type $\mathfrak t$; if  $\dim \mathfrak u(\mathfrak t)=\infty$, then the completion of $\mathfrak u(\mathfrak t)$ is isomorphic to the algebra of formal power series of infinitesimal symmetries of the flat distribution of $\mathfrak t$ (\cite{tan1} section 6, \cite{yamag} section 2).

The universal algebraic prolongation can be explicitly realized by constructing $\mg^0(\mathfrak t)$,  $\mg^1(\mathfrak t)$ etc inductively.
Indeed, $\mg^0(\mathfrak t)$ is realized as the Lie algebra of all derivations $a$ of $\mathfrak t$, preserving the grading (i.e. such that  $a \mg^i\subset \mg^i$ for all $i<0$). Further, the space $\mathfrak t\oplus \mg^0(\mathfrak t)$ is equipped with the structure of Lie algebra in the natural way and we can define $\mg^1(\mathfrak t)$ as follows:

$$\mg^{1}(\mathfrak t)=\left\{f\in\displaystyle{\bigoplus_{i<0}}\mathrm{Hom}(\mg^i(\mathfrak t),\mg^{i+1}(\mathfrak t)):f([v_1,v_2])=[f(v_1),v_2]+
   [v_1,f(v_2)],
  \forall v_1,v_2\in \mathfrak t \right\}$$

Now assume that spaces $\mg^l(\mathfrak t)\subset \displaystyle{\bigoplus_{i<0}{\rm Hom}(\mg^i(\mathfrak t),\mg^{i+l}(\mathfrak t))}$ are defined  for
all $0<l<k$. Set

\begin{equation}
\label{br2}
[f,v]=-[v, f]=f(v) \quad \forall f\in \mg^l (\mathfrak t), 0\leq l<k, \text{ and }v\in\mathfrak t.
\end{equation}
 Then let
\begin{equation}
\label{mgk}
\mg^k(\mathfrak t)\stackrel{\text{def}}{=}\left\{f\in \bigoplus_{i<0}{\rm Hom}(\mg^i(\mathfrak t),\mg^{i+k}(\mathfrak t)): f ([v_1,v_2])=[f (v_1),v_2]+[v_1, f( v_2)]\,\,\forall\, v_1, v_2 \in \mathfrak t\right\}.
\end{equation}
%Directly from this definition and the fact that $\mathfrak t$ is fundamental (that is, it is generated by $\mg^{-1}$) it follows that if
%$f\in \mg^k$ satisfies $f|_{\mg^{-1}}=0$, then $f=0$.
The space $\bigoplus_{i\in Z} \mg^i(\mathfrak t)$ can be naturally endowed with the structure of a graded Lie algebra.
The brackets of two elements from $\mathfrak t$ are as in $\mathfrak t$. The brackets of an element with non-negative weight and an element from $\mathfrak t$ are already defined by  \eqref{br2}.
It only remains to define  the brackets $[f_1,f_2]$ for $f_1\in\mg^k(\mathfrak t)$, $f_2\in \mg^l(\mathfrak t)$ with $k,l\geq 0$.
%, because brackets of all other pair of elements are already defined.
The definition is inductive with respect to $k$ and $l$: if $k=l=0$ then the bracket $[f_1,f_2]$ is as in $\mg^0(\mathfrak t)$. Assume that $[f_1,f_2]$
is defined for all $f_1\in\mg^k(\mathfrak t)$, $f_2\in \mg^l(\mathfrak t)$ such that a pair  $(k,l)$ belongs to  the set
\begin{equation*}
\{(k,l):0 \leq k\leq \bar k, 0\leq l\leq \bar l\}\backslash \{(\bar k,\bar l)\}.
\end{equation*}
Then define $[f_1, f_2]$  for $f_1\in\mg^{\bar k}(\mathfrak t)$, $f_2\in \mg^{\bar l}(\mathfrak t)$ to be the element of $\displaystyle{\bigoplus_{i<0}{\rm Hom}(\mg^i(\mathfrak t),\mg^{i+\bar k+\bar l})(\mathfrak t)}$ given by
\begin{equation}
\label{posbr}
[f_1,f_2]v\stackrel{\text{def}}{=} [f_1(v), f_2]+[f_1,f_2(v)]\quad \forall v\in\mathfrak t.
\end{equation}
It is easy to see that $[f_1,f_2]\in \mg^{k+l}(\mathfrak t)$ and that $\bigoplus_{i\in Z} \mg^i$ with bracket product defined as above is a graded Lie algebra.
%As a matter of fact (\cite{tan} \S 5) this graded Lie algebra satisfies Properties 1-3 from Subsection 1.4. That is it is a realization of the algebraic universal prolongation $\mg(\mathfrak m, \mg^0)$ of the algebra $\mathfrak m\oplus\mg^0$.

One of the main results of Tanaka theory \cite{tan1} applied to the geometry of distributions can be formulated as follows:

\begin{theorem}
\label{tanthm}
%(\textsf{Tanaka, 1970})}
Assume that $\dim \mathfrak u(\mathfrak t)<\infty$.
%and \tb{$ k\geq 0$}
%is the maximal integer such that the \tb{$k$}th algebraic prolongation \tb{$\mathfrak g^k(\mathfrak m)$}
%does not vanish.
%\medskip
%\emph{
%\begin{enumerate}
%\item
Then to a distribution $D$ with constant symbol $\mathfrak t$ one can assign in a canonical way a bundle over $M$ of dimension equal to $\dim \mathfrak u(\mathfrak t)$ equipped with a canonical
frame.
%\pause
%\item
In particular the dimension of algebra of infinitesimal symmetries of $D$ is not greater than $\dim \mathfrak u(\mathfrak t)$.
%\pause
%\item
Moreover, this upper bound is sharp and is achieved if and only if a distribution is locally equivalent to the flat distribution $D_\mathfrak t$.
%\end{enumerate}}
%The analogous formulation in the case  when
%$\mathfrak u(\mathfrak m, \mg^0)$ is infinite dimensional may be found in \cite{tan} \S 6.
\end{theorem}

Thus, the Tanaka approach allows one to predict the number of prolongations steps and the dimension of the bundle, where the canonical frame lives,
%,}
without making concrete normalization on each step as the original Cartan method of equivalence suggests. For this it is enough to calculate the universal algebraic prolongation of the symbol $\mathfrak t$, or, equivalently, all algebraic prolongations $\mg^i(\mathfrak t)$ up to the first vanishing one (if it exists). According to \eqref{mgk} the latter consists of solving systems of linear algebraic equations.

\begin{remark}
\label{Tanakastr}
{\rm
More generally, given any subalgebra $\mg^0$ of the algebra $\mg^0(\mathfrak t)$ one can define the \emph{universal algebraic prolongation
\begin{equation}
\label{iprolong}
\mathfrak u(\mathfrak t, \mg^0)=\bigoplus_{i\in \mathbb Z}\mathfrak u^i(\mathfrak t, \mg^0)
\end{equation}
of the pair $(\mathfrak t,\mg^0)$} as the maximal (nondegenerate) graded Lie algebra, containing the graded algebra $\mathfrak t\oplus\mg^0$ as its non-positive part. This algebra can be constructed by the inductive procedure described by formulas analogous to \eqref{mgk}: for this set $u^i(\mathfrak t, \mg^0)=\mg^i$ for $i\leq 0$ and recursively
\begin{equation}
\label{prolongk}
u^i(\mathfrak t, \mg^0)\stackrel{\text{def}}{=}\left\{f\in \bigoplus_{j<0}{\rm Hom}\bigl(u^j(\mathfrak t, \mg^0),u^{j+i}(\mathfrak t, \mg^0)\bigr): f ([v_1,v_2])=[f (v_1),v_2]+[v_1, f( v_2)]\,\,\forall\, v_1, v_2 \in \mathfrak t\right\}
\end{equation}
for $i>0$. The space $u^i(\mathfrak t, \mg^0)$ is called the \emph{$i$th Tanaka algebraic prolongation of the pair $(\mathfrak t, \mg^0)$}.

The pairs $(\mathfrak t, \mg^0)$ appear as basic characteristics of distributions with constant Tanaka symbol $\mathfrak t$ and with certain additional structures on them. Such structures are called \emph {Tanaka structures of type $(\mathfrak t, \mg^0)$}.  To define these structures first
to a distribution of type $\mathfrak{t}$ one assigns the principal bundle $P^0(\mathfrak t)$ of partial frames with the structure group $\rm{Aut}(\mathfrak{t})$ of automorphisms
of the graded Lie algebra $\mathfrak{t}$:
 %that is, the group of all automorphisms $A$ of the linear space $\mathfrak{t}$ preserving both the Lie brackets ($A([v,w])=[A(v),A(w)]$ for any $v,w\in \mathfrak{t}$) and the grading ($A (\mg^i)=\mg^i$ for any $i<0$).
%Let
%$\text{Iso}\bigl(\mathfrak{m},\mathfrak m(x)\bigr)$ be the set of all graded Lie algebra isomorphisms $\vf:
%\mathfrak m\mapsto \mathfrak m(x)$ and
%\begin{equation}
%\label{P0}
$P^0(\mathfrak t)$ is the set of all pairs $(x,\vf)$, where  $x\in M$ and  $\vf:\mathfrak{t}\to\mathfrak t(x)$
is an isomorphism of the graded Lie algebras $\mathfrak {t}$ and $\mathfrak t(x)$ .
Since $\mg^{-1}$ generates
 %the whole
 $\mathfrak t$, the group $\rm{Aut}(\mathfrak{t})$ can be identified with a subgroup of $\text{GL}(\mg^{-1})$. By the same reason a point $(x,\vf)\in P^0(\mathfrak t)$ of a fiber of $P^0(\mathfrak t)$ is uniquely defined by $\vf|_{\mg^{-1}}$. So one can identify
$P^0(\mathfrak t)$ with the set of pairs $(x,\psi)$, where  $x\in M $ and $\psi:\mg^{-1}\to D(x)$ i.e. it is a partial frame that can be extended to an automorphism of the graded Lie algebras $\mathfrak {t}$ and $\mathfrak t(x)$. Let $G^0\subset {\rm Aut}(\mathfrak t)$ be a Lie group
with the Lie algebra $\mg^0$.
A \emph{Tanaka structure of constant type $(\mathfrak t, \mg^0)$} is any $G^0$-reduction of the bundle  $P^0(\mathfrak t)$.
In \cite{tan1} Theorem \ref{tanthm} (except the last statement about the uniqueness of the maximally symmetric model) was proved not only for distributions but for any Tanaka structure of type $(\mathfrak t, \mg^0)$. The last statement of Theorem \ref{tanthm} is not true for general Tanaka structures (for example it is not true for the classcal Riemannian metrics which is a particular case of them) but it is true for distributions (i.e. when $\mg^0=\mg^0(\mathfrak t)$) due to the presence of the grading element in $\mg^0(\mathfrak t)$.
%$\Box$

% }\text{Iso}\bigl(\mathfrak{m},\mathfrak m(x)\bigr)\}$$
%Speaking informally, $P^0(\mathfrak t)$
%can be seen as a $G^0(\mathfrak{t})-$reduction of the bundle of all frames of the distribution $D$.

%and  a theorem analogous to }

%As a matter of fact in \cite{tan1} Tanaka considered more general structures than just distributions, namely distributions with additional structures on them. In more details, first
%to a distribution of type $\mathfrak{t}$ one can assign a principal bundle $P^0(\mathfrak t)$ of partial frames with the structure group $\rm{Aut}(\mathfrak{t})$ of automorphisms
%of the graded Lie algebra $\mathfrak{t}$:
% %that is, the group of all automorphisms $A$ of the linear space $\mathfrak{t}$ preserving both the Lie brackets ($A([v,w])=[A(v),A(w)]$ for any $v,w\in \mathfrak{t}$) and the grading ($A (\mg^i)=\mg^i$ for any $i<0$).
%%Let
%%$\text{Iso}\bigl(\mathfrak{m},\mathfrak m(x)\bigr)$ be the set of all graded Lie algebra isomorphisms $\vf:
%%\mathfrak m\mapsto \mathfrak m(x)$ and
%%\begin{equation}
%%\label{P0}
%$P^0(\mathfrak t)$ is the set of all pairs $(x,\vf)$, where  $x\in M$ and  $\vf:\mathfrak{t}\to\mathfrak t(x)$
%is an isomorphism of the graded Lie algebras $\mathfrak {t}$ and $\mathfrak t(x)$ .
}
\end{remark}

\begin{remark}
\label{parabrem}
{\rm Further, Tanaka showed \cite{tan2, mori, cap1, yamag} that if the algebra $\mathfrak u(\mathfrak t, \mathfrak g^0)$ is semisimple, then the normalization conditions in the prolongation procedure can be chosen such that the resulting coframe is the Cartan connection (a coframe with special nice properties) modeled by a parabolic geometry. The E. Cartan and R. Bryant constructions for $(2,5)$ and $(3,6)$ distribution are the particular cases of this theory, which includes also conformal and CR structures %as special cases}
\cite{tanCR,cap2}.}
%$\Box$
\end{remark}

\subsection {Limits and weak points of Tanaka approach}
\label{tanakalim}
All constructions of the Tanaka theory depend on the symbol. Therefore
in order to apply the Tanaka machinery
to all bracket-generating distributions of the given rank $l$ on a manifold of the given dimension $n$, one has to classify all $n$-dimensional graded nilpotent Lie algebra with $l$ generators. Note also that  the set of all possible Tanaka symbols may contain moduli, i.e. may depend on continuous parameters.

The problem
of classification of all Tanaka symbols is quite
nontrivial already in dimension $7$ (see \cite{kuz}) and it looks completely
hopeless for arbitrary dimensions. For example, as was shown in \cite{kuz}
already in dimension $7$ the continuous parameters appear in symbols of rank
$(3,6,\ldots )$-distributions
(see models $m7\_3\_3(\alpha)$ and
$m7\_3\_13(\alpha)$ there), and there are $6$ more non-isomorphic symbols in
addition to that. In the case of $(2,n)$-distributions with $\dim D^{-3}=5$ there are $3$ non-isomorphic symbols Tanaka symbols if $n=6$, then  $8$ non-isomorphic symbols if $n=7$, and continuous parameters appear in the set of all symbols for $n\geq 8$.

%one has to generalize the Tanaka prolongation procedure to distributions with non-constant symbol, because
%it is hopeless in general to classify all possible graded nilpotent Lie algebras and the set of all such algebras contains moduli (continuous parameters).
Hence,  first,  generic distributions may have non-isomorphic Tanaka symbols at different points so that in this case the Tanaka theory cannot be directly applied, although a generalization of the Tanaka prolongation procedure to distributions with non-constant Tanaka symbols is certainly possible ( a slightly different aspect concerning the upper bound for the infinitesimal symmetries algebra of distribution with non-constant Tanaka symbols was addressed in \cite{krug1, krug2}).
Second, even restricting to distributions with a constant Tanaka symbol, without the classification of these symbols we do not have a complete picture about the geometry of distributions.
%For example,
%as shown in \cite{kuz},
%already in dimension $7$ the continuous parameters appear in symbols of rank
%3 distributions with $6$ dimensional square
%%s.g.v. $(3,6,\ldots)$
%%(see models $m7\_3\_3(\alpha)$ and
%%$m7\_3\_13(\alpha)$ there),
%and there are 6 more non-isomorphic symbols in
%addition to that. For rank 2 distributions with five dimensional cube
%%s.g.v. $(2,3,5,\ldots)$
%%in dimension 6
%%three different symbol algebras are possible and
%continuous parameters for symbols appear starting from dimension 8.
In particular, even in the case of rank 2 and rank 3 distributions the Tanaka theory gives a complete picture only in few cases of low dimensions.

\subsection {Geometry of distribution via symplectification.}
In \cite{agrachev, agrgam1} Andrei Agrachev put forward the idea that invariants of a submanifold of the tangent bundle of a manifold $M$ (with respect to the natural action of the group of diffeomorphisms of $M$) can be obtained by studying the flow of extremals of variational problems naturally associated with this submanifold. In the case when this submanifold is a bracket-generating distribution, the Agrachev idea can be briefly described as follows (see section \ref{abnsec} below for more detail).

First recall that a Lipschitzian curve $\alpha(t)$ is called \emph{admissible (horizontal)} curve of a distribution $D$, if $\alpha'(t)\in \mathcal D\cap T_{\alpha(t)}M$ for almost every $t$. Fixing any two points on $M$ consider any functional of integral type (for example, the length functional with respect to some Riemannian metric) on the space of admissible curves
connecting these points. Note that this space is not empty by the Rashevsky-Chow theorem.
The Pontryagin extremals of a variational problem are special curves in the cotangent bundle $T^*M$ such that their projections to the base manifold $M$ are candidates for an extremum. They are described by the Pontryagin Maximum Principal of Optimal Control \cite{pontr}. Among all this extremals there can be extremals with zero Lagrangian multiplier near the functional. Such extremals are called \emph{abnormal}.

Abnormal extremals do not depend on the functional but on the distribution $D$ only and can be described purely geometrically in terms of $D$ without referring to any auxiliary functional on the space of admissible curves of $D$ and using just the canonical symplectic structure on the cotangent bundle $T^*M$.
%=\{(p,q): q\in M, p\in T^*_qM$.
%(see section \ref{abnsec} below).

For this consider the dual object to the distribution $D$ in $T^*M$, which is its annihilator $D^\perp$, i.e. the submanifold of $T^*M$ consisting of all $(p,q)\in T^*M$, where  $q\in M$ and $p\in T^*_qM$, such that the covector $p$ annihilates the space $D(q)$. Also let  $(TM)^\perp$ be the zero section of $T^*M$.

An abnormal extremal is a Lipshitzian curve in $D^\perp\backslash (TM)^\perp$ such that the tangent line to it at
almost every point belongs to the kernel of the restriction
%$\widehat\sigma|_{D^\perp}$ of $\widehat\sigma$
of the canonical symplectic form $\widehat\sigma$ of $T^*M$
to $D^\perp$ at this point. Hence, abnormal extremals lie in the \emph {degeneracy locus} $\widetilde W_D$ of the form $\widehat\sigma|_{(D^\perp}$, i.e. in the set of all points, where the form $\widehat\sigma|_{D^\perp}$ is degenerate.

It turns out (see section \ref{abnsec}) that if $\rank \, D$ is odd then $\widetilde W_D=D^\perp$ and if $\rank\, D$ is even then for any point $q\in M$ the intersection $\widetilde W_D(q)$ of $W_D$ with the fiber $D^\perp(q)$ over $q$ is the zero level set of a homogeneous polynomial of degree $\cfrac {1}{2}\, \rank\, D$, which is the Pfaffian of the so-called Goh matrix. In particular, in the case of rank $2$ distributions $\widetilde W_D(q)$ is a vector subspace of $D^\perp(q)$. Moreover, in this case $\widetilde W_D$ is equal
to the annihilator of $D^{-2}$.

As a consequence of the previous paragraph, for all odd rank distributions $D$ the set $\widetilde W_D$ is an odd dimensional submanifold of $T^*M$  and for even rank distributions satisfying very mild genericity assumptions an open and dense subset  of $\widetilde W_D$ is a smooth submanifold of odd dimension. In particular, in the case of $(2, n)$-distributions  with $n\geq 3$ bracket genericity assumption can be taken as such genericity assumption.
%the latter is the case for bracket generating $(2,n)$-distributions for $n\geq 3$.
%It turns out that for generic distributions the set of smooth points of $\widetilde W_D$ is a submanifolds  of $D^\perp$ of odd dimension and
Thus in all these cases the form $\widehat\sigma|_{\widetilde W_D}$  is degenerate at the point of smoothness of $\widetilde W_D$.

Further, consider the ``regular part'' $W_D$ of the degeneracy locus $\widetilde W_D$ consisting of all points of smoothness $\lambda$ in $\widetilde W_D$ such that the kernel of the form  $\widehat\sigma(\lambda)|_{\widetilde W_D}$ is one dimensional.
%(see section \ref{abnsec} for explicit description of the sets  $\widetilde W_D$ and $W_D$).
These kernels define the characteristic rank $1$ distributions $\widehat C$ on $W_D$ and any integral curve of this characteristic distribution  is an abnormal extremal. Such abnormal extremals are called \emph{regular}.

The set $W_D$ and the characteristic rank $1$ distribution on it can be described explicitly in terms of the weak derived flag \eqref{wdf} and a local basis of $D$ and the set $W_D$ is non-empty (and therefore open in $\widetilde W_D$) for generic germs of distributions (see section \ref{abnsec} for details). In particular,
$W_D=(D^{-2})^\perp\backslash (D^{-3})^\perp$ if $\rank\, D=2$ and $W_D=(D)^\perp\backslash (D^{-2})^\perp$ if $\rank\, D=3$ and therefore $W_D$ is non-empty for bracket generating distributions of rank $2$ and $3$ in an ambient manifold $M$ of dimension greater or equal than $4$.
%Also, the characteristic rank $1$ distribution $\widehat C$ on $W_D$ depends   on brackets of length not greater than $2$ of a local basis in the case of distributions of odd rank and on brackets of length not greater than $3$ of a local basis in the case of distributions of odd rank.

After the projectivization of each fiber of the cotangent bundle $T^*M$ regular abnormal extremals foliate
the even-dimensional submanifold %$W_D$
$\mathbb P W_D$
of the  projectivized cotangent bundle $\mathbb P T^*M$. %(note that in Example 3 of section 3 we use a different notation for $\mathcal H_D$).
%The Pontryagin Maximum Principal of Optimal Control gives a very efficient way to describe critical o
%Maximum Principal of Optimal Control gives a very efficient way to describe extremals  of such variational problem.
%There is two types of Pontryagin extremals of an optimal control problem, \emph{normal extremals} and \emph{abnormal extremals}. The Lagrangian multiplier near the functional of the optimal problem is not equal to zero in the first case and equal to zero in the second one.
%
%variational problem on a space of admissible curves of this distribution with fixed endpoints.
% Then one can associate with $\mathcal U$  the following family of the so-called time-minimal problems: given two points $q_0$ and $q_1$ in $M$ to steer from $q_0$ to $q_1$ in a minimal time moving along admissible curves of $\mathcal U$.
%For sub-Riemannian (sub-Finslerian) structures these time-minimal problems are exactly the length minimizing problems.
%In the most cases we can distinguish intrinsically certain submanifold $\mathcal H$ of  $T^*M$  foliated by either normal or abnormal extremals.
%The submanifold $\mathcal H$ inherits from $T^*M$ the structure of a fiber bundle (over $M$).
The space of regular abnormal extremals can be considered as the quotient manifold $\mathcal N$ by the foliation of these extremals.
The natural contact distribution $\Delta$ is induced on $\mathcal N$ by the tautological $1$-form (the Liouville form)  in $T^*M$. The natural conformal symplectic structure , i.e. a symplectic form $\sigma_\gamma$
%, defined up to a nonzero scalar mupltiple,
is defined, up to a nonzero scalar multiple, on each space $\Delta(\gamma)$, $\gamma\in \mathcal N$.

Further, the distributions $D$ can be lifted to $W_D$ by the pullback with respect to the canonical projection  $W_D\rightarrow M$.
Denote the resulting distribution by $\mathcal J$.
Given an extremal $\gamma\in \mathcal N$, for any point $\lambda\in \gamma$ consider the images $\mathfrak J_\gamma(\lambda)$ of the space $\mathcal J(\lambda)$ under the differentials of the canonical map from $W_D$ to the quotient manifold $\mathcal N$. These images define a curve
$\lambda\mapsto \mathfrak J_{\gamma}(\lambda)$, $\lambda\in\gamma$ of subspaces in the space $\Delta(\gamma)$, called the \emph{Jacobi curve} of $\gamma$.

The motivation for this name comes from the considering the case of rank $2$ distributions. In this case all subspaces $\mathfrak J_\gamma(\lambda)$ are Lagrangian subspaces of $\Delta(\gamma)$ with respect to the  symplectic form $\sigma_\gamma$, i.e. the curve
$\mathfrak J_\gamma$ is a curve in a Lagrangian Grassmannian.
Actually this curve describes the space of solutions of the Jacobi equations along the extremal $\gamma$ and
 the index of the second variation along the extremal $\gamma$ can be expressed in terms of certain intersection index (a Maslov type index) of it
 %$\mathcal J_\gamma$
 \cite{agrachev}. In the case of distributions of the rank greater than $2$ all subspaces $\mathfrak J_\gamma (\lambda)$ are strictly coisotropic, i.e. they strictly contain the skew-orthogonal complements of themselves with respect to the form $\sigma_\gamma$. Although in this case the index of the second variation along an abnormal extremal is generically infinite, so that there is no analogy with the Calculus of Variation as in the case of rank 2 distributions,
 we still keep to call the curves $\mathfrak J_\gamma$ by the name Jacobi curve and we can exploit them for the local geometry of distributions. Namely, \emph{ one can
construct canonical frames and invariants of the original distribution $D$ by studying
 the differential geometry of the Jacobi curves $\mathfrak J_\gamma$ with respect to the action the group $\rm{CSp}(\Delta(\gamma)$ of the conformal symplectic transformation of $\Delta(\gamma)$  (i.e. the group of the linear transformation preserving the symplectic form $\sigma_\gamma$, up to a scalar multiple) for generic abnormal extremals $\gamma$.}

\subsection{Jacobi symbol of distributions}
 It turns out \cite{flag1,flag2} that the differential geometry of Jacobi curves (and even more general curves in flag varieties) can be treated similarly to the Tanaka theory of filtered structures.  The central notion in this treatment is the flag symbol of the Jacobi curve at a given point, which in turn gives a new basic characteristic of distributions, called the Jacobi symbol. Here we give a brief informal description of what is the Jacobi symbols of a distribution, postponing the detailed description to section \ref{jacsymbsec}.

  %The main advantage of this approach that the analog of Tanaka symbol in the geometry of Jacobi curves is much simpler algebraic object than the Tanaka symbol of distributions
 First we build a curve of flags of $\Delta(\gamma)$ out of the Jacobi curve $\mathfrak J_\gamma$  by collecting the osculating spaces of the Jacobi curve of any order together with their skew-orthogonal complements with respect to the form $\sigma_\gamma$. The resulting curve of flags is called the \emph{extended Jacobi curve of  the extremal $\gamma$} and will be denoted by $\widehat{\mathfrak J}_\gamma$. Each flag $\widehat {\mathfrak J}_\gamma(\lambda)$ consists of isotropic/coisotropic subspaces of $\Delta(\gamma)$ and is stable with respect to the operation of taking the skew-orthogonal complement. Such flags will be called \emph{symplectic flags} for shortness.  Note that in this case the graded space $\gr \widehat {\mathfrak J}_\gamma(\lambda)$ corresponding to the filtration $\widehat {\mathfrak J}_\gamma(\lambda)$ is endowed with the conformal symplectic structure naturally induced by the conformal symplectic structure on $\Delta(\gamma)$.

 Moreover the extended Jacobi curve $\widehat {\mathfrak J}_\gamma$ satisfies an important  additional property. Let us describe it first for the case of distributions of odd rank.
 If the subspaces of the flags $\widehat{\mathfrak J}_\gamma(\lambda)$ are indexed by $\mathbb Z$ such that the corresponding filtration is nonincreasing by inclusion, then we have

{\bf Osculation property}
 \emph{For any $i\in \mathbb Z$  and any $\lambda\in \gamma$ the osculating space at $\lambda$ of the curve of subspaces of $\widehat {\mathfrak J}_\gamma$ indexed by $i$ is contained in the subspaces of $\widehat {\mathfrak J}_\gamma(\lambda)$ indexed by $i-1$.}

 The latter implies that the tangent line to the curve $\widehat{\mathfrak J}_\gamma$
 can be identified with a line of degree $-1$ endomorphisms of the graded space $\gr\widehat {\mathfrak J}_\gamma(\lambda)$, belonging to the Lie algebra of the group of conformal symplectic transformations. This line, up to a conjugation by a conformal symplectic transformation preserving the grading, is called the \emph{flag symbol} or the \emph{Jacobi symbol}  of the curve  $\widehat{\mathfrak J}_\gamma$ (or the curve ${\mathfrak J}_\gamma$)
 %of either $T_\gamma \mathcal X$ in the normal case or $\mathcal C(\gamma)$ in the abnormal case. This curve is called \emph{(extended) Jacobi curve of  the extremal $\gamma$}.
 %For this first, by analogy with the Tanaka theory, we can distinguish the basic invariant of a Jacobi curve $\Delta(\gamma)$ at a point by passing, as in the case of the Tanaka symbol of a distribution, from the natural filtered structures on $\Delta(\gamma)$ associated with the Jacobi curve to the corresponding  graded structures.

Extended Jacobi curves of rank $2$ distributions satisfy the osculation property as well.
It is not the case for distributions of even rank greater than $2$ if the subspaces of the flags $\widehat{\mathfrak J}_\gamma(\lambda)$ are indexed by $\mathbb Z$. However, the osculation property becomes valid also in this case, if the subspaces of the flags $\widehat{\mathfrak J}_\gamma(\lambda)$ are indexed by $\cfrac{1}{2}\mathbb Z=\{\frac{i}{2}: i\in \mathbb Z\}$  such that the corresponding filtration is nonincreasing by inclusion. So, the flag symbol as a line of degree $-1$ endomorphisms of the corresponding graded space $\gr\widehat {\mathfrak J}_\gamma(\lambda)$ can be defined for even rank distribution as well.

 To make things uniform it is more convenient  in the case of rank $2$ distributions to index the subspace of the extended Jacobi curve $\widehat {\mathfrak J}_\gamma$ by $\cfrac {1}{2} \mathbb Z_{{\rm odd}}:=\cfrac{1}{2}\mathbb Z\backslash \mathbb Z$ (see subsection \ref{flagsec1} below).
 Then the flag symbols of extended Jacobi curves of distributions of even rank greater than $2$ are actually the direct sum of a flag symbol of curves of flags appearing as Jacobi curves of a distribution of odd rank and one flag symbol of curves of flags appearing as Jacobi curves of a  rank 2 distribution.
 This choice of the set of indices is also justified by the theory of irreducible representations of $\sll_2$, which plays an important role in the sequel (sections \ref{findimsec}-\ref{secantsec}) : the spectrum of a semisimple element of such representations constitute of an arithmetic progression centered at $0$  and in the case when the difference between consecutive terms of this progression is equal to $1$ and the dimension of the corresponding $\sll_2$-module is even this spectrum belongs to $\cfrac {1}{2} \mathbb Z_{{\rm odd}}$.

In contrast to the Tanaka symbols of distributions the flag symbols of Jacobi curves can be easily classified \cite[subsection 7.2]{flag2} and subsection \ref{flagsec2} below. Besides, the set of all possible flag symbols is discrete. Therefore, if for any $\lambda\in W_D$ we denote by $\delta(\lambda)$ the flag symbol at $\lambda$ of the Jacobi curve of the regular abnormal extremal passing through $\lambda$, then the map $\lambda\mapsto\delta(\lambda)$ is constant in a neighborhood  of a generic point of the bundle $W_D$. Moreover, as it will be shown in subsection \ref{flagsec3}
for a generic point $q\in M$ there is a neighborhood $U$ of $q$ in $M$ such that the map $\lambda\mapsto \delta(\lambda)$ is constant on the generic subset $V$  of $\tilde \pi^{-1}(U)\cap \mathbb PW_D$, where $\tilde \pi:\mathbb P T^*M\rightarrow M$  is the canonical projection. In this case $\delta(\lambda)$, $\lambda\in V$ is called the \emph {Jacobi symbol of the distribution $D$ at the point $q$}.
Jacobi symbol is  a new basic characteristic of a distribution $D$.
%, which is coarser than the Tanaka symbol.
%\subsection{Prolongation of distributions with given Jacobi symbol}
\subsection{Main results and organization of the paper}
Although the definition of the Jacobi symbol is more involved, this is a simpler algebraic object %and a coarser characteristic of a distribution
than the Tanaka symbol. The natural task is \emph{to describe the construction of the canonical frames uniformly for all distributions with given Jacobi symbol $\delta$.}

To begin with this task, in subsection \ref{flagsec4}, given any Jacobi symbol $\delta$, we define in a sense the most simple model of a distribution with Jacobi symbol $\delta$. We call it symplectically flat distributions because it plays the same role in the local geometry of distributions with given Jacobi symbol, as the flat distribution of fixed type in the local geometry of distributions with given Tanaka symbol (see the last sentence of Theorem \ref{maintheor} below), i.e. it is the maximally symmetric distribution within the considered class (if one assumes the finiteness of type). In the case of distributions of odd rank or rank $2$ the symplectically flat distributions are also  flat distribution in the Tanaka sense for some fundamental Tanaka symbol.
An interesting phenomenon here is that this is not the case for a symplectically flat distribution of even rank greater than $2$. However, it can be considered as the flat distribution in the Tanaka sense for some nonfundamental Tanaka symbol (see Example \ref{rank4flatex} below).

The previous constructions assign to any distribution with given Jacobi symbol the following structure on the generic subset of the space of regular abnormal extremals $\mathcal N$: the contact distribution $\Delta$ and the curves  $\widehat{\mathfrak J}_\gamma$ of symplectic flags on each fiber $\Delta(\gamma)$. In the case of distributions of odd rank and rank $2$  this is a particular case of the flag structures (with constant flag symbol $\delta$) introduced and studied in our recent preprint \cite{quasi}, while the flag structure in the case of distributions of even rank greater than $2$ can be treated similarly.
The construction of the canonical frames for distributions with given Jacobi symbol $\delta$ is given in section \ref{jacsymbsec}, Theorem \ref{maintheor}.
Instead of trying to adopt a prolongation procedure
for the construction of a canonical frame
directly to an original distribution, %$\mathcal V$,
as in the classical approaches, we adopt it to the corresponding flag structure. This can be done in terms of natural algebraic operations on the Jacobi symbol $\delta$ in the category of graded Lie algebras, using the Tanaka like theory of curves of symplectic flags developed in \cite{flag1, flag2} and the Tanaka like prolongation of flag structures developed in \cite{quasi}. Theorem \ref{maintheor} plays the same role in the local geometry of distributions with given Jacobi symbol as Theorem \ref{tanthm} in the Tanaka theory of distributions.

%While our previous paper \cite{quasi} is devoted to the prolongation procedure of general flag structures, in the present paper we focus on the flag structures associated with distributions.
In section \ref{findimsec} we describe all Jacobi symbols for which the prolongation  procedure ends up in a finite number of steps, i.e. all Jacobi symbols of finite type ( see Theorems \ref{finitetype} and \ref{finitetype1} there). Our considerations there are based on two general criteria for finiteness of type, the Spencer criterium ( see Theorem \ref{spencerthm} and \cite{spencer,gul,ottazzi}) and the Tanaka criterium (see Theorem \ref{Tanakafin} and \cite{tan1}), and some theory of $\sll_2$ representations.
%and distinguish the symplectically flat distributions with given Jacobi symbol, which in the case of finite type are the maximally symmetric distribution among all distribution with given Jacobi symbol. Also, for the most of

Sections \ref{pstandartsec} and \ref{secantsec} are devoted to more detailed study of the resulting prolongation algebra, which is also isomorphic to the algebra of infinitesimal symmetries of the corresponding symplectically flat distribution with given Jacobi symbol $\delta$. This resulting algebra is closely related to the algebra of infinitesimal symmetries of the flat curve of flags with flag symbol $\delta$, see Definition \ref{flatcurve} (actually the latter algebra is equal to the degree $0$ component in the resulting prolongation algebra).
Theorem \ref{genpr_i0} and Corollary \ref{corl(X)} reduce the calculation of the  prolongation algebra to the calculation of the standard prolongation (as described, for example, in \cite{kob, stern}) of certain subspace in the algebra of infinitesimal symmetries of the flat (unparametrized) curve of flags associated with the Jacobi symbol. In particular, we show
 that as such subspace one can take the algebra of infinitesimal symmetries of the flat curve with flag symbol $\delta$ considered as parametrized curve (i.e. only symmetries preserving the parametrization are considered).

 Finally in section \ref{secantsec} we make an attempt to describe the resulting prolongation algebra more explicitly using the fact that if a subspace in the symplectic algebra can be identified with a space of quadratic polynomials vanishing on a given projective variety (in the projectivization of the ambient linear symplectic space), then the standard prolongations of this subspace can be described as the space of certain polynomials vanishing on the secant variety of certain order of this projective variety. For general Jacobi symbols in this way we are able to  show that the resulting prolongation algebra is a subspace in a space of this kind, obtaining an upper bound for the dimension of this resulting prolongation algebra (see Theorem \ref{thmmainsec8}) and an alternative proof of finiteness of type in the considered case, which is more constructive than the proof of Theorem \ref{genpr_i0}.
  %The goal of this section is given a flag symbol satisfying conditions of Theorem \ref{genpr_i0} to show that the $i$th algebraic prolongation $\mathfrak u^i(\eta, \mathfrak u^F(\delta_{\rm mod}))$ is a subspaces of certain spaces of polynomials vanishing on a certain projective variety and
Also, we show (Theorem \ref{thmmainsec8.5}) that for a large class of Jacobi symbols the resulting prolongation algebra can be identified with some spaces of polynomials of this kind (and not only with some subspaces of them).  In this way we get the description of the prolongation algebra for all Jacobi symbols of finite type appearing in rank $2$ and rank $3$ distributions and also for most of such Jacobi symbols appearing in rank $4$ distributions.

For $(2,n)$-distributions there was already an explicit description (see Theorem \ref{rank2thm}), based on the classical work of Cartan \cite{cart10} for $n=5$ and our previous works \cite{doubzel1, doubzel2} for $n>5$.
For rank $3$ distributions the Jacobi symbol is described by a skew Young diagram consisting of $2$ rows. If this skew Young diagram is nonrectangular (and the Jacobi symbol is of finite type) the prolongation algebra can be described  in terms of the tangential developable and the secant varieties of a rational normal curve (Theorem \ref{thmmainsec81}), while in the case of rectangular Young diagram it is described explicitly by Theorem \ref{rank3thm}. For applicability of Theorem \ref{thmmainsec8.5} in the case of rank $4$ distirbution see Remark \ref{234}, item 3 there.

\section{Abnormal extremals of distributions}
\setcounter{equation}{0}
\setcounter{theorem}{0}
\setcounter{lemma}{0}
\setcounter{proposition}{0}
\setcounter{definition}{0}
\setcounter{corollary}{0}
\setcounter{remark}{0}
\setcounter{example}{0}
\label{abnsec}
Let $D$ be a distribution on a manifold $M$. The crucial notion in the symplectification procedure  of
distributions is the notion of \emph{an abnormal extremal}.
%First, under some natural generic assumptions we distinguish a characteristic $1-$
%foliation
%(the foliation of abnormal extremals) on a special odd-dimensional submanifold
%of the cotangent bundle associated with $D$
%any rank 3 distribution $D$
%For this introduce some notations.
%Taking Lie brackets of vector fields tangent to a distribution $D$ (i.e. sections of $D$)   one can define a filtration
To describe them, let $D^{-1}\subset D^{-2}\subset\ldots$
%of the tangent bundle, called
be a \emph{weak derived flag
%or a \emph{small flag
of $D$}, as defined in the beginning of section \ref{Tanakasub}.
% More precisely,
%set
%$D^{-j}$ be the $j$-th power of the distribution $D$.
%$D^{-1}=D$ and define recursively $D^{-j}=D^{-j}+[D,D^{-j+1}]$.
%We
%assume that all $D^{j}$ are subbundles of the tangent bundle.
Now denote by
$(D^j)^{\perp}\subset T^*M$ the annihilator of $D^j$, namely
\[
(D^j)^{\perp}=\{(p,q)\in T_q^*M\mid p\cdot v=0\quad\forall\,v\in D^j(q)\}.
\]
Let $\pi\colon T^*M\mapsto M$ be the canonical projection. For any $\lambda\in
T^*M$, $\lambda=(p,q)$, $q\in M$, $p\in T_q^*M$, let
$\varsigma(\lambda)(\cdot)=p(\pi_*\cdot)$ be the canonical Liouville form
and $\widehat\sigma=d\varsigma$ be the standard symplectic structure on $T^*M$.
Also denote by $(TM)^\perp$ the zero section of $T^*M$.

\begin{definition}
\label{abndef}
An unparametrized Lipshitzian curve in $D^\perp\backslash (TM)^\perp$ is called
an abnormal extremal of a distribution $D$ if the tangent line to it at
almost every point belongs to the kernel of the restriction
$\widehat\sigma|_{D^\perp}$ of $\widehat\sigma$ to $D^\perp$ at this point.
\end{definition}

The term
abnormal extremals comes from Optimal Control Theory: abnormal extremals of
the distribution $D$ are exactly Pontryagin extremals with zero Lagrange multiplier near the
functional for any variational problem with constrains, given by the
distribution $D$. We are interested not in all abnormal extremals but in those satisfying certain generic properties. Their description given below  is based on elementary properties of antisymmetric matrices and their Pfaffians.
Similar description is given in \cite{chitour} (abnormal extremals of our interest here correspond to the lifts of singular curves of \emph{minimal order} in the terminology of \cite{chitour}).

%Now let us study more precisely,  in what submanifold of $T^*M$ foliated by t.
Directly from the definition it follows that any abnormal extremal belongs to the set
\begin{equation}
\label{tildeWD}
\widetilde W_D:=\{\lambda\in D^\perp:{\rm ker}\bigl(\widehat\sigma|_{D^\perp}(\lambda)\bigr)\neq 0\}
\end{equation}

Since ${\rm codim} D^\perp =\rank D$  and a skew-symmetric form in an odd dimensional vector space has a nontrivial kernel, one immediately have the following

\begin{proposition}
\label{Pfaffianlemodd}
If  ${\rm rank}\, D$ is odd, then $\widetilde W_D=D^\perp$.
%, because a skew-symmetric form in an odd dimensional vector space has nontrivial kernel.
\end{proposition}
In order to understand the set $\widetilde W_D$ in the case of distributions of even rank,
let us introduce some notation from the Hamiltonian Dynamics. Given
a function $h:T^*M\mapsto \mathbb R$ denote by $\overrightarrow h$
the corresponding Hamiltonian vector field defined by the
relation $i_{\vec h}\widehat\sigma
%(\overrightarrow G,\cdot)
=-d\,h
%(\cdot)
$. Further, given a vector field $X$ in $M$ define the corresponding ``quasi-impulse'' $H_X:T^*M\rightarrow\mathbb R$, $H_X(p,q)=p\bigl(X(q)\bigr)$, $q\in M,\,p\in T_q^*M$. The corresponding Hamiltonian vector field $\vec H_X$ is called the \emph{Hamiltonian lift of $X$ to $T^*M$} (note that by the definition $\pi_*(\vec H_X)=X$, where $\pi:T^*M\rightarrow M$  is the canonical projection). Note that the following identities hold for any vector fields $X$ and $Y$ on $M$:

\begin{equation}
\label{ident}
d\, H_Y\bigl(\vec H_X\bigr)=H_{[X,Y]}
\end{equation}
(as a matter of fact $d\, H_Y(\vec H_X)$ is exactly the Poisson brackets $\{H_X,H_Y\}$ of the Hamiltonians $H_X$ and $H_Y$).

Now let $(X_1,\ldots, X_l)$ be a local basis of the distribution $D$, $D(q)=\Span\{X_1(q),\ldots, X_l(q)\}$ for any $q$ from some open set of $M$.
Then locally
\begin{equation}
\label{Dperpcoord}
D^\perp=\{H_{X_1}=H_{X_2}=\ldots H_{X_l}=0\}.
\end{equation}
Further, for any $\lambda\in D^\perp$ define the following $l\times l$ antisymmetric matrix
\begin{equation}
\label{Goh}
G(\lambda)=\bigl(H_{[X_i,X_j]}(\lambda)\bigr)_{i,j=1}^l.
\end{equation}

This matrix is called the \emph{Goh matrix (at the point $\lambda$) of the local basis $(X_1,\ldots, X_l)$ of the distribution $D$}.

Using formula \eqref{ident} and Definition \ref{abndef} it is easy to show that locally  a curve $\gamma(\cdot)\subset D^{\perp}$ is an abnormal extremal if and only if
\begin{equation}
\label{coord1}
\gamma'(t)=\sum_{i=1}^l \alpha_i \vec H_{X_i}\bigl(\gamma(t)\bigr) \quad \text{for almost every } t,
\end{equation}
where the vector $(\alpha_1,\ldots,\alpha_l)^T$ belongs to the kernel of the Goh matrix $G\bigl(\gamma(t)\bigr)$,
\begin{equation}
\label{coord2}
(\alpha_1,\ldots,\alpha_l)^T\in {\rm ker} G\bigl(\gamma(t)\bigr), \quad \text{for almost every }   t.
\end{equation}

It is convenient to describe the degeneracy conditions and the kernel for a skew-symmetric matrix using the notion of Pfaffian. Recall that the Pfaffian $\pf (A)$ of an $l\times l$ skew-symmetric matrix $A=(a_{ij})_{i,j=1}^l$ is by definition the square root of its determinant. It is equal to zero if $l$ is odd and to a certain homogeneous polynomial of degree $k$ in the entries of $A$ if $l$ is even, where $k=\frac{1}{2}l$. To describe this polynomial let $\omega=\displaystyle{\sum_{0\leq i<j\leq l} a_{ij} dx^i\wedge dx^j}$ be the corresponding $2$-form on $\mathbb R^l$ taken with standard coordinates $(x^1,\ldots, x^l)$. Then

\begin{equation}
\label{Pfaffdef}
\omega^{k}=k! \pf(A) dx^1\wedge\ldots\wedge dx^l,
\end{equation}
where $\omega^k$ is the $k$th wedge power of $\omega$.

%The properties of the Pfaffian
As a consequentce of \eqref{coord1} and \eqref{coord2}, we have the following

\begin{proposition}
\label{Pfafflemeven}
If $\rank D$ is even, the set $\widetilde W_D$, defined by \eqref{tildeWD}, is given locally by the following formula
\begin{equation}
\label{Pfaff1}
\widetilde W_D=\{\lambda\in D^{\perp}: \pf\bigl(G(\lambda)\bigr)=0\}.
\end{equation}
%where $\pf(A)$ denotes the Pfaffian of the antisymmetric matrix $A$.
\end{proposition}
For example, if $D$ is a rank $2$ distributions
\begin{equation}
\label{Pfaffrank2}
G(\lambda)=\begin{pmatrix}0&H_{[X_1,X_2]}(\lambda)\\-H_{[X_1,X_2]}(\lambda)&0\end{pmatrix}, \quad \pf\bigl(G(\lambda)\bigr)=H_{[X_1,X_2]}(\lambda).
\end{equation}
Hence
\begin{equation}
\label{tildeWrank2}
\text{if $l=2$, then } \widetilde W_D=(D^{-2})^\perp,
\end{equation}
because Proposition \ref{Pfafflemeven} implies locally that
$$\widetilde W_D=\{H_{X_1}=H_{X_2}=H_{[X_1,X_2]}=0\}=(D^{-2})^\perp.$$
Since we originally work over $\mathbb R$, for $l$ divisible by $4$ the set $\widetilde W_D$ might be empty (such distributions are called \emph{fat} (\cite{montbook}). However, in this case we can complexify the fibers of $T^*M$. Denote the obtained complex vector bundle (over real manifold $M$) by $(T^*M)^{\mathbb C}$. In the same way we define the complexified bundles $\bigl((D^j)^\perp\bigr)^{\mathbb C}$ and the Goh matrix $G(\lambda)$ for any $\lambda\in (D^\perp)^{\mathbb C}$. So, the set $(\widetilde W_D)^{\mathbb C}=\{\lambda\in (D^{\perp})^{\mathbb C}: \pf\bigl(G(\lambda)\bigr)=0\}$ is not empty. In the sequel in case we need to work with complexified objects we will do it without special mentioning and use the same notation as in the real category.

%Now let us describe the submanifold of $T^*M$ foliated by the abnormal extremals.
%First set
%\begin{equation}
%\label{tildeWD}
%\widetilde W_D:=\{\lambda\in D^\perp:{\rm ker}\bigl(\hat\sigma|_{D^\perp}(\lambda)\bigr)\neq 0\}
%\end{equation}
%Then
%\begin{enumerate}
%\item
%If ${\rm rank}\, D$ is odd, then $\widetilde W_D=D^\perp$, because a skew-symmetric form in an odd dimensional vector space has nontrivial kernel;
%\item
%If ${\rm rank}\, D=2$, then it is easy to show \cite[Proposition 2.2]{zelrigid} that $\widetilde W_D=(D^{-2})^{\perp}$;
%\item
%More generally, if ${\rm rank}\, D=2k$ then  the intersection of $\widetilde W_D$ with the fiber $D^\perp(q)$ of $D^{\perp}$ over a point $q\in M$
%is a zero level set of a certain homogeneous polynomial of degree $k$ on $D^\perp(q)$. This polynomial at a point $\lambda=(p, q)\in D^\perp$ is the  Pfaffian of the so-called Goh matrix at $\lambda$: if $X_1,\ldots X_k$ is a local basis of the distribution $D$, then the Goh matrix at $\lambda$ (w.r.t. this basis) is the matrix $\bigl(p\cdot[X_i,X_j](q)\bigr)_{i,j=1}^{2k}$.
%%In particular, $\widetilde W_D$ is a codimension $1$ submanifold of $D^\perp$
%\end{enumerate}
%
%In the sequel, for simplicity, we will assume that ${\rm rank}\, D$ is odd or ${\rm rank}\, D=2$.
%In both cases $\widetilde W_D$ is an odd dimension submanifold.
%Therefore ${\rm ker}\bigl(\hat\sigma|_{\widetilde W_D}(\lambda)\bigr)$ is nontrivial.
Now define a subset $W_D$ of $\widetilde W_D$ as follows:

\begin{equation}
\label{WD}
W_D:=\{\lambda\in \widetilde W_D:{\rm ker}\bigl(\widehat\sigma|_{\widetilde W_D}(\lambda)\bigr) \text{ is one-dimensional}\}.
\end{equation}

By constructions, the kernels of $\widehat\sigma|_{W_D}$ form the \emph{characteristic rank $1$ distribution} $\widehat{\mathcal C}$ on $W_D$ and
the integral curves  of this distribution are
%called \emph{
%regular
abnormal extremals of the distribution $D$. Actually, they are exactly the abnormal extremals which lie entirely in $W_D$. Such abnormal extremals are called \emph{regular}.
Note that in general an abnormal extremal does not need to belong to $W_D$ , however, for our purposes it is enough to work with regular abnormal extremals only.

Now describe the set $W_D$ and the characteristic rank $1$ distribution $\widehat{\mathcal C}$ in terms of a local basis $(X_1,\ldots, X_l)$ of the distribution $D$.
%Again the description is different in the cases of odd and even rank.
For this recall several facts on kernels of skew-symmetric matrices using the notion of Pfaffian. Let $A=(a_{ij})_{i,j=1}^l$ be an $l\times l$ skew-symmetric matrix. Let $A_{i}$ be the Pfaffian of the skew-symmetric matrix obtained from $A$ by removing the $i$th row and the $i$th column.  Further, for $i<j$ let $A_{ i j}$ be the Pfaffian of the skew-symmetric matrix obtained from $A$ by removing the $i$th and $j$th row and the $i$th and $j$th column. If $A$ was a $2\times 2$ matrix, so that after such removal we get an empty matrix, we set $A_{{1}{2}}=1$.
To define $A_{j i}$ for any ordered pair $(i,j)$ (with integer $i$ and $j$ running from $1$ to $l$) set
\begin{equation}
\label{extpfaffian}
A_{ j i}=-A_{i  j}, \quad A_{i i}=0.
\end{equation}

Then the following proposition  is well known and can be easily obtained from \eqref{Pfaffdef}:

\begin{proposition}
\label{skewkernel}

\begin{enumerate}

\item
If $l$ is odd, then $A$ has one-dimensional kernel if and only if there exists
 an index $i_0$ such that $\pf(A_{{i}_0})\neq 0$. In this case $\ker A$ is spanned by the vector
 \begin{equation}
 \label{kerodd}
 \bigl(A_{ 1}, -A_{2},\ldots, (-1)^{i-1} A_{ i},\ldots, A_{l}\bigr)^T.
 \end{equation}
\item If $l$ is even then
\begin{equation}
\label{kerformeven}
\sum_{j=1}^l (-1)^j a_{sj}A_{i j}=\begin{cases} 0 &i\neq s\\(-1)^{s-1} \pf(A) & i=s. \end{cases}
\end{equation}
In particular,
\begin{enumerate}
\item
 if $\pf(A)=0$, then  for any index $1\leq i\leq l$ the vector
\begin{equation}
\label{videf}
v_i=\bigl(-A_{i 1}, A_{ i  2},\ldots, (-1)^{j} A_{ i  j},\ldots, A_{ i l}\bigr)^T
\end{equation}
belongs to $\ker A$.

\item
The matrix $A$ has $2$-dimensional kernel if and only if there exist indices $i_0$ and $j_0$ such that $A_{{i}_0 {j}_0}\neq 0$. In this case,

\begin{equation}
\label{kereven}
 \ker A=\Span\{v_{i_0}, v_{j_0}\}
\end{equation}
where vectors $v_{i_0}$ and $v_{j_0}$ are as in \eqref{videf}.
\end{enumerate}

\end{enumerate}

\end{proposition}

Now apply this proposition to the Goh matrix $G(\lambda)$ using notations $G_i(\lambda)$ and
$G_{ i j}(\lambda)$ for the pfaffians of the corresponding skew-symmetric submatrices of $G(\lambda)$ (obtained from $G(\lambda)$ by removing the corresponding rows and columns with the relation analogous to \eqref{extpfaffian}).
%Denote by $G_i(\lambda)$ the antisymmetric matrix obtained from the Goh matrix $G(\lambda)$
%by removing the $i$th row and the $i$th column.

\subsection{The case of distributions of odd rank}
%from the basic properties of antisymmetric matrices and formulas \eqref{coord1}-\eqref{coord2},
>From item (1) of Proposition \ref{skewkernel} it follows that
if $l=\rank D$ is odd,
then locally

\begin{equation}
\label{WDcoord}
\begin{split}
~&W_D=\{\lambda\in D^\perp:\text{ there exist $i$, $1\leq i\leq l$ such that } G_{ i}(\lambda)\neq 0\}=\\
~&D^\perp\backslash\{\lambda\in \widetilde D:G_{1}(\lambda)=\ldots=G_{ l}(\lambda)=0\},
\end{split}
\end{equation}
\begin{equation}
\label{Codd}
\widehat{\mathcal C}(\lambda)=\Span\left\{\sum_{i=1}^l(-1)^{i-1}G_{ i}(\lambda)\vec H_{X_i}(\lambda)\right\}.
\end{equation}
>From \eqref{WDcoord} it follows that for any $q\in M$ the complement of $W_D(q)$ ($:=W_D\cap D^\perp(q)$) to the fiber $D^\perp(q)$ of $D^\perp$ is an algebraic variety depending on the \emph{second jet} of $D$ at $q$. It is easy to construct examples of germs of $(l,n)$-distributions with odd $l$ and $n>l$ such that this complement is a proper algebraic subvariety of the corresponding fiber of the annihilator of the distribution (see subsection \ref{flagsec4}). This implies  that \emph{for generic germs of distributions of odd rank $l$ on manifolds of dimension greater than  $l$ the set $W_D$ is an open and dense subset of $\widetilde W_D$}.

For example, if $D$ is a rank 3 distribution,  then
$$G(\lambda)=\begin{pmatrix}0&H_{[X_1,X_2]}(\lambda)&H_{[X_1,X_3]}(\lambda)\\-H_{[X_1,X_2]}(\lambda)&0&H_{[X_2,X_3]}(\lambda)\\-H_{[X_1,X_3]}(\lambda)&
-H_{[X_2,X_3]}(\lambda)&0 \end{pmatrix},$$

$$G_{1}(\lambda)=H_{[X_2,X_3]}(\lambda),\quad G_{2}(\lambda)=H_{[X_1,X_3]}(\lambda), \quad G_{ 3}(\lambda)=H_{[X_1,X_2]}(\lambda).$$
Note that $\{\lambda\in D^{\perp}: H_{[X_1,X_2]}(\lambda)=H_{[X_1,X_3]}(\lambda)=H_{[X_2,X_3]}(\lambda)=0\}=(D^{-2})^\perp$.
Consequently, in this case by \eqref{WDcoord}  we have
\begin{equation}
\label{WD3}
 W_D=D^\perp\backslash (D^{-2})^\perp
\end{equation}

and by \eqref{Codd}
\begin{equation}
\label{C3}
\widehat{\mathcal C}(\lambda)=\Span\left\{H_{[X_2,X_3]}(\lambda)\vec H_{X_1}(\lambda)-H_{[X_1,X_3]}(\lambda)\vec H_{X_2}(\lambda)+H_{[X_1,X_2]}(\lambda)\vec H_{X_3}(\lambda)\right\}.
\end{equation}

\subsection{The case of distributions of even rank}
Now let us describe the set $W_D$ for distributions of even rank.
%Now consider the case of distributions of even rank $l\geq 4$.
%For any $1\leq i <j\leq l$ denote by $G_{ij}(\lambda)$ the antisymmetric matrix obtained from the Goh matrix $G(\lambda)$
%by removing the $i$th and $j$th rows and the $i$th and $j$th columns and set $G_{ii}(\lambda):=0$, $1\leq i\leq l$.
For any $i$, $1\leq i\leq l$, define the following vector field on $T^*M$:
\begin{equation}
\label{Ydef}
\mathcal Y_i(\lambda) =\sum_{j=1}^l (-1)^j G_{ {i} {j}}(\lambda)\vec H_{X_j}(\lambda).
\end{equation}
Then from item (2)(a) of Proposition \ref{skewkernel} it follows that
%, using basic properties of Pfaffians, one can
%show that
for any $\lambda\in \widetilde W_D$
\begin{equation}
\label{Yi}
\mathcal Y_i(\lambda)\in {\rm ker} \bigl(\widehat\sigma|_{D^{\perp}}(\lambda)\bigr).
\end{equation}
The set  $W_D$ does not contain the following subset $\mathcal S_1$ of  $\widetilde W_D$:
$$\mathcal S_1=\{\lambda\in \widetilde W_D: \dim\, {\rm ker} \bigl(\widehat\sigma|_{D^{\perp}}(\lambda)\bigr)>2\}.$$
By item (2)(b) of Proposition \ref{skewkernel} locally the set $\mathcal S_1$ can be described as follows:
\begin{equation}
\label{S1def}
\mathcal S_1=\{\lambda\in D^{\perp}: \pf G(\lambda)=0 \text{ and }   G_{{i}{j}}(\lambda) =0 \text{ for all }  1\leq i<j\leq l \,\}.
\end{equation}
Besides, by the same item (2)(b) of Proposition \ref{skewkernel}, for any $\lambda\in \widetilde W_D\backslash \mathcal S_1$ the kernel of $\bigl(\widehat\sigma|_{D^{\perp}}(\lambda)\bigr)$ is $2$-dimensional so that
\begin{equation}
\label{kernel}
\text{if
$G_{{i_0}{j_0}}(\lambda)\neq 0$  for some $1\leq i_0<j_0\leq l$ , then }   {\rm ker} \bigl(\widehat\sigma|_{D^{\perp}}(\lambda)\bigr)=\Span\{\mathcal Y_{i_0}(\lambda), \mathcal Y_{j_0}(\lambda)\}
\end{equation}
Note that for $\lambda\in \widetilde W_D\backslash \mathcal S_1$ the kernel of the form $\widehat\sigma|_{\widetilde W_D}(\lambda)$ is one-dimensional if and only if the kernel of the form $\widehat\sigma|_{D^{\perp}}(\lambda)$ is not tangent to $\widetilde W_D$.
Therefore,
%in this case if $G_{{i_0}{j_0}}\neq 0$
locally
\begin{equation}
\label{WDeven}
%\begin{split}
W_D= \widetilde W_D\backslash \Bigl(\mathcal S_1\cup\mathcal S_2\Bigr),
\end{equation}
where

\begin{equation}
\label{S2}
\mathcal S_2= \{\lambda\in \widetilde W_D
%\backslash \mathcal S_1
: d(\pf G)(\lambda) \bigl(\mathcal Y_{i_0}(\lambda)\bigr)=d(\pf G)(\lambda) \bigl(\mathcal Y_{j_0}(\lambda)\bigr)=0\}\Bigr),
\end{equation}
and the characteristic rank $1$ distribution $\widehat C(\lambda)$ satisfies
\begin{equation}
\label{Ceven}
\widehat C(\lambda)=\Span\left\{ d(\pf G)(\lambda)\bigl( \mathcal Y_{i_0}(\lambda)\bigr)\mathcal Y_{j_0}(\lambda)-d(\pf G)(\lambda) \bigl(\mathcal Y_{j_0}(\lambda)\bigr)\mathcal Y_{i_0}(\lambda)\right\},
%\end{split}
\end{equation}
where the pair of indices $(i_0,j_0)$ is as in \eqref{kernel}.

>From \eqref{S1def} and \eqref{WDeven}  it follows that for any $q\in M$ the complement of $W_D(q)$ ($:=W_D\cap D^\perp(q)$) to the fiber $\widetilde W_D(q)$ is an algebraic variety depending on the \emph{third jet} of $D$ at $q$. It is easy to construct examples of germs at $q$ of $(l,n)$-distributions with even $l$ and $n>l$ such that this complement is an algebraic subvariety (of $D^\perp(q)$), which is a proper subset of  $\widetilde W_D(q)$ (see subsection \ref{flagsec4}). This implies  that \emph{for generic germs of distributions of even rank $l$ on manifolds of dimension greater than  $l$ the set $W_D$ is an open subset of $\widetilde W_D$}.
%\emph{for generic germs of distributions of even rank $l$ on a manifold $M$ of dimension greater than  $l+1$ the set $W_D$ is an open and dense subset of $\widetilde W_D$}.

 Let us consider the case of rank $2$ distributions in more detail (see also \cite{doubzel2, zelrigid}). In this case the set $\mathcal S_1$ is empty. Also, from \eqref{ident} \eqref{Pfaffrank2}, and \eqref{S2} it follows that $\mathcal S_2=(D^{-3})^\perp$. Substituting this into \eqref{WDeven} we get
 \begin{equation}
\label{Wdrank21}
%\begin{split}
W_D=(D^{-2})^\perp\backslash (D^{-3})^\perp.
 \end{equation}
 %kernel of the form $\widehat\sigma|_{D^{\perp}}(\lambda)$ are tangent to $\widetilde W_D$ . From \eqref{tildeWrank2} it follows that in this case an abnormal extremal $\gamma(\cdot)$ must satisfy
%$d H_{[X_1,X_2]}(\gamma(t))\gamma'(t)=0$.
Besides, $\mathcal Y_1=\vec H_{X_1}$ and $\mathcal Y_2=-\vec H_{X_2}$.
Substituting these formulas into \eqref{Ceven} and using  and using \eqref{Pfaffrank2} we get
 that
\begin{equation}
\label{Wdrank22}
%\begin{split}
%~&W_D=(D^{-2})^\perp\backslash (D^{-3})^\perp,\\
%~&
\widehat{\mathcal C}(\lambda)=\Span\left\{H_{\bigl[X_1,[X_1,X_2]\bigr]}(\lambda)\vec H_{X_2}(\lambda)-H_{\bigl[X_2,[X_1,X_2]\bigr]}(\lambda)\vec H_{X_1}(\lambda)\right\}.
%\end{split}
\end{equation}

%By direct calculation one can show that
%\begin{enumerate}
%\item If ${\rm rank}\, D=2$, then $W_D=(D^{-2})^\perp\backslash (D^{-3})^\perp$ (\cite[section 2]{zelrigid});
%\item If ${\rm rank}\, D=3$, then $W_D=D^\perp\backslash (D^{-2})^\perp$ (\cite{doubzel3}).
%\end{enumerate}

 Finally note that in all cases the characteristic rank $1$ distribution $\widehat{\mathcal C}$ can be computed explicitly in terms of a local basis of the distribution and it depends polynomially on the fibers of $W_D$.

\begin{remark}
\label{3brackrem}
 Note that the set $\mathcal S_2$, defined by \eqref{S2}, depends on brackets of length $3$ of a local  basis of the distribution $D$. Therefore if  $D$ is a flat distribution in Tanaka sense, of even rank,   and of degree of nonholonomy $2$ (i.e., such that  $D\neq TM$ and $D^2=TM$), then  the set $\mathcal S_2$ coincides with $\widetilde W_D$, i.e. the set $W_D$ is empty and $D$ has no regular abnormal extremals. Therefore our subsequent theory does not cover such distributions.
For example, this is the case for two possible flat $(4,7)$-distributions with $\dim D^2=7$.
\end{remark}

%Consequently, for any bracket generating rank $2$ or rank $3$ distribution on a manifold $M$ with
%$\dim M\geq 4$ the set $W_D$ is an open and dense subset of $\widetilde W_D$.
%In the sequel we will assume that the set $W_D$ is an open and dense subset of $\widetilde W_D$.
%Note that this is a generic assumption for distributions of odd rank greater than $3$. See also Remark \ref{WDempt} below addressing the case when this assumptions does not hold.
%If $rank D$ is odd and greater than $3$ then the set $W_D$ is an open and dense subset of $\widetilde W_D$ for generic $D$.
%Since $\dim D=3$, the submanifold $D^\perp$ has odd codimension in $T^*M$, and
% the
%abnormal extremals of the distribution $D$ lying in the complement to
%$(D^2)^\perp$.

\section{Jacobi curves of regular abnormal extremals}
\setcounter{equation}{0}
\setcounter{theorem}{0}
\setcounter{lemma}{0}
\setcounter{proposition}{0}
\setcounter{definition}{0}
\setcounter{corollary}{0}
\setcounter{remark}{0}
\setcounter{example}{0}
\label{jaccurvesec}

>From now on we consider distributions $D$ such that the set $W_D$ is an open subset in $\widetilde W_D$.
Having the manifold $
%\mathbb P
W_D$ foliated by regular abnormal extremals at our disposal, we can make a symplectification of the problem by lifting the original distribution $D$ to $
%\mathbb P
W_D$. The dynamics of this lift along a regular abnormal extremal is described by a curve of flags consisting of isotropic and coisotropic subspaces in a linear symplectic space, which is stable under the skew-orthogonal complement (or, shortly, of curves of symplectic flags). Such curves have a nontrivial geometry that illuminates the geometry of the distribution $D$ itself.

To begin with,  let $
 %\tilde
 \pi:
 %\mathbb P
 T^*M\rightarrow M$ be the canonical projection. The \emph{lift} $\widehat{\mathcal J}$  of the distribution $D$ to $W_D$ is the pullback of the original distribution $D$ to $
 %\mathbb P
 W_D$ by the canonical projection $
 %\tilde
 \pi$:
\begin{equation}
\label{pullJ}
\widehat{\mathcal J}(\lambda)=\{v\in T_\lambda
%\mathbb P
W_D:
%\tilde
\pi_* v\in D(
%\tilde
\pi(\lambda))\}.
\end{equation}
Further, let $\widehat {\mathcal V}$  be the distribution of the tangent spaces to the fibers of $W_D$ or the \emph{vertical distribution}
\begin{equation}
\widehat {\mathcal V}(\lambda)=\{v\in T_\lambda W_D: %\tilde
\pi_* v=0\}.
\end{equation}
%, i.e the tangent space to the fibers of \tb{$\mathcal H_D$}.
%\pause

Obviously $\widehat {\mathcal V}\subset \widehat {\mathcal J}$. Also, from \eqref{coord1} it follows that $\widehat{\mathcal C}\subset \widehat {\mathcal J}$. Hence,
$\widehat {\mathcal V}+\widehat{\mathcal C}\subset \widehat {\mathcal J}$ so that the distribution $\widehat{\mathcal J}$ has at least two distinguished subdistributions and, in particular, a distinguished line subdistribution $\widehat{\mathcal C}$, while the original distribution $D$ does not have such additional structures in general. Therefore it is much more simple to study the geometry of distribution $\widehat{\mathcal J}$ instead of the geometry of the original distribution $D$ and under our assumptions the geometry of $D$ is recovered from the geometry of $\widehat{\mathcal J}$ (see Corollary \ref{recovercor} below).
%\mathcal C$ allows us to produc

Note that the distribution $\widehat{\mathcal V}$ is involutive. Also, given a point $q\in M$ for any $\lambda$ belonging to the fiber $W_D(q)$ and any vector $X\in \mathcal J(\lambda)$ the vector $\pi_*X$ belongs to the same vector space $D(q)$. The last two facts easily imply that the Lie brackets of a vertical vector field, i.e. a section of the vertical distribution $\widehat{\mathcal V}$, with a section of the distribution $\widehat{\mathcal J}$ is a section of $\widehat{\mathcal J}$.
In particular, we have the following relation:
%, which follows directly from definitions:

\begin{equation}
\label{VC}
[\widehat{\mathcal C},\widehat{\mathcal V}]\subset \widehat{\mathcal J}.
\end{equation}

Given a subspace $\Lambda$ of $T_\lambda W_D$
%$\widetilde\Delta(\lambda)$
denote by $\Lambda^\angle$ the skew-orthogonal complement of $\Lambda$ with respect to the skew-symmetric form $
%\widetilde
\widehat\sigma|_{W_D}(\lambda)$: $\Lambda^\angle=\{v\in \Lambda: %\widetilde
\widehat\sigma(v,w)=0 \,\, \forall\, w\in \Lambda\}$. The subspace $\Lambda$ is called \emph{isotropic} if $\Lambda\subset\Lambda^\angle$ and \emph{coisotropic} if $\Lambda^\angle\subset\Lambda$.

\begin{proposition}
\label{isolem} The following statement hold:
\begin{enumerate}
\item

The subspace $\widehat{\mathcal J}(\lambda)^\angle$ satisfies
%the following properties:
\begin{equation}
\label{Jskew}
\widehat{\mathcal J}(\lambda)^\angle=\{w\in T_\lambda W_D: \pi_* w\in \pi_*{\rm ker}(\widehat\sigma|_{D^{\perp}}(\lambda))\}.
\end{equation}

%and it satisfy the following statements hold:

\item  $\widehat{\mathcal J}(\lambda)^\angle\subseteq \widehat{\mathcal J}(\lambda)$, i.e. the space $\widehat{\mathcal J}(\lambda)$ is coisotropic with respect to the form $\widehat\sigma$;
\item  If the distribution $D$ is of odd rank, then
\begin{equation}
\label{recoverVodd}
\widehat{\mathcal J}(\lambda)^\angle=\widehat{\mathcal V}(\lambda)\oplus \widehat{\mathcal C}(\lambda)
\end{equation}

\item If the distribution $D$ is of even rank then
\begin{equation}
\label{Jskeweven}
\dim\Bigl(\widehat{\mathcal J}(\lambda)^\angle\Bigr)=\dim\Bigl(\widehat{\mathcal V}(\lambda)\oplus \widehat{\mathcal C}(\lambda)\Bigr)+1
\end{equation}
and
%\item If the distribution $D$ is of even rank, then
\begin{equation}
\label{CJ^angle}
\Bigl(\widehat{\mathcal V}(\lambda)\oplus \widehat{\mathcal C}(\lambda)\Bigr)^\angle=\widehat{\mathcal J}(\lambda)+[\widehat C,\bigl(\widehat{\mathcal J}(\lambda)\bigr)^\angle].
\end{equation}
\item If the distribution $D$ is of rank $2$, then

\begin{equation}
\label{lagr1}
\widehat{\mathcal J}(\lambda)^\angle=\widehat{\mathcal J}(\lambda)
\end{equation}
and
\begin{equation}
\label{pullD2}
[\widehat C,\widehat{\mathcal J}(\lambda)]=\{v\in T_\lambda
%\mathbb P
W_D:
%\tilde
\pi_* v\in D^{-2}(
%\tilde
\pi(\lambda))\}.
\end{equation}
\end{enumerate}

\end{proposition}

\begin{proof}
%Note that for $v\in \widehat V(\lambda)$ and $w\in T_\lambda W_D$ one has
%$\sigma(v,w)=v(\pi_*w)$,  Since $W_D\subseteq D^\perp$, this implies that $\widehat %V(\lambda)\subseteq \widehat{\mathcal J}(\lambda)^\angle$.

%Besides,
If $w\in T_\lambda W_D$ such that  $\pi_* w\in \pi_*{\rm ker}(\widehat\sigma|_{D^{\perp}}(\lambda))$, then $w=w_1+v$, where $w_1\in {\rm ker}(\widehat\sigma|_{D^{\perp}}(\lambda))$ and $v\in T_\lambda D^\perp\bigl(\pi(\lambda)\bigr)$. Since $\widehat{\mathcal J}(\lambda)\subset T_\lambda D^\perp$, for any $z\in \widehat{\mathcal J}(\lambda)$ one has $\widehat\sigma(w_1,z)=0$. Also, $\widehat\sigma(v,z)=v(\pi_*z)=0$, where in the second term of this chain of equalities $v$ is considered as an element of $D^\perp\bigl(\pi(\lambda)\bigr)$ under the natural identification of the tangent space to a vector space with the vector space itself. Therefore,  $\widehat\sigma(w,z)=0$, i.e. the righthand side of \eqref{Jskew} is contained in $\widehat{\mathcal J}(\lambda)^\angle$.

The equality in \eqref{Jskew} follows from the dimension count and the fact that the kernel of  $
%\widetilde
\widehat\sigma|_{W_D}(\lambda)$ is one-dimensional. Indeed, let ${\rm rank}\, D=l$ and   $\dim M=n$. Also, set $\varepsilon=0$, if $l$ is odd, and $\varepsilon =1$ if $l$ is even.
Then $\dim T_\lambda W_D=2n-l-\varepsilon$, $\dim \widehat {\mathcal V}(\lambda)=n-l-\varepsilon$,  dimension of the righthand side of  \eqref{Jskew} is equal to
$\dim \widehat {\mathcal V}(\lambda) +1+\varepsilon=n-l+1$, and $\dim \widehat J(\lambda)=\dim \widehat {\mathcal V}(\lambda)+{\rm rank} D=n-\varepsilon$. The latter implies that $\dim \widehat J(\lambda)^\angle=\dim T_\lambda W_D-\dim \widehat J(\lambda)+1=n-l+1$, i.e. dimension of the righthand side of \eqref{Jskew} is equal to $\dim \widehat J(\lambda)^\angle$, which together with the conclusion of the last paragraph implies \eqref{Jskew}, i.e. item (1).

Item (2) of  the lemma follows from the definition of the subspace  $\widehat {\mathcal J}(\lambda)$ given by \eqref{pullJ} and the fact that $\pi_*{\rm ker}(\widehat\sigma|_{D^{\perp}}(\lambda))\subset D\bigl(\pi(\lambda)\bigr)$. Item (3) follows from the fact that if $l$ is odd then ${\rm ker}\Bigl(\widehat\sigma|_{D^{\perp}}(\lambda)\Bigr)=\widehat{\mathcal C}(\lambda)$.
Formula \eqref{Jskeweven} of item (4) follows from the fact that if $l$ is even, then $\dim \Bigl({\rm ker}(\widehat\sigma|_{D^{\perp}}(\lambda)\Bigr)=2$.

Now let us prove formula \eqref{CJ^angle} of item (4). We shall use a local basis $X_1,\ldots, X_l$ of the distribution $D$. Let the matrix $G(\lambda)$ be as in \eqref{Goh}, and $G_{ i  j}$  are defined as in the paragraph before Proposition \ref{skewkernel}. From item (1) and formulas \eqref{kernel}, \eqref{Ceven} it follows that if
$G_{{i_0}{j_0}}(\lambda)\neq 0$  for some $1\leq i_0<j_0\leq l$ , then
\begin{equation}
\label{YY1}
\widehat{\mathcal J}(\lambda)+[\widehat C,\bigl(\widehat{\mathcal J}(\lambda)\bigr)^\angle] =\{w\in T_\lambda W_D: \pi_* w\in D\bigl(\pi(\lambda)\bigr)+\pi_*\bigl([\mathcal Y_{i_0}, \mathcal Y_{j_0}](\lambda)\bigr)\}.
\end{equation}
Then from \eqref{Ydef} it follows by direct computations that

\begin{equation}
\label{YY2}
[\mathcal Y_{i_0}, \mathcal Y_{j_0}]=\sum_{1\leq s_1,s_2\leq l} (-1)^{s_1+s_2}G_{ i  {s_1}} G_{ j {s_2}} \vec H_{[X_{s_1}, X_{s_2}]} \quad \mathrm{mod}\,\mathrm{span} \{\vec H_{X_1},\ldots, \vec H_{X_l}\}.
\end{equation}
This implies that for $v\in \widehat{\mathcal V}(\lambda)$ we have
\begin{equation}
\label{YY3}
\widehat\sigma\bigl([\mathcal Y_{i_0}, \mathcal Y_{j_0}],v\bigr)=-\sum_{s_1=1}^l (-1)^{s_1}G_{ i  {s_1}}\sum _{s_2=1}^l (-1)^{s_2} G_{ j_0 {s_2}} d H_{[X_{s_1}, X_{s_2}]}(v).
\end{equation}

Using  \eqref{kerformeven} and the fact that $d (\mathrm {pf}  G)(v)=0$, we get

$$\sum _{s_2=1}^l (-1)^{s_2} G_{ j_0 {s_2}} d H_{[X_{s_1}, X_{s_2}]}(v)=-\sum _{s_2=1}^l (-1)^{s_2} H_{[X_{s_1}, X_{s_2}]}dG_{ j_0 {s_2}}(v)$$
Substituting this into \eqref{YY3} we obtain that

\begin{equation}
\label{YY4}
\widehat\sigma\bigl([\mathcal Y_{i_0}, \mathcal Y_{j_0}](\lambda),v\bigr)=\sum_{s_2=1}^l (-1)^{s_2}\left(\sum _{s_1=1}^l (-1)^{s_1} H_{[X_{s_1}, X_{s_2}]} G_{ i_0 {s_1}}\right)dG_{ j_0  {s_2}}(v).
\end{equation}
Using \eqref{kerformeven} again and the fact that $\mathrm{pf} \,G$ vanishes on $W_D$, we obtain that $$\sum _{s_1=1}^l (-1)^{s_1} H_{[X_{s_1}, X_{s_2}]}(\lambda) G_{ i_0 {s_1}}(\lambda)=0 ,\quad \forall \lambda\in W_D.$$ Therefore,
\begin{equation}
\label{YY5}
\widehat\sigma\bigl([\mathcal Y_{i_0}, \mathcal Y_{j_0}](\lambda),v\bigr)=0.
\end{equation}

Taking into account relation \eqref{YY1} and the fact that the restriction of the form $\widehat\sigma$ to the $\widehat{\mathcal V}(\lambda)$ vanishes we get that $\widehat{\mathcal V}+\mathcal C$ is contained in the skew-orthogonal complement of $\widehat{\mathcal J}(\lambda)+[\widehat C,\bigl(\widehat{\mathcal J}(\lambda)\bigr)^\angle]$. Counting the dimension, we obtain in fact that it is equal to this skew-orthogonal complement, which completes the proof of item (4).

Relation \eqref{lagr1} of item (5) follows immediately from formula \eqref{Jskeweven} of item (4) and the fact that $\dim \widehat{\mathcal J}-\dim \widehat {\mathcal V}(\lambda)=2$ for $l=2$.
Finally, relation \eqref{pullD2} of item (5) follows  from \eqref{YY2}.
%and \eqref{Wdrank22}.
%from direct computations using \eqref{Wdrank22} and the second formula of \eqref{pullJ}.
\end{proof}

As a direct consequence of \eqref{recoverVodd} in the odd rank case and of \eqref{CJ^angle} in the even rank case, and the fact that the distirbution $\widehat V$ is the maximal rank involutive subdistributon of $\widehat{\mathcal J}$ we have the following important:

\begin{corollary}
\label{recovercor}
The original distribution $D$ can be uniquely recovered from its lift $\widehat {\mathcal J}$.
\end{corollary}
This corollary justify that the original equivalence problem for distributions on $M$ can be reduced to the equivalence problem for their lifts.

It is more convenient to work with the projectivization of  $\mathbb PT^*M$ rather than with $T^*M$.
Here $\mathbb PT^*M$ is the fiber bundle over $M$ with the fibers that are the projectivizations of the fibers of $T^*M$.
The canonical projection $\Pi\colon T^*M\rightarrow \mathbb PT^*M$ sends the characteristic distribution
$\widehat{\mathcal C}$ on $W_D$ to the line distribution $\mathcal C$ on
$\mathbb P W_D$ ($:=\Pi(W_D)$), which will be also called the \emph{characteristic distribution}
of the latter manifold.
%The manifold $\mathbb P W_D$ and the distribution $\mathcal A$ play the role of $\widetilde {\mathfrak M}$ and $\mathcal C$, respectively, from the general linearization procedure of subsection \ref{stat}.

Note that the corank 1 distribution on $T^*M
%\backslash S_0
$ annihilating the tautological Liouville form $\varsigma$ on $T^*M$ induces a contact distribution on $\mathbb P T^*M$, which in turns induces the even-contact (quasi-contact) distribution $\widetilde\Delta$
on $\mathbb P(D^2)^\perp\backslash \mathbb P(D^3)^\perp$. The characteristic line distribution $\mathcal C$ is exactly the Cauchy characteristic distribution of $\widetilde\Delta$, i.e. it is the maximal subdistribution of $\widetilde\Delta$ such that
\begin{equation}
\label{Cauchy}
[\mathcal C,\widetilde \Delta]\subset \widetilde \Delta.
\end{equation}

Further set $\mathcal J=\Pi_*\widehat{\mathcal J}$, $\mathcal V=\Pi_*\widehat              {\mathcal V}$.
Each fiber of the even-contact distribution $\widetilde \Delta$ is endowed with the canonical skew-symmetric form  $\widetilde \sigma$, up to multiplication by a non-zero constant.
%i.e with a canonical conformal symplectic structure.
This form is equal to the restriction to $\widetilde \Delta$ of the differential of any non-zero $1$-form annihilating this distribution.
 %or, equivalently, to the pushforward of the form $\sigma|W_d$ by $\Pi$ restricted to the distribution $\widetilde \Delta$. the latter description together with
% From Lemma \ref{isolem} it follows immediately that
%
% \begin{lemma}
% \label{Jskew1}
% The following statements hold:
%\begin{enumerate}
%\item  ${\mathcal J}(\lambda)^\angle\subseteq {\mathcal J}(\lambda)$, i.e. the space ${\mathcal J}(\lambda)$ is coisotropic with respect to the form $\widetilde\sigma$;
%\item  If the distribution $D$ is of odd rank then
%\begin{equation*}
%%\label{Jskewodd}
%{\mathcal J}(\lambda)^\angle={\mathcal V}(\lambda)\oplus {\mathcal C}(\lambda)
%\end{equation*}
%
%\item If the distribution $D$ is of even rank then
%\begin{equation*}
%%\label{Jskeweven}
%\dim\Bigl(\mathcal J(\lambda)^\angle\Bigr)=\dim\Bigl(\mathcal V(\lambda)\oplus {\mathcal C}(\lambda)\Bigr)+1
%\end{equation*}
%
%\item If the distribution $D$ is of rank 2, then  ${\mathcal J}(\lambda)^\angle={\mathcal J}(\lambda)$.
%\end{enumerate}
% \end{lemma}

Now let $\gamma$ be a segment of a regular abnormal extremal,
$O_\gamma$ be a neighborhood of $\gamma$ in $\mathbb P W_D$ such that the quotient
\begin{equation}
\label{abnormspace}
\mathcal N=O_\gamma/(\textrm{the charactrestic one-foliation of regular abnormal extremals})
\end{equation}
is a well defined smooth manifold.
%\pause
Further, let $\Phi:O_\gamma\rightarrow \mathcal N$ be the canonical projection to the quotient manifold. Since $\mathcal C$ is the Cauchy characteristic distribution of the even-contact distribution
$\widetilde\Delta$, then  the push-forward
%\pause
$\Delta:=\Phi_*\widetilde{\Delta}$  of $\widetilde \Delta$ by $\Phi$ is a contact distribution on $\mathcal N$. Each fiber of this contact distribution is endowed with the canonical symplectic form $\sigma$, defined up to up to multiplication by a non-zero constant, i.e. with the canonical conformal symplectic structure. As before, one can define the notion of skew-orthogonal complement and also of isotropic and coisotropic subspaces of the fibers of $\Delta$ with respect to this symplectic form. The subspaces of the fibers of $\Delta$ that are both isotropic and coisotropic (i.e. coincide with their skew-symmetric complement) are called \emph{Lagrangian}.

For any $\lambda\in\gamma$ define two subspaces of $\Delta(\gamma)$ as follows:

\begin{equation}
\label{Jacdef}
J(\lambda):=d\Phi \bigl(\mathcal J(\lambda)\bigr), \quad V(\lambda):=d\Phi \bigl(\mathcal V(\lambda)\bigr).
\end{equation}

Note that by constructions
\begin{equation}
\label{JVdiff}
\dim\, J(\lambda)-\dim\,V(\lambda)={\rm rank}\, D-1.
\end{equation}

Now we need to introduce some terminology. Let $\Lambda(\cdot)$ be a curve in a Grassmannian of subspaces of a vector space $L$. Let $t\mapsto\Lambda(t)$ be a parametrization of the curve $\Lambda(\cdot)$. Denote by $C(\Lambda)$ the \emph{tautological bundle over the curve $\Lambda(\cdot)$}: the fiber of $C(\Lambda)$ over the point $\Lambda(t)$ is the vector space $\Lambda(t)$.  Let $\Gamma(\Lambda)$ be the space of all smooth sections of $C(\Lambda)$ . Set
 %$\Lambda^{(i)}$ (or the $i$-th osculating space) and the $i$-th contraction $\Lambda_{(i)}$ of the curve $\Lambda(\cdot)$ as
 $\Lambda^{(0)}(t):=\Lambda(t)$ and  define inductively
$$\Lambda^{(j)}(t)=\Span\{\frac{d^k}{dt^k}\ell(t): \ell\in\Gamma(\Lambda), 0\leq k\leq j\}$$ for $j\geq0$.
The space $\Lambda^{(j)}(t)$ is called the \emph{$j$th extension} or the \emph{$j$th osculating subspace} of the  curve $\Lambda(\cdot)$ at the point $\Lambda(t)$. Obviously, this definition is independent of the parametrization of the curve $\Lambda(\cdot)$.

Recall that for any two smooth vector fields $X$ and $Y$ on a manifold, one has $[X,Y]=\cfrac{d}{dt} (e^{-tX})_* Y|_{t=0}$, where $e^{tX}$ denotes the flow generated by the field $X$. This together with \eqref{Jacdef} implies that

\begin{equation}
\label{Jacfirstext}
J^{(1)}(\lambda)=d\Phi \bigl([\mathcal C,\mathcal J](\lambda)\bigr), \quad V^{(1)}(\lambda):=d\Phi \bigl([\mathcal C, \mathcal V](\lambda)\bigr).
\end{equation}

>From relation \eqref{VC} and the last relation
%and the definitions of the distributions $\mathcal V$ and $\mathcal J$
it follows that

\begin{eqnarray}
&~& V^{(1)}(\lambda)\subset J(\lambda), \label{VC1}\\
&~& \dim\,J^{(1)}(\lambda)-\dim J(\lambda)\leq {\rm rank} D-1. \label{rankDineq}
\end{eqnarray}

As a direct consequence of Proposition \ref{isolem}, we get

 \begin{proposition}
 \label{Jskew2}
 Assume that $\lambda\in \gamma$. The following statements hold:
\begin{enumerate}
\item  $J(\lambda)^\angle\subseteq J(\lambda)$, i.e. the space $J(\lambda)$ is a coisotropic subspace of $\Delta(\gamma)$ with respect to the canonical conformal symplectic structure;
\item  If the distribution $D$ is of odd rank then
\begin{equation}
\label{Jskewodd}
J(\lambda)^\angle= V(\lambda)
\end{equation}

\item If the distribution $D$ is of even rank then
\begin{equation*}
%\label{Jskeweven}
\dim\Bigl(J(\lambda)^\angle\Bigr)=\dim\Bigl(V(\lambda)\Bigr)+1
\end{equation*}
and
\begin{equation}
\label{CJ^angle1}
V(\lambda)=\left(J(\lambda)+\Bigl(\bigl(J(\lambda)\bigr)^\angle\Bigr)^{(1)}\right)^\angle.
\end{equation}
\item If the distribution $D$ is of rank 2, then  $J(\lambda)^\angle= J(\lambda)$, i.e  the space $J(\lambda)$ is a Lagrangian subspace of $\Delta(\gamma)$ with respect to the canonical conformal symplectic structure.
\end{enumerate}
 \end{proposition}

%$\forall \lambda\in\gamma\quad F_\gamma(\lambda):=\underbrace{\Phi_*(J(\lambda))} coisotropic subspace \subset\Delta(\gamma)$
%\pause

The curve of flags $\lambda \mapsto J(\lambda)$, $\lambda\in \gamma$, %,\,\,\lambda\in\gamma$,
of the space $\Delta(\gamma)$
%is a curve of coisotropic subspaces of $\Delta(\gamma)\subset T_\gamma N$,
is called  the \emph{(non-extended) Jacobi curve of the regular abnormal extremals $\gamma$}. Since the construction of the Jacobi curves is intrinsic, any invariant of the Jacobi curve  with respect to the natural  action on the isotropic  and coisotropic subspaces of $\Delta(\gamma)$  of the conformal symplectic group $CSp\bigl(\Delta(\gamma)\bigr)$ produces an invariant of the distribution $D$.

\begin{remark}
\label{recover}
{\rm Note that from item (2) of Proposition \ref{Jskew2}, if ${\rm rank}\,\, D$ is odd, and formula \eqref{CJ^angle1} of item (3) of Proposition \ref{Jskew2}, if ${\rm rank}\,\, D$ is even, it follows that the curve $\lambda\mapsto V(\lambda)$, $\lambda\in \gamma$ can be recovered from the Jacobi curve
$\lambda \mapsto J(\lambda)$, $\lambda\in \gamma$.
%This implies in particular
%then the curve $\lambda\mapsto V(\lambda)$, $\lambda\in \gamma$ can be recovered from the curve
%$\lambda \mapsto J(\lambda)$, $\lambda\in \gamma$. In other words, the whole Jacobi curve of $\gamma$ can be recovered from the curve
%$\lambda \mapsto J(\lambda)$, $\lambda\in \gamma$.  It can be shown  (see \cite{doubzel2} or \eqref{??}) that  the same is true, if ${\rm rank}\,\, D=2$. This is the reason why in our previous papers, mainly devoted  to these cases, by the Jacobi curve of $\gamma$ we meant the curve $\lambda \mapsto J(\lambda)$, $\lambda\in \gamma$. However, if  ${\rm rank}\,\, D$ is even and greater than $2$ the curve $\lambda\mapsto V(\lambda)$ cannot be recovered from the curve $\lambda\mapsto J(\lambda)$, $\lambda\in \gamma$ and we need to keep the information about the curve $\lambda\mapsto V(\lambda)$ as a part of the definition of the Jacobi curve as well.}
%%assuming also that $\rank \,\, D^2=3$
}
\end{remark}

Our next goal is to generate more curves of isotropic and coisotropic spaces from the Jacobi curves using recursively the following two basic operations with curves $\lambda\mapsto V(\lambda)$ and  $\lambda\mapsto J(\lambda)$, $\lambda\in \gamma$: the osculations and taking the skew-orthogonal complement. For this let $\frac{1}{2}\mathbb Z:=\{\frac{i}{2}: i\in \mathbb Z\}$ and $\frac{1}{2}\mathbb Z_{\rm{odd}}:=
(\frac{1}{2}\mathbb Z)\backslash \mathbb Z$.
%=\{i+\frac{1}{2}: i\in \mathbb Z\}$.
Depending on the rank of a distribution $D$, to the Jacobi curve $\lambda \mapsto \{V(\lambda)\subset J(\lambda)\}$, $\lambda\in \gamma$, of a regular abnormal extremal $\gamma$ we assign the curve of monotonically non-increasing (by inclusion) flags
$\lambda \mapsto \{J^i(\lambda)\}_{i\in I}$, $\lambda\in \gamma$, where

\begin{equation}
\label{indexset}
I=\begin{cases}\mathbb Z,& \text{if } {\rm rank}\, D \text { is odd},\\
\frac{1}{2}\mathbb Z_{\rm{odd}}, & \text{if } {\rm rank}\, D=2,\\
\frac{1}{2}\mathbb Z, & \text{if }{\rm rank}\, D \text{ is  even and greater than } 2 .
\end{cases}
\end{equation}

  This curve of flags is called the \emph{extended Jacobi curve of the regular abnormal extremal $\gamma$}. The construction of this curve is slightly different in each of the aforementioned cases. Therefore, we describe it separately in each
  %of the following three
  these cases:

\begin{enumerate}
\item
{\bf $D$ has odd rank}. In this case  the extended Jacobi curve of a regular abnormal extremal $\gamma$ is the curve of flags $\lambda \mapsto \{J^i(\lambda)\}_{i\in \mathbb Z}$, $\lambda\in \gamma$, defined as follows: First, we set $J^0(\lambda):=J(\lambda)$. Second, let $J^{-i}(\lambda):=J^{(i)}(\lambda)$ for integer $i>0$ (i.e. the $i$th osculating subspace of the curve $\lambda\mapsto J(\lambda)$, $\lambda\in \gamma$ at the point $\lambda$). Finally, assume that
$J^{i}(\lambda):=\bigl(J^{(1-i)}(\lambda)\bigr)^\angle$ for integer $i>0$.

By construction, $J^i(\lambda)\subset J^{i-1}(\lambda)$ for all $i\in\mathbb Z$, subspaces $J^i(\lambda)$ are coisotropic for integer $i\leq 0$, and  isotropic for integer $i>0$. Besides, using \eqref{VC1} and the definitions of $J^i(\lambda)$ it is easy to see that

\begin{equation}
\label{compatdiff}
(J^{i})^{(1)}\subset  J^{i-1}
\end{equation}
for all $i\in\mathbb Z$. Finally by \eqref{Jskewodd} we have $V=J^{1}$.
\item
 {\bf $D$ is a rank $2$ distribution}. In this case the extended Jacobi curve of a regular  abnormal extremal $\gamma$ is the curve of flags $\lambda \mapsto \{J^i(\lambda)\}_{i\in \frac{1}{2}\mathbb Z_{\rm{odd}}}$ $\lambda\in \gamma$, defined as follows: First, we set $J^{\frac{1}{2}}(\lambda):=J(\lambda)$. Second, let $J^{\frac{1}{2}-i}(\lambda):=J^{(i)}(\lambda)$ for integer $i>0$.
%(i.e. the $i$th osculating subspace of the curve $\lambda\mapsto J(\lambda)$, $\lambda\in \gamma$ at the point $\lambda$).
Finally, assume that
$J^{\frac{1}{2}+i}(\lambda):=\bigl(J^{\frac{1}{2}-i}(\lambda)\bigr)^\angle$ for integer $i>0$.

By construction,  $J^i(\lambda)\subset J^{i-1}(\lambda)$ for all $i\in\mathbb Z$, subspaces $J^{\frac{1}{2}+i}(\lambda)$ are coisotropic for integer $i\leq 0$, and isotropic for integer $i\geq 0$. In particular, item (4) of Proposition \ref{Jskew2} says that the subspaces $J^{\frac{1}{2}}(\lambda)$ are Lagrangian. Besides, relation \eqref{compatdiff} holds for all $i\in \frac{1}{2}\mathbb Z_{\rm{odd}}$.

Finally, from \eqref{pullD2} and \eqref{Wdrank21} it follows easily that

\begin{equation}
\label{Vrecov2}
V=(J^{-\frac{1}{2}})^{\angle}=J^{\frac{3}{2}}.
\end{equation}
\item  {\bf  $D$ has an even rank greater than $2$}. The extended Jacobi curve of a regular abnormal extremal $\gamma$ is the curve of flags $\lambda \mapsto \{J^i(\lambda)\}_{i\in \frac{1}{2}\mathbb Z}$, $\lambda\in \gamma$, defined as follows:
     First, we set $J^{0}(\lambda):=J(\lambda)$ and $J^{\frac{1}{2}}(\lambda):=J(\lambda)^\angle$. Second, let $$J^{-\frac{1}{2}}=\bigl(J^{\frac{1}{2}}\bigr)^{(1)}+J^0.$$ Further, define recursively $J^{i-1}(\lambda):=(J^{i})^{(1)})(\lambda)$ for all negative $i\in \frac{1}{2}\mathbb Z$. Finally, set $J^i(\lambda):=\bigl(J^{\frac{1}{2}-i}(\lambda)\bigr)^\angle$ for all positive $i\in \frac{1}{2}\mathbb Z$.
%(i.e. the $i$th osculating subspace of the curve $\lambda\mapsto J(\lambda)$, $\lambda\in \gamma$ at the point $\lambda$).

By construction,  $J^i(\lambda)\subset J^{i-\frac{1}{2}}(\lambda)$ for all $i\in \frac{1}{2}\mathbb Z$, subspaces $J^{i}(\lambda)$ are coisotropic for all nonnegative  $i\in \frac{1}{2}\mathbb Z$ and isotropic for all positive $i\in \frac{1}{2}\mathbb Z$. Besides, relation \eqref{compatdiff} holds for all $i\in \frac{1}{2}\mathbb Z$.
\end{enumerate}

\begin{remark}
\label{evendiffrem}
Note that, in contrast to the  first two cases, in the third case the first osculation of each curve $\lambda \mapsto J^i(\lambda)$  does not belong in general to the next larger subspace in the flag but to the second next larger subspace in the flag. In particular, it does not satisfy the so-called compatibility with respect to differentiation assumption of \cite{flag2, quasi}. However, these case will be treated similarly to the first two.
\end{remark}

We say that a point $\lambda_0\in \gamma$ is a \emph{regular point of the extended Jacobi curve $\lambda \mapsto \{J^i(\lambda)\}_{i\in I}$, $\lambda\in \gamma$} if for every $i\in I$ subspaces
$J^i(\lambda)$ have the same dimension in a neighborhood of $\lambda_0$. The set of regular points is open and dense in $\gamma$. Also, if $\lambda$ is a regular point, then
\begin{equation}
\label{jumps}
\dim\, J^{i-2}(\lambda)-\dim\, J^{i-1}(\lambda)\leq \dim\,J^{i-1}(\lambda)-\dim \, J^i(\lambda), \quad \forall\, i\in I,\,i\geq 0.
\end{equation}
The last inequality together with \eqref{JVdiff} and \eqref{rankDineq} imply that
\begin{equation}
\label{jumpsrankD}
\dim\,J^{i-1}(\lambda)-\dim \, J^i(\lambda)\leq {\rm rank}\,D-1, \quad \forall\, i\in I.
\end{equation}

Finally, we say that a point $\lambda_0 \in \gamma$ is a \emph{point of maximal class of the extended Jacobi curve $\lambda \mapsto \{J^i(\lambda)\}_{i\in I}$, $\lambda\in \gamma$}, if there exists $\nu\in I$ such that $J^\nu(\lambda_0)=\Delta(\gamma)$. If $\lambda_0$ is not a point of maximal class then there exist $\nu\in I$ such that  $J^{\nu'}(\lambda_0)=J^{\nu}(\lambda_0)\subsetneq\Delta(\gamma)$ for all $\nu^\prime\in I$ such that $\nu'\leq \nu$. If the point $\lambda_0$ is regular and $\nu\in I$ is the maximal index with the property above, then the subspace $J^{\nu}(\lambda)$ is independent of $\lambda$ in a neighborhood $\tilde\gamma$ of $\lambda_0$ in $\gamma$. Hence, the subspace $J^{\nu}(\lambda_0)^\angle$ is the (maximal) common subspace for all spaces $J^{i}(\lambda)$, $i\in I$, $\lambda\in \tilde\gamma$ and the whole differential-geometric information on the extended Jacobi curve
$\lambda \mapsto \{J^i(\lambda)\}_{i\in I}$,  $\lambda\in \tilde\gamma$, is contained in the curve of flags
$\lambda \mapsto \{J^i(\lambda)/ J^{\nu}(\lambda_0)^\angle\}_{i\in I}$, $\lambda\in \tilde\gamma)$ in the conformal symplectic space $J^{\nu}(\lambda_0)/ J^{\nu}(\lambda_0)^\angle$.

\section{ Flag symbols of Jacobi curves and Jacobi symbols of distributions}
\label{flagsec}
\setcounter{equation}{0}
\setcounter{theorem}{0}
\setcounter{lemma}{0}
\setcounter{proposition}{0}
\setcounter{definition}{0}
\setcounter{corollary}{0}
\setcounter{remark}{0}
\setcounter{example}{0}

\subsection{Flag symbols of Jacobi curves}
\label{flagsec1}
Now we are ready to define the flag symbol of the (extended) Jacobi curve
%$\lambda \mapsto \{J^i(\lambda)\}_{i\in \frac{1}{2}\mathbb Z_{\rm{odd}}}$
at a point.
%\lambda$.
Informally speaking, the flag symbol at a point represents the tangent line to this curve at this point. To begin with the precise definition, recall that the velocity of a parametrized curve  $t\mapsto\Lambda(t)$ in a Grassmannian of subspaces of a vector space $L$ at $t=t_0$ can be naturally identified with an element of ${\rm Hom}(\Lambda(t_0), \Lambda^{(1)}(t_0)/\Lambda(t_0))$. For this, take $l\in\Lambda(t_0)$ and consider any smooth section $\ell(t)$ of the tautilogical bundle $C(\Lambda)$ such that $\ell(t_0)=l$. Then the assignment $$l\mapsto \cfrac{d}{dt} \ell(t_0)\,\,{\rm mod}\,\Lambda(t_0)$$ is independent of a choice of such section and defines the desired element of ${\rm Hom}(\Lambda(t_0), \Lambda^{(1)}(t_0)/\Lambda(t_0))$.

Using this construction and relation \eqref{compatdiff} one can identify the tangent line to the
curve $\lambda \mapsto \{J^i(\lambda)\}_{i\in I}$, $\lambda\in \gamma$, at $\lambda$ with the line in the space ${\rm Hom}(J^{i}(\lambda), J^{i-1}(\lambda)/J^{i}(\lambda))$. Moreover, again from \eqref{compatdiff} this line factors through the line in the space  ${\rm Hom}(J^{i}(\lambda)/J^{i+1}(\lambda), J^{i-1}(\lambda)/J^{i}(\lambda))$.

Therefore the tangent line to the
(extended) Jacobi curve $\lambda \mapsto \{J^i(\lambda)\}_{i\in I}$, $\lambda\in \gamma$ at $\lambda$ can be identified with a line in the space

$$\bigoplus_{i\in I}{\rm Hom}(J^{i}(\lambda)/J^{i+1}(\lambda), J^{i-1}(\lambda)/J^{i}(\lambda)).$$

Now consider the following two cases separately (the reason for the difference between these cases is explained in Remark \ref{evendiffrem} above):

\begin{enumerate}
\item
 %The case of
 {\bf ${\rm rank}\, D$ is odd or equal to $2$}.
 %we can reformulate the latter statement as follows:
 In this case let ${\rm Gr}^i(\lambda):= J^{i}(\lambda)/J^{i+1}(\lambda)$, $i\in I$, and ${\rm Gr}(\lambda)=\displaystyle{\bigoplus_{i\in I} {\rm Gr}^i(\lambda)}$ be the graded space corresponding to the filtration $\{J^i(\lambda)\}_{i\in I}$. Here the set $I$ is as in \eqref{indexset}. Then the tangent line to the
(extended) Jacobi curve $\lambda \mapsto \{J^i(\lambda)\}_{i\in I}$, $\lambda\in \gamma$ at $\lambda$ can be identified with a line in the space of degree $-1$ endomorphism of the graded space $\rm{Gr}(\lambda)$, which will be denoted by $\delta_\lambda$.

\item
{\bf $D$ has an even rank greater than $2$}. In this case let  ${\rm Gr}^i(\lambda):= J^{i}(\lambda)/J^{i+\frac{1}{2}}(\lambda)$, $i\in \cfrac{1}{2}\mathbb Z$, and ${\rm Gr}(\lambda)=\displaystyle{\bigoplus_{i\in  \frac{1}{2}\mathbb Z} {\rm Gr}^i(\lambda)}$ be the graded space corresponding to the filtration $\{J^i(\lambda)\}_{i\in  \frac{1}{2}\mathbb Z}$. By relation \eqref{compatdiff} for any $i\in \frac{1}{2}\mathbb Z$ the line in the space ${\rm Hom}(J^{i}(\lambda)/J^{i+1}(\lambda) , J^{i-1}(\lambda)/J^{i}(\lambda))$, representing the tangent line to the
curve $\lambda \mapsto \{J^i(\lambda)\}_{i\in I}$, $\lambda\in \gamma$, at $\lambda$, factors through the line in the space ${\rm Hom}({\rm Gr}^i(\lambda), {\rm Gr}^{i-1}(\lambda))$, which will be denoted by $\delta^i_\lambda$. Note that passing from the line in ${\rm Hom}(J^{i}(\lambda)/J^{i+1}(\lambda) , J^{i-1}(\lambda)/J^{i}(\lambda))$ to the line
%in ${\rm Hom}({\rm Gr}^i(\lambda), {\rm Gr}^{i-1}(\lambda))$
$\delta^i_\lambda$, some information about the tangent line to the curve $\lambda \mapsto \{J^i(\lambda)\}_{i\in I}$, $\lambda\in \gamma$, at $\lambda$ is lost, however it is not important for our purposes. Let $\delta_\lambda$ be the degree $-1$ endomorphism of the graded spave ${\rm Gr}(\lambda)$ such that $\delta_\lambda|_{{\rm Gr}^i(\lambda)}=\delta^i_\lambda$.
%Therefore
%let ${\rm Gr}(\lambda)=\displaystyle{\bigoplus_{i\in \frac{1}{2} \mathbb Z} J^{i}(\lambda)/J^{i+\frac{1}{2}}(\lambda)}$,
\end{enumerate}

In both cases the degree $-1$ endomorphism $\delta_\lambda$ of the corresponding graded space ${\rm Gr}(\lambda)$ is defined.
Further note that the conformal symplectic structure  on $\Delta(\gamma)$, represented by the symplectic form $\sigma$, induces the conformal symplectic structure on the space $\rm{Gr}(\lambda)$ represented by the following symplectic form $\bar\sigma$: if $\bar x\in {\rm Gr}^i(\lambda)$ and $\bar y\in {\rm Gr}^j(\lambda)$ with $i+j=0$, then $\bar\sigma(\bar x,\bar y)=\sigma(x, y)$, where $x$ and $y$ are representatives of $\bar x$ and $\bar y$ in $J^i(\lambda)$ and $ J^j(\lambda)$ respectively; if $i+j\neq 0$, then  $\bar\sigma(\bar x,\bar y)=0$.

By construction the line of endomorphisms $\delta_\lambda $ belongs to conformal symplectic algebra $\mathfrak{csp}\bigl({\rm Gr}(\lambda)\bigr)$ (with respect to the conformal symplectic structure generated by the form $\bar \sigma$). The group of conformal symplectic endomorphisms of ${\rm Gr}(\lambda)$, preserving the grading, acts on the lines of degree $-1$ endomorphisms  by the conjugation. The orbit of the line $\delta_\lambda$ is called the \emph{symplectic flag symbol of the (extended) Jacobi curve
$\lambda \mapsto \{J^i(\lambda)\}_{i\in I}$, $\lambda\in \gamma$, at $\lambda$}.

Two flag symbols act in general on different graded spaces endowed with conformal symplectic structures. They are called \emph{isomorphic}, if they are conjugated by a conformal symplectic isomorphism  between these spaces (preserving the grading). Therefore, we can fix a model graded space $X=\displaystyle{\bigoplus_{i\in I}X^i}$ with conformal symplectic structure,  where , as before, $I=\mathbb Z,\,\frac{1}{2}\mathbb Z_{\rm odd}$, or $\frac{1}{2}\mathbb Z$, such that the symplectic product of elements of $X^i$ and $X^j$ with $i+j\neq 0$ are equal to zero. The graded space $X$ with these properties will be called the \emph{graded conformal symplectic space}. Further, one can look on the flag symbol as an orbit of a line of degree $-1$ endomorphisms from $\mathfrak{csp}(X)$ with respect to the action of the subgroup of degree $0$ endomorphisms from $\mathfrak{csp}(X)$  by the conjugation. Such object will be called a \emph{symplectic flag symbol}. In the sequel we usually refer to the flag symbol as to a line of degree $-1$ endomorphisms from $\mathfrak{csp}(X)$ representing this orbit.

\subsection{Classification of flag symbols of Jacobi curves}
\label{flagsec2}

Symplectic flag symbols can be easily classified. This classification is a little bit finer than the Jordan classification of nilpotent endomorphisms of a vector space. One has to take into account the presence of the grading and of the symplectic structure. When $I=\mathbb Z$ or $\frac{1}{2}\mathbb Z_{\rm odd}$, all symplectic flag symbols were classified in \cite[subsection 7.2]{flag2} and this classification includes as particular case all possible flag symbols of Jacobi curves  in the case of distribution of odd rank or rank $2$. The classification of  flag symbols of Jacobi curves in the case of distributions of even rank can be done similarly.
In all cases symplectic flag symbols can be described  by means of skew Young diagrams as described below.

Let $\delta_1$ and $\delta_2$ be symplectic flag symbols of the symplectic spaces $X_1$ and $X_2$, which are graded by sets $I_1$ and $I_2$ respectively (here in general we do not assume that $I_1=I_2$). The direct sum $X_1\oplus X_2$, graded by $I_1\cup I_2$, is equipped with  the natural symplectic grading such that its $i$th component is the direct sum of $i$th components of $X_1$ and $X_2$ and the symplectic form on  $X_1\oplus X_2$ is defined as follows: the restriction of this form to $X_i$ coincides with the symplectic form on $X_i$ for each $i=1,2$ and the spaces $X_1$ and $X_2$ are skew-orthogonal in $X_1\oplus X_2$.
The direct sum $\delta_1\oplus\delta_2$ is the degree $-1$ endomorphism of $X_1\oplus X_2$  such that the restriction of it to $X_i$ is equal to $\delta_i$ for each $i=1,2$.
 A symplectic flag symbol of a graded conformal symplectic space $X$ is called \emph{symplectically indecomposable} if
 it cannot be represented as a direct sum  of two symplectic flag symbols acting on (nonzero) graded conformal symplectic spaces.

 Any sympletic flag symbol is a direct sum of a finite number of symplectically indecomposable flag symbols.
 Therefore, in order to describe all symplectic flag symbols, it is enough to describe all symplectically indecomposable ones. Let us consider the same three cases as in the previous section separately:

\begin{enumerate}
\item
 In the case of $\mathbb{Z}$-graded conformal symplectic spaces (i.e., when $I=\mathbb Z$) symplectically indecomposable symbols can be described as follows \cite[subsection 7.2]{flag2}):
 Given a nonnegative integer $s$ and an integer $l$   such that $0\leq l\leq 2s$ let $X_{s,l}$ be a linear symplectic space with a basis $\{e_{s-l},\ldots, e_s,f_{-s},\ldots,f_{l-s}\}$,
\begin{equation}
\label{tuple1}
%\bigl\{\{e_i\}_{r\leq i\leq s}, \{f_i\}_{-s-\varepsilon\leq i\leq -r-\varepsilon}\bigr\}
X_{s,l}=\Span\{e_{s-l},\ldots, e_s,f_{-s},\ldots,f_{l-s}\},
\end{equation}
such that $\sigma(e_i,e_j)=\sigma(f_i,f_j)=0$, $\sigma(e_i,f_{-i})=(-1)^{s-i}$ and $\sigma(e_i,f_j)=0$ if  $j\neq -i$. We turn the space $X_{s,l}$ into the symplectic graded spaces by defining the subspaces of degree $i$ to be equal to the span of all vectors with index $i$ appearing in the tuple \eqref{tuple1}. Then denote by $\delta_
%_{rs}^{\varepsilon}
{s,l}$ a degree $-1$ endomorphism of
$X_
%_{rs}^{\varepsilon}
{s,l}$ from the symplectic algebra such that $\delta_{s,l}(e_i)=e_{i-1}$ for $s-l+1\leq i\leq s$, $\delta_{s,l}(e_{s-l})=0$, $\delta_{s,l}(f_i)=f_{i-1}$ for $-s+1\leq i\leq l-s$, and
$\delta_{s,l}(f_{-s})=0$. Then \emph {any indecomposable symplectic flag symbol of a $\mathbb Z$-graded conformal symplectic space is isomorphic to the symplectic flag symbol represented by the line  $\mathbb R \delta_{s,l}$ for a nonnegative integer $s$ and an integer $l$   such that $0\leq l\leq 2s$. Moreover, symplectic flag symbols corresponding to different pair $(s, l)$ are not isomorphic.}

It is convenient to visualize a symplectically indecomposable symbol on $\mathbb Z$-graded conformal symplectic spaces via the following skew Young diagrams, consisting of two rows, having central symmetry and with the most left box in the first row   being located either above or from the left to the most left box in the second row:

\begin{equation}
\label{youngodd}
\begin{ytableau}
~& \dots&~ & ~&\dots&~\\
\none&\none&\none&~&\dots&~& ~&\dots&~
\end{ytableau} \,\,\,\, or\,\,\,\,
\begin{ytableau}
~& \dots&~\\
\none&\none&\none&\none& \none&~&\dots&~
\end{ytableau}
%\young(\ \dots \ \ \dots \ ,:::\ \dots \ \ \dots \ )\,\,\, \text {  or  } \,\,\,\young(\ \dots %\ \ \dots \ ,::::::::::\ \dots \ \ \dots \ )
\end{equation}

The left diagram in \eqref{youngodd} corresponds to the case $s\leq l\leq 2s$. In this case the first $2s-l$ columns from the left consist of one box (situated at the top row), the next $2(l-s)+1$ columns consist of two boxes, and the last $2s-l$ columns consist of one box (situated at the bottom row). The right diagram in \eqref{youngodd} corresponds to the case $s>l$. In this case the first $l+1$ columns from the left consist of one box (situated at the top row), next $2(s-l)-1$ columns are just void, and the last $l+1$ columns consist of one box (situated at the bottom row).

\begin{remark}
\label{middleodd}
Note that in both cases the middle part (of columns with two or no boxes) consist of odd columns.
\end{remark}

In order to assign to a diagram of type \eqref{youngodd} the symplectic flag symbol one fills out the boxes of the diagram by vectors of the bases \eqref{tuple1} such that the top row is filled  from the left to the right with $e_s,\ldots e_{s-1},\ldots, e_{s-l}$  and the bottom row is filled  from the left to the right with $f_{l-s},f_{l-s-1},\ldots, f_{-s}$. For example, in the case $s\leq l$ the diagram is filled as follows:

\begin{equation}
\label{youngodd1}
\ytableausetup
{boxsize=2.7em}
\begin{ytableau}
e_s& \dots & e_{l-s+1} & e_{l-s}&\dots&e_{s-l}\\
\none&\none&\none&f_{l-s}&\dots&f_{s-l}& f_{s-l-1}&\dots&f_{-s}
\end{ytableau}
\end{equation}

 Then the endomorphism $\delta_{s,l}$ corresponds to the \emph{right shift} in the obtained skew Young tableaux \eqref{youngodd1} i.e. a vector in the basis \eqref{tuple1} is sent either to the vector located in the box to the right to it in the tableaux or to zero if there is no such box. Also note that the skew Young diagram \eqref{youngodd} contains all information on the grading of the corresponding vector spaces: the subspaces of the same degree are spanned by the vectors situated in the same column. Finally, the (conformal) symplectic structure is encoded by the skew
 Young diagram  as well: the symplectic product of the vectors situated in the boxes which are not centrally symmetric is equal to zero; it is equal to $\pm1$ for the vectors situated  in the boxes which are not centrally symmetric such that it is equal to $1$ for two vectors situated in the most left and the most right boxes of the diagram, i.e.  vectors $e_s$ and $f_{-s}$), and it alternates signs when we took symmetric pairs by moving with the same speed from the those boxes toward the middle of the diagram.

Now let us discuss what symplectic flag symbols may appear at regular points of Jacobi curves of a distribution $D$ of odd rank. From the relation \eqref{JVdiff}, the fact that the subspaces $J^i(\lambda)$ are obtained by osculations of the curve $\lambda\mapsto J(\lambda)$, regularity, and the inequalities \eqref{jumpsrankD} it follows that \emph{a symplectic flag symbol  at a regular point of Jacobi curves of a distribution $D$ of odd rank is isomorphic to a direct sum of $\frac{1}{2} ({\rm rank} D-1)$ of indecomposable flag symbols with $s\leq l\leq 2s$, i.e. given by skew Young diagrams as in the left of \eqref{youngodd}.}

In particular, for rank $3$ distributions  a symplectic flag symbol at a regular point $\lambda$  of a Jacobi curve is isomorphic to an  indecomposable one given by a skew Young diagrams as in the left of \eqref{youngodd}, i.e. $\delta_{s,l}$ with $s\leq l\leq 2s$. If, in addition, $\lambda$ is a point of maximal class and the dimension of the ambient manifold $M$ is equal to $n$,  then the number of boxes in this skew diagram must be equal to the rank of the contact distribution $\Delta$, i.e to $2n-6$. On the other hand, if the symbol is represented by the endomorphism $\delta_{s,l}$ then the number of box in the skew diagram is equal to $2l+2$. Consequently, in the considered case $l=n-4$ and $\frac{n-4}{2}\leq s\leq n-4$.

\begin{remark}
\label{6dim}
Note that a rank 3 distribution $D$  has the Jacobi symbol $\delta_{s,l}$ with $s< l\leq 2s$ if and only if its square $D^2(q)$ is $6$-dimensional at a generic point $q$, $\dim D^2(q)=6$. It easily follows from the expression \eqref{C3} for the characteristic rank 1 distribution. Also from this one can show that $\rank D^2=4$ if and only if  the Jacobi symbol of $D$ is equal to $\mathbb R\delta_{0,0}$, i.e. its Young diagram consists of one column with two boxes and $\rank D^2=5$ if and only if  the Jacobi symbol of $D$ is equal top $\mathbb R\delta_{s,s}$ with $s>0$.
\end{remark}

\item
In the case of $\frac{1}{2}\mathbb{Z}_{\rm odd}$-graded conformal symplectic spaces (i.e., when $I=\frac{1}{2}\mathbb{Z}_{\rm odd}$) there are two types of symplectically indecomposable symbols.
First type is represented by an endomorphism $\delta_{s,l}$ as in the previous item with $s$ being positive and belonging to $\frac{1}{2}\mathbb{Z}_{\rm odd}$ and $l$ being integer and satisfying $0\leq l\leq 2s-1$. It is encoded again by a skew Young diagram consisting of two rows as before.

Second type can be described as follows: Given a positive $m\in \frac{1}{2} Z_{\rm {odd}}$ let
\begin{equation}
\label{Lm}
\mathcal L_m=\displaystyle{\bigoplus_{\tiny{\begin{array}{c}-m\leq i\leq m,\\i\in\frac{1}{2}\mathbb Z_{\rm {odd}}\end{array}}}}E_i
\end{equation}
be a symplectic graded spaces such that $\dim E_i=1$ for every $-m\leq i\leq m$ and let $\tau_m$ be a degree $-1$ endomorphism of $\mathcal L_m$  which sends $E_i$ onto $E_{i-1}$ for every admissible $-m\leq i< m$ and  $\tau_m(E_{-m})=0$.  An indecomposable symplectic flag symbol of second type on  a $\frac{1}{2}\mathbb{Z}_{\rm odd}$-graded conformal symplectic space is isomorphic to the symplectic flag symbol represented by the line  $\mathbb R $ for a positive $m\in \frac{1}{2}\mathbb Z_{\rm {odd}}$.

It is convenient to visualize the symplectically indecomposable symbol, which is  isomorphic to the symbol represented by the endomorphism $\tau_m$,  via the Young diagrams consisting of one row with $2m+1$ boxes (note that $2m+1$ is even, because  $m\in \frac{1}{2}\mathbb Z_{\rm {odd}}$). By analogy with the previous case the endomorphism $\tau_m$   corresponds to the right shift in this diagram where the boxes from the left to the right are filled by vectors generated lines $E_i$, i.e.
\begin{equation}
\label{youngodd2}
\ytableausetup
{boxsize=2.5em}
\begin{ytableau}
\varepsilon_m& \varepsilon _{m-1}&\dots &\varepsilon_{1-m}& \varepsilon_{-m}
\end{ytableau}\,\,,  \quad \text{ where } E_i=\Span\{\varepsilon_i\}
\end{equation}

By the same arguments as in the previous item for rank $2$ distributions a symplectic flag symbol  at a regular point of its Jacobi curves is isomorphic to the indecomposable symbol, represented by the endomorphism $\tau_m$ for some positive $m\in \frac{1}{2}\mathbb Z_{\rm {odd}}$.  If, in addition, $\lambda$ is a point of maximal class  and the dimension of the ambient manifold $M$ is equal to $n$,  then the number of boxes in the corresponding skew diagram must be equal to the rank of the contact distribution $\Delta$, i.e. to $2n-6$. In this case all symplectic flag symbol at regular points are isomorphic to the symbol represented by the endomorphism $\tau_{n-\frac{7}{2}}$.

\item Now consider the case of $\frac{1}{2}\mathbb Z$-graded  symplectic space $X$. Obviously, this space is the direct sum of  its $\mathbb Z$-graded component $X_\mathbb Z$ and its $\frac{1}{2}\mathbb Z_{\rm {odd}}$-graded component $X_{\frac{1}{2}\mathbb Z_{\rm {odd}}}$. Therefore, any symplectic flag symbol $\delta$ on $V$ is the direct sum of a symplectic flag symbol on $\delta_{\mathbb Z}$ on $X_\mathbb Z$ and of a symplectic flag symbol $\delta_{\frac{1}{2}\mathbb Z_{\rm {odd}}}$ on  $X_{\frac{1}{2}\mathbb Z_{\rm {odd}}}$,
    %-graded component of this symplectic space,
    i.e the classification of symplectic flag symbols in the considered case is reduced to the classification of symplectic flag symbols described in the previous two items. The flag symbols  $\delta_{\mathbb Z}$ and $\delta_{\frac{1}{2}\mathbb Z_{\rm {odd}}}$ will call
    the $\mathbb Z$-graded and $\frac{1}{2}\mathbb Z_{\rm {odd}}$-graded component of the flag symbol $\delta$, respectively.

Now let us discuss what symplectic flag symbols may appear at regular points of Jacobi curves of a distribution $D$ of even rank greater than $2$. From item (3) of Proposition \ref{Jskew2} it follows that $\dim\,{\rm Gr}^{\frac{1}{2}}(\lambda)=1$ (where ${\rm Gr}^{\frac{1}{2}}(\lambda)=J^{\frac{1}{2}}(\lambda)/ J^{1}(\lambda)$). This implies that the $\frac{1}{2}\mathbb Z_{\rm {odd}}$-graded component of a  symplectic flag symbol  at regular points of its Jacobi curves is isomorphic to the indecomposable symbol, represented by the endomorphism $\tau_m$ for some positive $m\in \frac{1}{2}\mathbb Z_{\rm {odd}}$. Besides, by the same arguments as in item (1) the $\mathbb Z$-graded component of  symplectic flag symbol  at a regular point of Jacobi curves is isomorphic to a direct sum of $\frac{1}{2} ({\rm rank} D-2)$ of indecomposable flag symbols with $s\leq l\leq 2s$, i.e. given by skew Young diagrams as in the left of \eqref{youngodd}. In other words, in the considered case the symplectic flag symbol at a regular point of Jacobi curves of a distribution $D$ of even rank greater than $2$ is the direct sum of a symplectic flag symbol appearing at regular points of  of Jacobi curves of distributions of rank equal to ${\rm rank}\, D-1$ and of a symplectic flag symbol  appearing at regular points of Jacobi curves of rank $2$ distributions.

%In the case $\mathbb K=\mathbb R$ we also assume that $\sigma(\tau_m^{\mathfrak{sp}} (e), e)\geq 0$ for all $e\in E_{\frac{1}{2}}$.

  %are encoded by skew Young diagrams satisfying the following properties (see  also the figure below):
% \begin{enumerate}
% \item They consist of 2 rows;
% \item They have central symmetry;
% \item The first column (from the left) either consist of 2 boxes (equivalently, the diagram is rectangular) or of one box sitting in the upper row;
% \item The parity of the number of the columns from the left
%
% \end{enumerate}
%
% $$\young(\ \dots \ \ \dots \ ,:::\ \dots \ \ \dots \ )\,\,\, \text {  or  } \,\,\,\young(\ \dots \ \ \dots \ ,::::::::::\ \dots \ \ \dots \ ) $$
% It consist of two rows and have a central symmetry:

\end{enumerate}

\begin{remark}
As a matter of fact all three cases considered separately can be unified in item (3): the case of item (1) corresponds to $X_{\frac{1}{2}\mathbb Z_{\rm {odd}}}=0$ and the case of item (2) corresponds to $X_{\mathbb Z}=0$.
\end{remark}

\begin{remark}
\label{Youngrem}
To an arbitrary symplectic symbol $\delta$ we can assign the \emph{skew Young diagram $\Diag(\delta)$  of $\delta$} by attaching the skew Young diagrams of its indecomposable components such that the boxes corresponding to elements of the same degree are located in one column.
Obviously, $\Diag(\delta)$ is defined up to a permutation of sub-diagrams corresponding to indecomposable components. Note that by our convention  in the $\Diag(\mathbb R\delta_{s,l})$ the most left box of the upper row is located from the left or above of the most left box of the bottom row.
Filling the boxes of the skew Young diagrams by vectors of the model graded symplectic space $X$ as in \eqref{youngodd1} and \eqref{youngodd2} (for each undecomposable component) one gets a \emph{skew Young tableaux} $\Tab(\delta)$ of $\delta$. In this way to any row of $\Tab(\delta)$ we can assign the invariant subspace of $\delta$ in $X$, spanned by the vectors appearing in the boxes of this row.
\end{remark}

\subsection{Jacobi symbols of distributions}
\label{flagsec3}

In all three cases above the set of symplectic flag symbols is discrete. Moreover, by construction the  symplectic flag symbol  $\delta_\lambda$ of the extended Jacobi curve
$\lambda \mapsto \{J^i(\lambda)\}_{i\in I}$, $\lambda\in \gamma$ is determined algebraically by the jets of the (non-extended) Jacobi curve $\lambda \mapsto \{V(\lambda)\subset J(\lambda)\}$, $\lambda\in \gamma$.
%For fixed rank \tb{$D$} and \tb{$\dim M$} {\bf the set of all possible symbols of Jacobi curves, up to a conjugation, is finite}.
This implies that for any point $q\in M$ there is a Zariski open subset of the fiber of the bundle $\mathbb P W_D(q)\rightarrow M$ over the point $q$ such that the symplectic flag symbols of the extended Jacobi curves at all points of this subset are isomorphic to one symplectic flag symbol $\delta$. In this case we will say that $\delta$ is the \emph{Jacobi symbol} of a distribution $D$ at the point $q$.
%+ classification of symplectic symbols
%\pine
%\centerline{$\Downarrow$}\pause

Moreover, in contrast to Tanaka symbols, \emph{Jacobi symbols of distributions are always locally constant, i.e. in a neighborhood of a generic point they are isomorphic one to each other.}
This leads to the following new formulation: \emph{to construct canonical frames for distributions
%according to their Tanaka symbols to do it according to their
with given constant Jacobi symbol.}

Note that for a point $\lambda_0$ of nonmaximal class the symplectic flag symbol of the Jacobi curve at $\lambda_0$ consists of the endomorphisms
acting  on the graded conformal symplectic space ${\rm Gr}(\lambda_0)$  of dimension less than the rank of the contact distribution $\Delta$.
Thereby, it will be more convenient in the sequel to modify the definition of the Jacobi symbol as follows. Let $\nu\in I$
be the maximal index for which the osculating subspaces of the Jacobi curve stabilize, i.e. such that $J^{\nu^{\,\prime}}(\lambda_0)=J^{\nu}(\lambda_0)\subsetneq\Delta(\gamma)$ for all $\nu^{\,\prime}\in I$ such that $\nu^{\,\prime}\leq \nu$. This implies that ${\rm Gr}^{\nu^{\,\prime}}(\lambda_0)=0$ for all $\nu^{\,\prime}<\nu$ and $\nu^{\,\prime}\geq \varepsilon-\nu$, where $\varepsilon=1$ if $D$ has odd rank and $\varepsilon=\frac{1}{2}$ if $D$ has even rank.
In other words, ${\rm Gr}(\lambda_0)=\displaystyle{\bigoplus_{i\in I, \nu\leq i<\varepsilon-\nu} {\rm Gr}^i(\lambda_0)}$. Assume that $\lambda_0$ belongs to the  regular abnormal extremal $\gamma$ and define the graded conformal symplectic space

$$\widetilde{{\rm Gr}}(\lambda_0)=\Delta(\gamma)/J^\nu(\lambda_0)\oplus {\rm Gr}(\lambda_0)
%\Bigl(\bigoplus_{i\in I, \nu\leq i<s-\nu} {\rm Gr}^i(\lambda_0)}\Bigr)
\oplus J^\nu(\lambda_0)^\angle$$
such that $\Delta(\gamma)/J^\nu(\lambda_0)$ has weight preceding $\nu$ in $I$ and  $J^\nu(\lambda_0)^\angle$ has the weight $\varepsilon-\nu$. The \emph{modified symplectic flag symbol} of the extended Jacobi curve at $\lambda_0$ is represented by the line of endomorphisms $\delta_{\rm{mod}}$ from $\mathfrak{csp}\Bigl(\widetilde{{\rm Gr}}(\lambda_0)\Bigr)$ such that the restrictions of these endomorphisms to  ${\rm Gr}(\lambda_0)$ form the line of endomorphisms
of $\mathfrak{csp}\Bigl({\rm Gr}(\lambda_0)\Bigr)$, representing the symplectic flag symbol  of the extended Jacobi curve at $\lambda_0$  and the restriction of the endomorphisms $\delta_{\rm mod}$ to the added spaces $\Delta(\gamma)/J^\nu(\lambda_0)$ and $J^\nu(\lambda_0)^\angle$ is equal to zero. In other words the modified symplectic flag symbol is the direct sum of the  original flag symbol with $\dim  J^\nu(\lambda_0)^\angle$ copies of the symbol $\mathbb R\delta_{\varepsilon-\nu;0}$. Note that for the points of maximal class (and only for them) the modified symbol coincides with the original one. Finally note that in the same way one defines the modified Jacobi symbol of a distribution $D$ at a point.
In the sequel the Jacobi symbol of a distribution will be denoted by $\delta$ and the modified Jacobi symbol by $\delta_{\rm mod}$.

\subsection{Symplectically flat distributions with given Jacobi symbols}
\label{flagsec4}

We finish the section with an example of a distribution with given Jacobi symbol $\delta$.
For this fix  a skew Young tableaux $\Tab(\delta)$ and consider the subtableuax $\Tab_{-}(\delta)$ consisting of all columns with vectors of weight not greater than $\frac{1}{2}$. Consider the Lie algebra  $\mathfrak f_\delta$ spanned by the vectors in the boxes of $\Tab_{-}(\delta)$ and two additional vectors, denoted by $x$ and $z$, such that the only nonzero Lie brackets in the product table of the Lie algebra  $\mathfrak f_\delta$ with respect to the chosen basis  come from the following properties:
\begin{enumerate}
\item The operator ${\rm ad}\, x$ acts as the right shift in the tableaux   $\Tab_{-}(\delta)$;
\item For any subtableaux of $\Tab(\delta)$, corresponding to an indecomposable symbol represented by $\delta_{s,l}$ (if exists),  the Lie brackets of two vectors situated in the column corresponding to weight $0$ is equal to $\pm z$;
\item   For a subtableaux of  $\Tab(\delta)$, corresponding to an indecomposable symbol represented by $\tau_m$ (if exists), the Lie brackets of two vectors in the box of weight $\frac{1}{2}$ and $-\frac{1}{2}$ is equal to $\pm z$.
\end{enumerate}

Then consider the distribution $D^{\rm sp}_\delta$ which is the left invariant distribution on the connected simply connected Lie group with the Lie algebra $\mathfrak \mathfrak f_\delta$ such that at the identity $e$ of the group the subspace $D^{\rm sp}_\delta(e)$ is equal to the span of $x$ and the vectors from the tableaux $\Tab_{-}(\delta)$ with the nonnegative weights (or ,equivalently, with weight $0$ or $\frac{1}{2}$).
It is easy to see from the definition that he distribution
$D^{\rm sp}_\delta$ has Jacobi symbol $\delta$.
% and is of maximal class.

\begin{definition}
The distribution $D^{\rm sp}_\delta$ is called the symplectically flat distribution with Jacobi symbol $\delta$.
\end{definition}

The reason for this name is that these distributions play the same role in local geometry of distribution with fixed Jacobi symbol, as flat distribution of fixed type in the local geometry of distributions with given Tanaka symbol (see the last sentence of Theorem \ref{maintheor} below).

Let us describe the symplectically flat distributions of indecomposable symbols more explicitly
\begin{example}
\label{rank3flatex}
Assume that $\delta=\mathbb R\delta_{s,l}$ with nonnegative $s$ and an integer $l$ satisfying $s\leq l\leq 2s$. Also assume that the Young tableaux $\Tab(\mathbb R\delta_{s,l})$ is as in \eqref{youngodd1}. Then the Lie algebra
$\mathfrak f_{\mathbb R\delta_{s,l}}$ is spanned by $e_0,\ldots, e_{s-l}, f_0,\ldots, f_{-s}, x, z$ (in particular it is $l+4$-dimensional) and the only nonzero Lie brackets  in the product table of the Lie algebra  $\mathfrak f_{\mathbb R\delta_{s,l}}$ with respect to the chosen basis are

\begin{equation}
\label{Lieprodrank3}
%\mathfrak f_{\delta_{s,l}=\span\{e_0,\ldots, e_{s-l}, f_0,\ldots, f_{-s}, x, z\}\nolabel\\
\begin{split}
~&[x, e_i]=e_{i-1},\quad \forall i:s-l+1\leq i\leq 0, i\in\mathbb Z;\\
~&[x, f_i]=f_{i-1},\quad \forall j:-s+1\leq j\leq 0,j\in\mathbb Z;\\
~&[e_0,f_0]=z.
\end{split}
\end{equation}

The symplectically flat distribution $D^{\rm sp}_{\mathbb R\delta_{s,l}}$ with Jacobi symbol  $\mathbb R\delta_{s,l}$ is the left invariant rank $3$ distribution on the connected simply connected Lie group with the Lie algebra $\mathfrak f_{\mathbb R\delta_{s,l}}$ such that
$$D^{\rm sp}_{\mathbb R\delta_{s,l}}(e)={\rm span}\{x,e_0, f_0\}.$$
We can introduce another grading by negative integers  on the Lie algebra $\mathfrak f_{\mathbb R\delta_{s,l}}$ such that $\mathfrak f_{\mathbb R\delta_{s,l}}^{-1}={\rm span}\{x,e_0, f_0\}$ and the Lie algebra  $\mathfrak f_{\mathbb R\delta_{s,l}}$ is fundamental with respect to this grading, i.e. it is generated by  $\mathfrak f_{\mathbb R\delta_{s,l}}^{-1}$. In other words, the distribution $D^{\rm sp}_{\mathbb R\delta_{s,l}}$ is the flat distribution in the Tanaka sense for the fundamental Tanaka symbol given by $\mathfrak f_{\mathbb R\delta_{s,l}}$ with the new grading.
\end{example}

Similarly to Example \ref{rank3flatex} symplectically flat distributions of odd rank are flat in Tanaka sense for the corresponding Tanaka symbol.

\begin{example}
\label{rank2flatex}
Assume that $\delta=\mathbb R\tau_m$ where $m$ is positive, belongs to $\frac{1}{2}\mathbb Z_{\rm {odd}}$ and the Young tableaux $\Tab(\mathbb R\tau_m)$ is as in \eqref{youngodd2}. Then the Lie algebra
$\mathfrak f_{\mathbb R\tau_m}$ is spanned by $\varepsilon_{\frac{1}{2}},\ldots, \varepsilon_{-m}, x, z$ (in particular it is $m+\frac{7}{2}$-dimensional) and the only nonzero Lie brackets  in the product table of the Lie algebra  $\mathfrak f_{\mathbb R\tau_m}$ with respect to the chosen basis are

\begin{equation}
\label{Lieprodrank2}\begin{split}
%\mathfrak f_{\delta_{s,l}=\span\{e_0,\ldots, e_{s-l}, f_0,\ldots, f_{-s}, x, z\}\nolabel\\
~&[x, \varepsilon_i]=\varepsilon_{i-1},\quad \forall i:-m+1\leq i\leq \frac{1}{2}, i\in\frac{1}{2}\mathbb Z_{{\rm odd}};\\
%&~&[x, f_i]=f_{i-1},\quad \forall j:-s+1\leq j\leq 0, j\in\mathbb Z;\\
~&[\varepsilon_{\frac{1}{2}},\varepsilon_{-\frac{1}{2}}]=z \noindent.
\end{split}
\end{equation}

The symplectically flat distribution $D^{\rm sp}_{\mathbb R\tau_m}$ with Jacobi symbol  $\mathbb R\tau_m$ is the left invariant rank $2$ distribution on the connected simply connected Lie group with the Lie algebra $\mathfrak f_{\mathbb R\tau_m}$ such that

$$D^{\rm sp}_{\mathbb R\tau_m}(e)={\rm span}\{x,e_\frac{1}{2}\}.$$

As in the previous example, we can introduce another grading by negative integers on the Lie algebra $\mathfrak f_{\mathbb R\tau_m}$ such that $\mathfrak f_{\mathbb R\tau_m}^{-1}={\rm span}\{x,e_\frac{1}{2}\}$ and the Lie algebra  $\mathfrak f_{\mathbb R\tau_m}$ is fundamental with respect to this grading.
%, i.e. it is generated by  $\mathfrak f_{\delta_{s,l}}^{-1}$.
In other words, the distribution $D^{\rm sp}_{\mathbb  R\delta_{s,l}}$ is the flat distribution (in the Tanaka sense) for the fundamental Tanaka symbol given by  $\mathfrak f_{\mathbb R\tau_m}$ with the new grading.
\end{example}

\begin{example}
\label{rank4flatex}
Assume that $\delta=\mathbb R\left(\delta_{s,l}\oplus\tau_m\right)$ where  $s$ is a nonnegative integer,  $l$ is an integer  satisfying $s\leq l\leq 2s$, and $m$ is positive and belongs to $\frac{1}{2}\mathbb Z_{\rm {odd}}$ . Also assume that the Young tableaux $\Tab(\delta)$ is the union of Young tableaux appearing in \eqref{youngodd1} and \eqref{youngodd2}. Then the Lie algebra
$\mathfrak f_{\delta}$ is spanned by $e_0,\ldots, e_{s-l}, f_0,\ldots, f_{-s},\varepsilon_{\frac{1}{2}},\ldots, \varepsilon_{-m}, x, z$ (in particular it is $l+m+\frac{11}{2}$-dimensional) and the only nonzero Lie brackets  in the product table of the Lie algebra  $\mathfrak f_{\delta}$ with respect to the chosen basis are as in \eqref{Lieprodrank3} and \eqref{Lieprodrank2}. The symplectically flat distribution $D^{\rm sp}_{\delta}$ with Jacobi symbol  $\delta$ is the left invariant rank $4$ distribution on the connected simply connected Lie group with the Lie algebra $\mathfrak f_{\delta}$ such that
\begin{equation}
\label{flatD4}
D^{\rm sp}_{\delta}(e)={\rm span}\{x,e_0, f_0,\varepsilon_{\frac{1}{2}}\}.
\end{equation}
An interesting phenomenon here is that in contrast to previously discussed cases of distributions of odd rank and rank $2$ the symplectically flat distribution $D^{\rm sp}_{\delta}$ is not the flat distribution in Tanaka sense for a fundamental Tanaka symbol. Indeed, if it was the case
all vectors appearing in \eqref{flatD4} must have weight $-1$. From this, on one hand, $z=[e_0, f_0]$ has weight $-2$, but on the other hand, $\varepsilon_{-\frac{1}{2}}=[x, \varepsilon_{\frac{1}{2}}]$ has weight $-2$ and therefore $z=[\varepsilon_{\frac{1}{2}}, \varepsilon_{-\frac{1}{2}}]$ must have weight $-3$. We got a contradiction.

However, we can take another grading of the Lie algebra $\mathfrak f_{\delta}$ by negative integers such that the Lie algebra $\mathfrak f_{\delta}$ is generated by $\mathfrak f_{\delta}^{-1}\oplus\mathfrak f_{\delta}^{-2}$ and the  distribution $D^{\rm sp}_{\delta}$ is the left invariant distribution with $D^{\rm sp}_{\delta}(e)=\mathfrak f_{\delta}^{-1}\oplus\mathfrak f_{\delta}^{-2}$. For this set
$$\mathfrak f_{\delta}^{-1}={\rm span}\{\varepsilon_{\frac{1}{2}}\}, \quad \mathfrak f_{\delta}^{-2}={\rm span}\{x,e_0, f_0\}.$$
Note that the same phenomenon holds for all symplectically flat distribution of even rank greater than $2$.
\end{example}

%First we need to show that
%\begin{proposition}
%\label{model} The distribution
%$D^{\rm sp}_\delta$ has Jacobi symbol $\delta$ and is of maximal class.
%\end{proposition}
%\tb{$q\in M$} there exists a neighborhood $U$ in $M$ s.t. the Jacobi symbols of Jacobi curves of abnormal extremals through a generic point of \tb{$\mathbb{P}\mathcal H_D$} over \tb{$U$} are isomorphic to one symbol
 %and belonging to the corresponding symplectic algebras.

%\begin{definition}
%The equivalence class of line of degree $-1$ endomorphisms $\delta_\lambda\in {\mathfrak{csp}({\rm Gr(\lambda) $ up to a conjugation by an element  conformal symplectic group that
%\end{definition}

%For this let ${\rm Gr}^i(\lambda)=J^i(\lambda)/J^{i+s}(\lambda)$,  where $s=1$ if $D$ has odd rank or ${\rm rank}\, D=2$ and $s=\frac{1}{2}$ if $D$ has even rank greater than $2$.
%consider the graded space corresponding to the filtration $\{J^i(\lambda)\}_{i\in I}$, name the space
%$${\rm Gr}(\lambda)= \bigoplus_{i\in I}$$

\section{Canonical frame for distributions with given Jacobi symbols}
\setcounter{equation}{0}
\setcounter{theorem}{0}
\setcounter{lemma}{0}
\setcounter{proposition}{0}
\setcounter{definition}{0}
\setcounter{corollary}{0}
\setcounter{remark}{0}
\setcounter{example}{0}
\label{jacsymbsec}

To construct the canonical frame for a distribution $D$ with given Jacobi symbol we use the general theory of curves of flags developed in \cite{flag2} (with a slight modification in the case of distributions of even rank greater than $2$) and the prolongation procedure for flag structures developed in our previous paper \cite{quasi} . By a flag structure we mean a manifold endowed with a distribution and a submanifold of a flag variety of each fiber of this distribution. In our case \emph{the flag structure associated with the distribution $D$} is obtained as follows: the ambient manifold is the space of regular abnormal extremals $\mathcal N$ of $D$ as in \eqref{abnormspace}, the distribution on $\mathcal N$ is the contact distribution $\Delta$ and the submanifold of flags on the fibers $\Delta(\gamma)$ are the extended Jacobi curve
$\lambda \mapsto \{J^i(\lambda)\}_{i\in I}$, $\lambda\in \gamma$. Note that although  $\mathcal N$ may not be a manifold in general we can always overcome this issue by restricting ourselves to local considerations on $W_D$. Besides, in the case of distributions of odd rank or rank 2  the symplectic flag symbols of the modified
extended Jacobi curves are exactly the flag symbols of this flag structure in the sense of \cite{quasi}.

Assume that the modified Jacobi symbol $\delta_{\mod}$ of the distribution $D$ is represented by a line $\delta_{\mod}$ of degree $-1$ endomorphisms from $\mathfrak {csp}(X)$ , where $X=\displaystyle{\bigoplus_{i\in I}X^i}$ is a graded conformal symplectic space, $\dim X={\rm rank}\, \Delta$. Note that the grading on $X$ induces the natural filtration $\{X_i\}_{i\in I}$ with
\begin{equation}
\label{filtX}
X_i=\displaystyle{\bigoplus_{k\geq i}}X^k.
\end{equation}
 Furthermore, the grading  and the filtration on $X$ induce the natural grading and filtration on $\mathfrak{csp}(X)$. It is important to stress that the graded space $\gr X=\displaystyle{\bigoplus_{i\in I}} X_{i}/X_{n(i)}$ associated with this filtration, where $n(i)$ is the index next to the index $i$ in $I$ with respect to the natural order,  is canonically identified with $X$ via the canonical identification of $X_{i}/X_{n(i)}$ with $X^i$.

 The grading on $\mathfrak{csp}(X)$ is defined as follows:
$\mathfrak{csp}(X)=\displaystyle{\bigoplus_{j\in J}\mathfrak{csp}^j(X)}$, where
$$\mathfrak{csp}^j(X)=\{A\in \mathfrak{csp}(X): A(X^i)\in X^{i+j} \quad \forall i\in I\}$$
and $J=\{i_1-i_2: i_1, i_2\in I\}$, i.e. $J=\mathbb Z$, if $I=\mathbb Z \text{ or } \cfrac{1}{2}\,\mathbb Z_{\rm odd}$ and $J=\cfrac{1}{2}\,\mathbb Z$, if $I=\cfrac{1}{2}\,\mathbb Z$. The filtration on  $\mathfrak{csp}(X)$ is defined as follows:
$\{\mathfrak{csp}_j(X)\}_{j\in J}$, where
$$\mathfrak{csp}_j(X)=\{A\in \mathfrak{csp}(X): A(X_i)\in X_{i+j} \quad \forall i\in I\}=\displaystyle{\bigoplus_{k\geq j}}\mathfrak{csp}^k(X).$$

Now the filtration on $\mathfrak{csp}(X)$ induces the filtration on any its subspace $L$: $\{L_j\}_{j\in J}$ with $L_j=L\cap \mathfrak{csp}_j(X)$. Let $\gr L$ be the corresponding graded space, $\gr L=\displaystyle{\bigoplus_{j\in J}}L_{j}/L_{n(j)},$ where, similarly to above,
 $n(j)$ is the index next to the index $j$ in $J$  with respect to the natural order. Note that we have the following canonical identifications:
 \begin{equation}
 \label{grid}
 \gr\mathfrak{csp}(X)\cong \mathfrak{csp}(\gr X) \cong \mathfrak{csp}(X).
 \end{equation}
The first identification is standard, in the last identification we use the identification of $\gr X$ with $X$ described above.  Using these identifications the space $\gr L$ can be seen as a subspace of $\mathfrak{csp}(X)$, i.e. of the same space as the original subspace $L$. Note that in general $\gr L$ differs from $L$.

%$\varepsilon=1$ if  $I=\mathbb Z \text{ or } \cfrac{1}{2}\,\mathbb Z_{\rm odd}$ and $\varepsilon =\cfrac{1}{2}$ if $I=\cfrac{1}{2}\,\mathbb Z$.

Further note that  the Tanaka symbol of the contact distribution $\Delta$ is isomorphic to the Heisenberg algebra of the corresponding dimension that will be denoted by $\eta$ with the natural grading $\eta=\eta^{-1}\oplus\eta^{-2}$. The model graded conformal symplecic space  $X$ can be identified with $\eta^{-1}$. Then the algebra $\mg^0(\eta)$ of all derivations of $\eta$ preserving the grading is equal to $\mathfrak{csp}(X)$.

As in Remark \ref{Tanakastr} denote by $P^0(\eta)$ the principal bundle over $\mathcal N$ such that the fiber $P^0(\eta)_\gamma$ over a point $\gamma\in \mathcal N$ consists of all isomorphisms between the conformal symplectic spaces $X$ and $\Delta(\gamma)$, i.e. linear isomorphisms preserving the conformal symplectic structures on $X$ and $\Delta(\gamma)$.
In other words, $P^0(\eta)$ is the bundle of conformal symplectic frames of each fiber of the contact distribution $\Delta$.
%Note that $P^0$ is a principle $CSP(X)$-bundle.
%
%
%
%\begin{remark}
%\label{tanakastr1}
%Note that the Tanaka symbol of the contact distribution $\Delta$ is isomorphic to the Heisenberg algebra of the corresponding dimension that will be denoted by $\eta$ with the natural grading $\eta=\eta^{-1}\oplus\eta^{-2}$, the space  $X$ is identified with $\eta^{-1$ and the bundle $P^0$ is the same as the bundle $P^0(\eta)$ in the terminology of Remark \ref{Tanakastr}.
%\end{remark}

%As mentioned in Remark \ref{tanakastr}
At the first step of our constructions, following \cite{flag2},
 we would like to encode the presence of the Jacobi curves on the fibers of the contact distribution $\Delta$ by adopting to them a fiber subbundle  of $P^0(\eta)$ of the minimal possible dimension in a canonical way.
 In contrast to Tanaka structures of Remark \ref{Tanakastr} such subbundle cannot be chosen in general as a principal reduction of the bundle $P^0(\eta)$. However, this subbundle belongs to  more general class of fiber subbundles of $P^0$ that we call \emph{quasi-principal bundles}.

 To define the latter notion first note that, since the tangent spaces to the fibers of $P^0(\eta)$ are canonically identified with the
 Lie algebra   $\mathfrak {csp}(X)$,  the tangent space to a fiber subbundle $P^0$ of $P^0(\eta)$
 at a point $\psi$ is identified canonically with a subspace $L_\psi$ of $\mathfrak{csp}(X)$.  If $P^0$ is a principal reduction of $P^0(\eta)$ these subspaces
 are the same for different points $\psi$ and equal to the Lie subalgebra of the structure Lie group of $P$.
 In general these subspaces vary when one varies the points $\psi$. However, we can weaken the assumption of constancy of these subspaces by assuming that the corresponding graded space $\gr L_\psi$, seen again as the subspace of $\mathfrak{csp}(X)$ under the identifications \eqref{grid}, is independent of $\psi$ and equal to a subalgebra $\mathfrak g^0$ of $\mathfrak{csp}(X)$. In this case we say that the fiber subbundle $P^0$ of $P^0(\eta)$ is a \emph{quasi-principal bundle of type $(\eta, \mathfrak g^0)$}.

%  %naturally associated with  the flag structure $(\Delta,\{\mathcal Y^\gamma\}_{\gamma\in\mathcal S})$.
%                                                                    %a subbundle of $P^0_+(\mathfrak m)$ of the minimal possible dimension.
%At the first step of our constructions  the most important object is For this  one needs  to take a closer look to the (extrinsic) geometry of submanifold of flags $\mathcal Y^\gamma$ with respect to the natural action of  the group $G^0_\gamma$ ($\sim G^0(\mathfrak m)$).
In our case the algebra $\mathfrak g^0$ can be described purely algebraically in terms of the modified Jacobi symbol $\delta_{\rm mod}$, namely in terms of
the \emph{universal algebraic prolongation $\mathfrak u^F(\delta_{\rm mod})$ of $\delta_{\rm mod}$} (in $\mathfrak {csp}(X)$), which plays the same role in the geometry of curves of flags as the Tanaka universal algebraic prolongation on the geometry of filtered structures \cite{flag2}. This is by definition the largest graded Lie subalgebra of the Lie algebra  $\mathfrak {csp}(X)$ such that its negative part coincides with $\delta_{\mod}$. The algebra $\mathfrak u^F(\delta_{\rm mod})$ can be explicitly constructed by induction.

First let us describe this inductive construction in the case when $I=\mathbb Z \text{ or } \cfrac{1}{2}\mathbb Z_{\rm odd}$ (corresponding to the cases when the original distributions $D$ is of odd rank or of rank 2, respectively).
Set $\mathfrak u_{-1}^F(\delta_{\rm mod}):=\delta_{\rm mod}$ and define by induction in $k\in \mathbb Z$, $k\geq 0$
\begin{equation}
\label{kprolong}
\mathfrak u_k^F(\delta_{\rm mod}):=\{v\in \mathfrak {csp}^{k}(X):[v,\delta_{\rm mod}]\in \mathfrak u_{k-1}^F(\delta)\}.
\end{equation}
The space $\mathfrak u_k^F(\delta_{\rm mod})$ is called the \emph {$k$th algebraic prolongation of the symbol $\delta_{\rm mod}$}. Then
\begin{equation}
\label{AUF}
\mathfrak u^F(\delta_{\rm mod})=\displaystyle{\bigoplus_{k\in \mathbb Z,\, k\geq -1}\mathfrak u_k^F(\delta_{\rm mod})}
\end{equation}

In the case when $I=\cfrac{1}{2}\,\mathbb Z_{\rm odd}$ (corresponding to the cases when the original distributions $D$ is of even rank greater than $2$) the following modification is needed:
Set $\mathfrak u_{-1}^F(\delta_{\rm mod}):=\delta_{\rm mod}$ and  $\mathfrak u_{-\frac{1}{2}}^F(\delta_{\rm mod}):=0$ and define $\mathfrak u_k^F(\delta_{\rm mod})$ as in \eqref{kprolong} where $k$ are nonnegative and belong  $\cfrac{1}{2}\,\mathbb Z_{\rm odd}$. Then the algebra $\mathfrak u_k^F(\delta_{\rm mod})$ satisfies \eqref{AUF} with $\mathbb Z$ replaced by $\cfrac{1}{2}\,\mathbb Z_{\rm odd}$ in the summation of the righthand side.
% is a graded
%subalgebra of $\gr\mg^0(\mathfrak m)$.

%\begin{remark}
%\label{flatrem}
%(Flat curves of flags with given flag symbol)
As in the theory of  Tanaka structures, one can define the flat curve with given symplectic flag  symbol, which is in a sense the simplest (the most symmetric) curve among all curves of symplectic flags with this flag symbol. In our case consider the filtration/ the flag $\{X_i\}_{i\in I}$  on the graded conformal symplectic space $X$, where subspaces $X_i$ are as in \eqref{filtX} .

\begin{definition}
\label{flatcurve}
The \emph{flat curve $\mathcal F_{\delta_{\mod}}$ of flags with flag symbol $\delta_{\mod}$} is, up to the action of the conformal symplectic group $\mathfrak{csp}(X)$ on the corresponding flag variety,  the closure of the orbit of this filtration under the action  of the one-parametric subgroup generated by the line $\delta_{\mod}$, namely the closure of the curve $x\mapsto \{e^x X_i\}_{i\in I}, \, x\in \delta_{\mod}$.
\end{definition}

It can be shown by direct computations that for the symplectically flat distribution  $D^{\rm sp}_{\delta}$ the germ of Jacobi curves at a generic point are germs of flat curves with flag symbol $\delta$. Moreover, the corresponding flag structure is locally isomorphic to the left-invariant flag structure  on the Heisenberg group of dimension equal to $\dim \mathcal N$ (where $\mathcal N$ is the space of regular abnormal extremals) such that at the identity of the group the curve of flags in the corresponding fiber of the contact distribution is the flat curve with flag symbol $\delta$.

The building blocks for a more geometric description of flat curves of flags with given flag symbol are rational normal curves in  projective spaces. Recall that a rational normal curve in $r$-dimensional projective space $\mathbb P^r$ is a curve represented as
%\begin{equation}
%\label{ratnormcurve}
$$t\mapsto [1:t:t^2:\ldots t^r]$$
%\end{equation}
 in some homogeneous coordinates.
If $\delta_{\mod}=\mathbb R\tau_m$ for some positive $m\in \frac{1}{2}\mathbb Z_{\rm {odd}}$ (as in the case of $(2, \frac{2m+7}{2})$-distributions of maximal class), then the flat curve $\mathcal F_{\mathbb R\tau_m}$ is the curve of complete flags such that the curve of one dimensional subspaces of the flags in $\mathcal F_{\mathbb R\tau_m}$ is a rational normal curve in the $(2m+1)$-dimensional  projective space $\mathbb P X$ and all other subspaces of the flags in $\mathcal F_{\mathbb R\tau_m}$ are obtained by osculation of the rational normal curve.

For a general symplectic flag symbol $\delta_{\mod}$ fix a skew Young tableaux  $\Tab(\delta_{\rm{mod}})$ of $\delta_{\rm{mod}}$, assume that this tableaux consists of $\tilde k$ rows (recall that $\tilde k=\rank D-1$),  and let $\widetilde L_i$, $1\leq i\leq \tilde k$, be
%%First, for each row of the skew Young tableaux $\Tab(\delta_{\rm{mod}})$ consider
the subspaces of $X$  spanned by the vectors appearing in the boxes of the $i$th row of $\Tab(\delta_{\rm{mod}})$. Then intersecting the subspaces of flags of the flat curve $\mathcal F_{\delta_{\mod}}$  with the subspace $\widetilde L_i$ one gets again the curve of osculating subspaces of a rational normal curve in the projective space $\mathbb P \widetilde L_i$. Informally speaking, the flat curve $\mathcal F_{\delta_{\mod}}$ is a direct sum of $\tilde k$ curves of osculating subspaces sitting in $\widetilde L_i$ of rational normal curves such that the grading is taken into account.
%$\Box$
%\end{remark}

\begin{remark}
\label{symmflatrem}
(Geometric interpretation of the universal algebraic prolongation in terms of flat curves)
Finally note (\cite{doubkom}) that the algebra $\mathfrak u^F(\delta_{\mod})$ is isomorphic to the algebra of infinitesimal symmetries (in $\mathfrak {csp}(X)$) of the flat curve $\mathcal F_{\delta_{\mod}}$.
%$\Box$

\end{remark}

\begin{theorem}
\label{quasidelta}
 To a germ of distribution $D$ with constant Jacobi symbol $\delta$ one can assign in a canonical way a quasi-principal subbundle of $P^0(\eta)$ of type $\bigl(\eta, \mathfrak  u^F(\delta_{\rm mod})\bigr)$, where $\delta_{\rm mod}$ is the modified Jacobi symbol.
\end{theorem}

Note that the assignment in Theorem \ref{quasidelta} is uniquely determined by a choice  of a so-called normalization condition for the structure equation of the moving frames associated with the extended Jacobi cirves.
%, \cite{flag2, flag1}
The normalization conditions are given  by a subspace  $W$  complementary to the subspace $\mathfrak u^F_+(\delta_{\rm mod})+[\delta_{\rm mod}, \mathfrak {csp}_+(X)]$ in $\mathfrak {csp}_+(X) $ , where $\mathfrak {csp}_+(X)$ is the nonnegative part of the graded algebra $\mathfrak {csp}(X)$
%of the group $G^0_+(\mathfrak m) $
and $\mathfrak u^F_+(\delta_{\rm mod})= u^F(\delta_{\rm mod})\cap \mathfrak {csp}_+(X)$.
%Similar description of normalization conditions can be given in the case of submanifolds of flags \cite{flag2, flag1}.
This theorem is a direct consequence of \cite[Theorem 4.4]{flag2} applied to the extended Jacobi curve of each regular abnormal extremal $\gamma$ in the case of distributions of odd rank or rank 2 and a straightforward generalization of this result in the case of distributions of even rank greater than $2$. As was already mentioned in Remark \ref{evendiffrem} the latter case does not satisfy the compatibility with respect to differentiation assumption.  However the constructions in this case are literally the same as in the proof of Theorem 4.4 of \cite{flag2}.

The result of Theorem \ref{quasidelta} can be seen as the zero step of the prolongation procedure for the construction of the canonical frame for the distribution $D$ if compared with the prolongation procedure for Tanaka structures (as in  Remark \ref{Tanakastr} and Theorem \ref{tanthm}). In \cite{quasi} we have shown that the analog of Theorem \ref{tanthm} is valid for quasi-principle bundle as well. This implies the following

\begin{theorem}
\label{maintheor}
Assume that $D$ is a distribution with constant  Jacobi symbol $\delta$ and $\delta_{\rm mod}$ is the modified Jacobi symbol of $D$.
Assume that the universal algebraic prolongation $\mathfrak u\bigl(\eta, \mathfrak u^F(\delta_{\rm mod})\bigr)$ of the pair $(\eta, \mathfrak u^F(\delta_{\rm mod}))$ is finite dimensional and $k\geq 0$ is the maximal integer such that $\mathfrak u^k\bigl(\eta,\mathfrak u^F(\delta_{\rm mod}) \bigr)$ is not zero.
%To any distribution with constant Jacobi symbol $\delta$
Then
%\begin{enumerate}
%\item
to the distribution $D$
one can assign, in a canonical way,
a sequence of bundles $\{P^i\}_{i=0}^k$ such that  $P^0$ is as in Theorem \ref{quasidelta} and $P^i$ for $i\geq 0$ is an affine bundle over $P^{i-1}$ with fibers being affine spaces over the linear space of dimension equal to $\dim \,\mathfrak u^i\bigl(\eta,\mathfrak u^F(\delta_{\rm mod})\bigr)$ for any $i=1,\dots k$ and such that $P^k$ is endowed with the canonical frame.
In particular,
the dimension of the algebra of infinitesimal symmetries of a distribution $D$ with Jacobi symbol $\delta$ is not greater then $\dim \mathfrak u\bigl(\eta, \mathfrak u^F(\delta_{\rm mod})\bigr)$. This upper bound
%for the algebra of infinitesimal symmetries
is sharp and is achieved if and only if a distribution is locally equivalent to the symplectically  flat distribution $D^{\rm sp}_{\delta}$ for which the algebra of infinitesimal symmetries is isomorphic to $\mathfrak u\bigl(\eta, \mathfrak u^F(\delta_{\rm mod})\bigr)$.
%\end{enumerate}
\end{theorem}

Note that the statement of this Theorem about sharpness of the bounds and uniqueness of maximally symmetric model up to an equivalence does not directly follow from Theorem 2.3 of \cite{quasi} but it can be proved without difficulties using the arguments of existence of so-called grading element
in the algebra $\mathfrak u\bigl(\eta, \mathfrak u^F(\delta_{\rm mod})\bigr)$  similarly to \cite{tan2},\cite{yamag}.
%We finish this section with two remarks, which clarify geometric meaning of the algebra  $u^F(\delta_{\rm mod})$ and will be used in section \ref{normrat} for geometric description of the algebra $\mathfrak u\bigl(\eta, \mathfrak u^F(\delta_{\rm mod})\bigr)$.

\section{Jacobi symbols $\delta$ with finite dimensional algebra $\mathfrak u(\eta, \mathfrak u^F(\delta_{\rm mod}))$}
\setcounter{equation}{0}
\setcounter{theorem}{0}
\setcounter{lemma}{0}
\setcounter{proposition}{0}
\setcounter{definition}{0}
\setcounter{corollary}{0}
\setcounter{remark}{0}
\setcounter{example}{0}
\label{findimsec}
Now the natural question is for what Jacobi symbols $\delta$ is the algebra  $\mathfrak u(\eta, \mathfrak u^F(\delta_{\rm mod}))$  finite dimensional?

\begin{theorem}
\label{finitetype}
Assume that $D$ is a distribution with constant  Jacobi symbol $\delta$ and $\delta_{\rm mod}$ is the modified Jacobi symbol of $D$.
\begin{enumerate}
\item
If ${\rm rank}\, D$ is odd then  $\mathfrak u\bigl(\eta, \mathfrak u^F(\delta_{\rm mod})\bigr)$ is finite dimensional if and only if  $\delta_{\rm mod}$ is isomorphic to the direct sum of symplectically
indecomposable symbols generated by $\delta_{s,l}$ with $s<l\leq 2s$, $s,l$ are positive integers. In other words, each  symplectically  indecomposable component of $\delta_{\rm mod}$ has the skew Young diagram consisting of two rows as in \eqref{youngodd} and  having at least three columns with two boxes;

\item If ${\rm rank}\, D=2$  then  $\mathfrak u\bigl(\eta, \mathfrak u^F(\delta_{\rm mod})\bigr)$ is finite dimensional if and only if  $\delta_{\rm mod}$ is isomorphic to the symbol generated by
$\tau_m$ with $m>\frac{1}{2}$, $m\in\cfrac{1}{2}\,\mathbb Z_{\rm{odd}}$. In other words, the skew  Young diagram of $\delta_{\rm mod}$ consists of one row with at least $4$ boxes.

\item
If ${\rm rank}\, D$ is even then  $\mathfrak u\bigl(\eta, \mathfrak u^F(\delta_{\rm mod})\bigr)$ is finite dimensional if and only if  $\delta_{\rm mod}$ is isomorphic to the direct sum of symplectically
indecomposable symbols generated by $\delta_{s,l}$ with $s<l\leq 2s$, $s,l$ are positive integers, and  of exactly one symplectic symbol  generated by $\tau_m$ with $m\geq\frac{1}{2}$, $m\in\cfrac{1}{2}\,\mathbb Z_{\rm{odd}}$. In other words, the symplectically indecomposable components of $\delta_{\rm mod}$ have the skew Young diagrams consisting of either two rows as in \eqref{youngodd} with at least 3 columns with two boxes or of one row  and the latter component appears exactly ones.
\end{enumerate}
%Assume that $D$ is a distribution with constant  Jacobi symbol $\delta$ and $\delta_{\rm mod}$ is the modified Jacobi symbol of $D$. Then the algebra $\mathfrak u\bigl(\eta, \mathfrak u^F(\delta_{\rm mod}\bigr))$ is finite dimensional if and only $\delta_{\rm mod}$ is isomorphic to the direct sum of
%indecomposable symplectic symbols generated by $\delta_{s;l}$ with $s<l\leq 2s$, $s,l$ are positive integers, and of at most one symplectic symbol  generated by $\tau_m$ with $m>\frac{1}{2}$, $m\in\cfrac{1}{2}}\mathbb Z_{\rm{odd}}$. In other words, the indecomposable components of $\delta_{\rm mod}$ have the skew Young diagrams consisting either two rows as in \eqref{youngodd} with at least 3 columns with two boxes or one row with at least $4$ boxes and the latter component appears at most ones (see also Corollary \ref{sepcases} below for another detailed description considering cases of different ranks of $D$ separately).
\end{theorem}

Note that in item (3) of this theorem, in contrast to item 2, the Young diagram of the  symplectically indecomposable component consisting of one row may have $2$ boxes.

As a direct consequence of Theorem \ref{finitetype} we get the following
\begin{corollary}
\label{maxclasscor}
If  $\mathfrak u\bigl(\eta, \mathfrak u^F(\delta_{\rm mod})\bigr)$ is finite dimensional then $\delta=\delta_{\rm mod}$, i.e. the distribution $D$ is of maximal class.
\end{corollary}

Also as a direct consequence of item (1) of Theorem \ref{finitetype} and Remark \ref{6dim} one has the following

\begin{corollary}
\label{finitetyperank3}
For a rank $3$ distribution $D$ the algebra $\mathfrak u\bigl(\eta, \mathfrak u^F(\delta_{\rm mod})\bigr)$ is finite dimensional if and only if $D$ is of maximal class and $\dim D^2(q)=6$ for generic point $q$.
\end{corollary}

Using Remark \ref{middleodd}, Theorem \ref{finitetype} can be equivalently reformulated in the following shorter way:

\begin{theorem}
\label{finitetype1}
Assume that $D$ is a distribution with constant  Jacobi symbol $\delta$ and $\delta_{\rm mod}$ is the modified Jacobi symbol of $D$. Then the algebra $\mathfrak u\bigl(\eta, \mathfrak u^F(\delta_{\rm mod})\bigr)$ is finite dimensional if and only if the skew Young diagram $\Diag(\delta_{\rm mod})$ of $\delta_{\rm mod}$ satisfies both of the following two conditions:
\begin{enumerate}
\item
 It does not contain two rows such that the most left box of one of the rows is located in the same column or from the right to the most right box of another row;
\item
 It does not consist of one row with two boxes.
 \end{enumerate}
\end{theorem}
%\begin{corollary}
%\label{sepcases}
%Assume that $D$ is a distribution with constant  Jacobi symbol $\delta$ and $\delta_{\rm mod}$ is the modified Jacobi symbol of $D$.
%\begin{enumerate}
%\item
%If ${\rm rank} D$ is odd then  $\mathfrak u\bigl(\eta, \mathfrak u^F(\delta_{\rm mod})\bigr)$ is finite dimensional if and only if  $\delta_{\rm mod}$ is isomorphic to the direct sum of
%indecomposable symbols generated by $\delta_{s;l}$ with $s<l\leq 2s$, $s,l$ are positive integers;
%
%\item If ${\rm rank} D=2$  then  $\mathfrak u\bigl(\eta, \mathfrak u^F(\delta_{\rm mod})\bigr)$ is finite dimensional if and only if  $\delta_{\rm mod}$ is isomorphic to the symbol generated by
%$\tau_m$ with $m>\frac{1}{2}$, $m\in\cfrac{1}{2}}\mathbb Z_{\rm{odd}}$
%
%\item
%If ${\rm rank} D$ is even then  $\mathfrak u\bigl(\eta, \mathfrak u^F(\delta_{\rm mod})\bigr)$ if and only if  to the direct sum of
%indecomposable symbols generated by $\delta_{s;l}$ with $s<l\leq 2s$, $s,l$ are positive integers, and  of exactly one symplectic symbol  generated by $\tau_m$ with $m>\frac{1}{2}$, $m\in\cfrac{1}{2}}\mathbb Z_{\rm{odd}}$.
%\end{enumerate}
%\end{corollary}

{\bf Proof of Theorem \ref{finitetype1}}
The proof is based on two general criteria for finiteness of prolongations of Tanaka and Spencer, respectively, and on some theory of representation of $\mathfrak{sl}_2$.

Let us first describe these criteria in general setting. Let as before $\mathfrak t=\displaystyle{\bigoplus_{j\le -1}\mg^i}$ be a graded nilpotent Lie algebra, $\mg^0(\mathfrak t)$ be the algebra of differentiations of $\mathfrak t$ preserving the grading, and $\mg^0$ be a subalgebra of $\mg^0(\mathfrak t)$.
Tanaka criterium reduces the question whether the Tanaka prolongation $\mathfrak u(\mathfrak t, \mg^0)$ of a pair $(\mathfrak t, \mg^0)$ is finite-dimensional or not to the same question for the
standard prolongation (as described, for example, in \cite{kob, stern}) of the following subalgebra $\mathfrak r(\mg^0)$ of $\mg^0$.
\begin{equation}
\label{srez}
%\mathfrak h^0
\mathfrak r(\mg^0)= \{f\in \mg^0 | f(x) = 0 \text{ for all }  x\in \mg^i, i \leq -2\}.
\end{equation}
Recall that $\mg^0(\mathfrak t)$ (and hence $\mathfrak r(\mg^0)$) can be considered as a subalgebra of $\mathfrak {gl}(\mg^{-1})$. The standard prolongation of  $\mathfrak r(\mg^0)$ is nothing but the Tanaka prolongation of the pair $(\mg^{-1}, \mathfrak r(\mg^0))$, where $\mg^{-1}$ is considered as a commutative graded Lie algebra with the trivial grading, i.e. such that all elements of $\mg^{-1}$ have degree $-1$. Note  that the standard prolongation can be defined not only for a subalgebra but for any subspace of $\mathfrak {gl}(\mg^{-1})$ (\cite{kob,stern}).
%We can naturally identify h0 with a subspace in End(m.1). Then
Tanaka criteruim can be formulated as follows:
\begin{theorem}
\label{Tanakafin}
{\bf Tanaka criterium} (\cite{tan1}).
The Tanaka prolongation $\mathfrak u(\mathfrak t, \mg^0)$ of the pair $(\mathfrak t, \mg^0)$ is finite-dimensional if and only if so is the standard prolongation of $\mathfrak r(\mg^0)$.
\end{theorem}

\begin{remark}
\label{symprem}
In the case when  $\mathfrak t=\eta(= \eta^{-1}\oplus\eta^{-2})$, the Heisenberg algebra, $\mg^0(\eta)$ is the conformal symplectic algebra $\mathfrak{csp}(\eta^{-1})$ of $\eta^{-1}$, and the $\mathfrak r\Bigl(\mg^0(\eta)\Bigr)$ is equal to the symplectic algebra $\mathfrak{sp}(\eta^{-1})$. This implies also that  the algebra $\mathfrak r\bigl(\mathfrak u^F(\delta_{\rm mod})\bigr)$ is equal to the universal algebraic prolongation of the symbol $\delta$ in the algebra $\mathfrak{sp}(\eta^{-1})$, i.e.  the largest graded Lie subalgebra of the Lie algebra  $\mathfrak {sp}(\eta^{-1})$ such that its negative part coincides with $\delta$. Obviously,

\begin{equation}\label{speq}
\mathfrak u^F(\delta_{\rm mod})\bigr)= \mathfrak r\bigl(\mathfrak u^F(\delta_{\rm mod})\bigr)\oplus \mathbb R\, \Id_{\eta^{-1}}.
\end{equation}
\end{remark}

%Note that  $\mathfrak{csp}(\mg^{-1})= \mathfrak{sp}(\mg^{-1})\oplus R \Id|_{\mg^{-1}}$.
In its turn, Spencer criterium provides a computationally efficient
method of detecting whether the standard prolongation of a linear subspace
$A\subset \mathfrak{gl}(Y)$ of a vector space $Y$ over the field $k$ is finite-dimensional or not. Let $\bar k$ be the algebraic closure of the field $k$ and
by $A^{\bar k}$ be the subspace in $\mathfrak{gl}(Y^{\bar k})$  obtained from $A$ by field extension.

\begin{theorem}
{\bf Spencer criterium}
\label{spencerthm}
(\cite{spencer,gul,ottazzi})
The standard prolongation of a subspace $A\subset \mathfrak{gl}(Y)$ is finite-dimensional if and only if $A^{\bar k}\subset \mathfrak{gl}(Y^{\bar k})$ does
not contain endomorphisms of rank $1$.
%A detailed and self-contained proof of Spencer criterium can be found
%in [12].
\end{theorem}

The  universal algebraic prolongations of symplectic flag symbols in the symplectic algebra (and therefore, according to \eqref{speq}, in the conformal symplectic algebra) were described quite explicitly in \cite[section 8]{flag2} using the theory of $\mathfrak{sl}_2$ representation.
Let us briefly list the main features of this algebras in context of the Spencer criterium.

We start with some terminology.
We say that a graded space $X$ which is also $\sll_2$-module is a \emph{nice $\sll_2$-module}, if the corresponding embedding of $\sll_2$ into $\gl(X)$ is spanned by endomorphisms of degree $-1$, $0$, and $1$. Let $Y_1$ and $Y_2$ be two
%graded vector space which are also
nice $\sll_2$-modules.
Then $\Hom(Y_1,Y_2)=Y_2\otimes Y_1^*$ is the $\sll_2$-module  and a graded space in a natural way.
Denote by $\mathfrak n(Y_1,Y_2)$ the maximal $\sll_2$-submodule of $\Hom(Y_1,Y_2)$  concentrated in the non-negative degree part. If $Y_1=Y_2$ we simply will write $\mathfrak n(Y_1)$ instead of $\mathfrak n(Y_1,Y_2)$.

Further, assume that $X$ is  a graded conformal symplectic
%(orthogonal)
space
%(with the symplectic form $\sigma$)
 which is also a nice $\sll_2$-module such that the corresponding embedding of $\sll_2$ into $\gl(X)$ belongs to $\mathfrak {sp}(X)$.
%($\mathfrak so(V)$).
In this case we will say that $X$ is a \emph{nice symplectic
%(orthogonal)
$\sll_2$-module}.
The conformal symplectic algebra $\mathfrak{csp}(X)$
%(the orthogonal algebra $\mathfrak{so}(V)$)
is a $\sll_2$-module and  a graded space in a natural way. Denote by $\mathfrak l
%^{\mathfrak{sp}}
(X)$
%($\mathfrak k(V)$)
the maximal $\sll_2$-submodule of $\mathfrak {sp}(X)$
%($\mathfrak so(V)$)
concentrated in the non-negative degree part. From the maximality assumption it follows  that $\mathfrak l
%^{\mathfrak{sp}}
(X)$
%($\mathfrak k(V)$)
is a subalgebra of $\mathfrak {sp}(X)$.
% (of $\mathfrak{so}(V)$).

Given a symplectic flag symbol $\rho$ on a graded conformal symplectic space $X$ we can make $X$ to a nice symplectic
%(orthogonal)
$\sll_2$-module with $\rho$ being degree $-1$ component of the corresponding embedding $\vf_\rho$ of $\sll_2$ into $\mathfrak {sp}(X)$. Moreover, if one fix a skew Young tableaux $\Tab(\rho)$ of $\rho$ (as described in Remark \ref{Youngrem}), then each row in  $\Tab(\rho)$  defines an irreducible $\sll_2$-submodule. From \cite[subsection 8.3]{flag2} it follows that the universal algebraic prolongation $\mathfrak r\bigl(\mathfrak u^F(\rho)\bigr)$ of the symbol $\rho$ in $\mathfrak{csp}(X)$ is equal to the semidirect sum of the constructed  embedding $\vf_\rho$ of $\sll_2$ into $\mathfrak{sp}(X)$ %($\mathfrak{so}(V)$)
and the algebra $\mathfrak l
%^{\mathfrak{sp}}
(X)$,

\begin{equation}
\label{semidir}
\mathfrak r\bigl(\mathfrak u^F(\rho)\bigr)=\mathfrak l
%^{\mathfrak{sp}}
\left(X\right)\rtimes_{\vf_\rho}\sll_2.
\end{equation}

\begin{remark}
\label{symparamrem}
%First note that the element from the component $\sll_2$ in \eqref{semidir} has rank 1 if and only if $\dim X\leq 2$
By Remarks \ref{symmflatrem} and \ref{symprem} the algebra $\mathfrak r\bigl(\mathfrak u^F(\rho)\bigr)$ is isomorphic to the subalgebra of  $\mathfrak{sp}(X)$ consisting of all infinitesimal symmetries of the flat curve $\mathcal F_\rho$ with the symbol $\mathfrak \rho$.
   %(see subsection \ref{univsect}),
Hence, to any element of $\mathfrak r\bigl(\mathfrak u^F(\rho)\bigr)$ we can assign the vector field on the curve $\mathcal F_\rho$. It turns out that the ideal $\mathfrak l(X)$ consists
    %Consider the subspace $\ng_{ne}(\mathfrak m)$ of $\ug(\mathfrak m)$, consisting
of all elements of $\mathfrak r\bigl(\mathfrak u^F(\rho)\bigr)$ for which the corresponding vector fields on $\mathcal F_\rho$ are identically zero. In other words, $\mathfrak l(X)$ is the subalgebra of  $\mathfrak{sp}(X)$ consisting of the infinitesimal symmetries of the flat curve $\mathcal F_\rho$ considered as a parametrized curve. This implies that $A\in \mathfrak l(X)$ if and only if

\begin{equation}
\label{lXcond}
A \bigl(e^{x}(X_i)\bigr)\subset e^{x}(X_i),\quad \forall x\in\rho,\,\, i\in I,
\end{equation}
where $\{X_i\}_{i\in I}$ is the filtration of $X$ as in \eqref{filtX}.
\end{remark}
\begin{remark}
\label{glpass}
Note that analogous statements are true if we forget about the presence of symplectic structure on $X$ and consider curves of flags with respect to the action of $GL(X)$. In all definitions and statements the algebras $\mathfrak{csp}(X)$ or $\mathfrak{sp}(X)$ should be replaced by $\gl(X)$ , the algebra $\mathfrak l(X)$ is replaced by $\mathfrak n(X)$ and the latter satisfies relation \eqref{lXcond}.
\end{remark}

{\bf 1. Rank $1$ elements in $\mathfrak l
%^{\mathfrak{sp}}
\left(x\right)$} Our first goal is to prove the following

\begin{proposition}
\label{propL}
 In the case when $\rho=\delta_{mod}$ the algebra $\mathfrak l
%^{\mathfrak{sp}}
\left(X\right)$ contains elements of rank 1 if and only if the corresponding skew Young diagram $\Diag(\delta_{mod})$ does not satisfy condition (1) of Theorem \ref{finitetype1}, i.e. it contains two rows such that the most left box of one of the rows is located in the same column or from the right to the most right box of another row
\end{proposition}
%in the case when $\rho=\delta_{mod}$.
\begin{proof}
In this case $\rho$ is isomorphic to the direct sum of $N(s,l)$ copies of the indecomposable symbols $\mathbb R\delta_{s,l}$ for some finite number of pairs $(s,l)$ with $s\geq 0$ and $0\leq l\leq 2s$ and of at most one copy of $\mathbb R\tau_m$ for some $m\in\cfrac{1}{2}\mathbb Z_{\rm {odd}}$.
Then we can split $X$ as
 \begin{equation}
\label{sympdecomp}
X=\displaystyle{\bigoplus_{i=1}^k L_i},
\end{equation}
for some $k\in\mathbb N$, where $N(s,l)$ of subspaces  $L_i$ are isomorphic  to $X_{s,l}$
%\otimes \mathbb R^ {N(s,l)}$
 and at most one subspace $L_i$ is isomorphic to $\mathcal L_m$,
 %(the latter occurs at most in one occasion)
 where spaces $X_{s,l}$ and $\mathcal L_m$ are as in \eqref{tuple1} and \eqref{Lm}, respectively. The restriction of $\rho$ to $L_i$ is isomorphic to
 %the direct sum of $N(s,l)$ copies of
 the indecomposable symbols $\mathbb R\delta_{s,l}$ in the first case and to $\mathbb R\tau_m$ in the second case.
% subspace corresponding to an indecomposable component of the symbol $\rho$
%(i.e. the subspace of $X$  where this indecomposable components acts).
%is a graded symplectic
%%(orthogonal)
%space and nice symplectic
%%(orthogonal)
%$\sll_2$-modules (with all structures induced from $V$).
Further, the restriction $\sigma|_{L_i}$ to $L_i$ of the symplectic form $\sigma$ of $V$
%%($a|_{L_i}$)
defines natural identification between $L_i$ and its dual space $L_i^*$.
Here $\sigma$ denotes the symplectic form on $X$.
%%in the symplectic case and $a$ denotes the non-degenerate symmetric form on $V$ in the orthogonal case.
Splitting \eqref{sympdecomp} defines the following (in general non-canonical) splitting
of $\gl(X)$ :
\begin{equation}
\label{glsplit}
\gl(X)=\bigoplus_{i,j=1}^k\Hom(L_i, L_j).
\end{equation}
An endomorphism  $A\in\gl(V)$, having the decomposition $A=\sum_{i,j=1}^k A_{ij}$ with respect to the splitting \eqref{glsplit},
belongs to $\mathfrak{sp}(X)$
%($\mathfrak {so}(V)$)
if and only if $A_{ii}\in \mathfrak{sp}(L_i)$ for all $1\leq i\leq q$ and $A_{ij}=-A_{ji}^*$ for all $1\leq i\neq j\leq q$ (here the dual linear map
$A_{ji}^*$ is considered as a map from $L_i$ to $L_j$ under the aforementioned identification $L_i\sim L_i^*$ and $L_j\sim L_j^*$).
Therefore the map $A\mapsto \sum_{i=1}^k A_{ii}+\sum_{1\leq i<j\leq q}A_{ij}$ defines the identification
\begin{equation}
\label{idsp2}
\mathfrak{sp} \left(X
%\bigoplus_{i=1}^k L_i
\right)\cong \bigoplus_{i=1}^k \mathfrak{sp}(L_i)\oplus\bigoplus_{1\leq i<j\leq q}\Hom(L_i,L_j).
\end{equation}
%in the symplectic case and the identification
%\begin{equation}
%\label{idsp2}
%\mathfrak{so} \left(\bigoplus_{i=1}^k L_i\right)\cong \bigoplus_{i=1}^k \mathfrak{so}(L_i)\oplus\bigoplus_{1\leq i<j\leq %q}\Hom(L_i,L_j).
%\end{equation}
%in the orthogonal case.
This implies that
 \begin{equation}
\label{idsp2pos}
\mathfrak{l}
%^{\mathfrak{sp}}
\left(X\right)\cong \bigoplus_{i=1}^k \mathfrak{l}
%^{\mathfrak{sp}}
(L_i)\oplus\bigoplus_{1\leq i<j\leq q}\mathfrak n(L_i,L_j).
\end{equation}
%\begin{equation}
%\label{idso2pos}
%\mathfrak{k} \left(\bigoplus_{i=1}^k L_i\right)\cong \bigoplus_{i=1}^k \mathfrak{k}(L_i)\oplus\bigoplus_{1\leq i<j\leq q}\mathfrak %n(L_i,L_j),
%\end{equation}
%where $\mathfrak n(L_i,L_j)$ is as in the previous subsection.

Note that by constructions the space $\mathfrak{l}
%^{\mathfrak{sp}}
\left(X\right)$ is an invariant subspace with respect to the projections on each component of the decomposition \eqref{idsp2}. Therefore the question of the existence of rank $1$ elements in $\mathfrak l
%^{\mathfrak{sp}}
\left(X\right)$  can be reduced to the same question regarding each component of the decomposition \eqref{idsp2pos}.
Moreover,  let us show that in the question of existence of rank $1$ elements in $\mathfrak l
%^{\mathfrak{sp}}
\left(X\right)$ one can get rid of multiple copies of $X_{s,l}$ in the splitting \eqref{sympdecomp} of $X$.

\begin{lemma}
\label{tensoroff}
The algebra  $\mathfrak l
%^{\mathfrak{sp}}
\left(X\right)$ contains rank $1$ elements if and only if at least one of the following algebras does:
\begin{enumerate}
\item $\mathfrak n \left(X_{s_1,l_1},X_{s_2,l_2} \right)$, if $X_{s_1,l_1}$ and $X_{s_2,l_2}$ appear as components in the splitting \eqref{sympdecomp};
\item  $\mathfrak n \left(\mathcal L_m,X_{s,l} \right)$, if $\mathcal L_m$ and $X_{s,l}$ appear as components in the splitting \eqref{sympdecomp};
\item $\mathfrak l
%^{\mathfrak{sp}}
\left(X_{s,l}\right)$, if $X_{s,l}$ appears as a component in the splitting \eqref{sympdecomp};
\item $\mathfrak l
%^{\mathfrak{sp}}
\left(\mathcal L_m\right)$, if $\mathcal L_m$ appears as a component in the splitting \eqref{sympdecomp};

\end{enumerate}
\end{lemma}
\begin{proof}
Applying formula \eqref{idsp2} to the graded conformal symplectic space $X_{s,l}\otimes\mathbb R^{N(s,l)}$ instead of $X$ it is easy to get the following natural identification:

\begin{equation}
\label{idspkN}
\mathfrak{sp}(X_{s,l}\otimes\mathbb R^{N(s,l)})\cong \Bigl(\mathfrak{sp}(X_{s,l})\otimes \mathbb R^{N(s,l)}\Bigr)\oplus\Bigl(\mathfrak{gl}(X_{s,l})\otimes\wedge^2\mathbb R^{N(s,l)}\Bigr),
\end{equation}
where $\wedge^2\mathbb R^{N(s,l)}$ denotes a skew-symmetric square of $\mathbb R^{N(s,l)}$.
Consequently,
 \begin{equation}
\label{idspkNpos}
\mathfrak{l}
%^\mathfrak{sp}
(X_{s,l}\otimes\mathbb R^{N(s,l)})\cong \Bigl(\mathfrak{l}
%^\mathfrak{sp}
(X_{s,l})\otimes \mathbb R^{N(s,l)}\Bigr)\oplus\Bigl(\mathfrak{n}(X_{s,l})\otimes\wedge^2 \mathbb R^{N(s,l)}\Bigr).
 \end{equation}

Besides, for any two nice $\sll_2$-modules $X_1$ and $X_2$
\begin{equation}
\label{tensorgl}
\mathfrak n(X_1\otimes\mathbb R^{N_1},X_2\otimes\mathbb R^{N_2})\cong
\mathfrak n(X_1, X_2)\otimes {\rm Hom}\Bigl(\mathbb R^{N_1},
 \mathbb R^{N_2}\Bigr).
\end{equation}

Combining \eqref{idspkNpos} and \eqref{tensorgl} with \eqref{idsp2pos}, one obtains the statement of the lemma.
\end{proof}

Now analyze the algebras appearing in Lemma \ref{tensoroff} separately:

\medskip
{\bf a) Algebras from the first two items of Lemma \ref{tensoroff}}

The calculation of algebras $\mathfrak n \left(X_{s_1,l_1},X_{s_2,l_2} \right)$ and   $\mathfrak n \left(\mathcal L_m,X_{s,l} \right)$ from  the first two items of Lemma \ref{tensoroff} is reduced to the calculation of $\mathfrak n(Y_1, Y_2)$ where $Y_1$  and $Y_2$ are two disjoint irreducible nice $\sll_2$-modules graded by the sets $I_1$ and $I_2$, where each $I_i$ is either $\mathbb Z$ or $\cfrac{1}{2} \mathbb Z_{odd}$ (note that $I_1$ and $I_2$ might be different). For each $i\in\{1,2\}$, let $r_i$ and $s_i$ be the lowest and the highest weights of element of $Y_i$ respectively and let $l_i=s_i-r_i$
%or  as in \eqref{indexset}.
($=\dim\, Y_i-1$). Recall that  the corresponding $\sll_2$-submodule $\Hom(Y_1, Y_2)=Y_2\otimes Y_1^*$ is decomposed into the irreducible $\sll_2$-modules as follows:
\begin{equation}
\label{Vrsrs}
\Hom(Y_1, Y_2)\cong\bigoplus_{i=0}^{\min\{
%\frac{l_1+l_2-|l_1-l_2|}{2}
l_1,l_2\}}\Pi_{l_1+l_2-2i},
%\Pi_{(s_1-r_1)+(s_2-r_2)+1}\oplus \Pi_{(s_1-r_1)+(s_2-r_2)-1}\bigoplus\Pi_{} ,$
\end{equation}
where $\Pi_j$ denotes an irreducible $\sll_2$-module of dimension $j+1$ (see, for example, \cite{Fulton}). Note that the following lemma immediately follows from the construction of the decomposition \eqref{Vrsrs}

\begin{lemma}
\label{larg1}
The only submodule in the decomposition \eqref{Vrsrs} containing rank 1 elements is the submodule $\Pi_{l_1+l_2}$ of the largest dimension.
\end{lemma}
Furthermore, the following formula can be easily verified (see \cite[subsection 8.2]{flag2}) using the decomposition \eqref{Vrsrs}:

\begin{equation}
\label{Y1Y2}
\mathfrak n(Y_1,Y_2)=\begin{cases}
\displaystyle{\bigoplus_{i\in\mathbb Z\cap\bigl[\max\{0, s_1-r_2\},\min\{
%frac{l_1+l_2-|l_1-l_2|}{2}
l_1,l_2
%,
%s_2-r_1
\}\bigr]}}
\Pi_{l_1+l_2-2i}, & s_2\geq s_1 \text { and } r_2\geq r_1 \\
\hspace{1in} 0, & \text{otherwise.}
\end{cases}
\end{equation}

>From this formula and Lemma \ref{larg1} it follows that $\mathfrak n(Y_1,Y_2)$ contains elements of rank $1$ if and only if $\mathfrak n(Y_1,Y_2)=\Hom(Y_1, Y_2)$ or, equivalently, all elements of $\Hom(Y_1, Y_2)$ have nonnegative degrees, or, equivalently, $s_1\leq r_2$, i.e. the maximal degree of elements of $Y_1$ does not exceed the minimal degree of elements of $Y_2$. Consider the skew Young diagram consisting of two rows with $l_1+1$ and $l_2+1$ boxes, respectively, such that the columns of it are indexed by the set $I$ in the decreasing order from left to right and for any $i\in I_1\cup I_2$ the number of boxes in the column indexed by $i$ is equal to the dimension of the subspace of degree $i$ in $Y_1\oplus Y_2$. Then the conditions above are equivalent to the fact that the obtained skew Young diagram does not satisfy condition~(1) of Theorem \ref{finitetype1}.

\medskip
{\bf b) Algebras from item (3) of Lemma \ref{tensoroff}}
Now consider the algebra $\mathfrak l
%^{\mathfrak{sp}}
\left(X_{s,l}\right)$ from item (3) of Lemma \ref{tensoroff}. The space $X_{s,l}$  is the sum of two irreducible nice  $\sll_2$-submodules
\begin{equation}
\label{EFrs}
E_{s,l}
%^{\mathfrak{sp}}
=\Span \{e_i\}_{s-l\leq i\leq s},\quad F_{s,l}
%^{\mathfrak{sp}}
=\Span \{f_i\}_{-s\leq i\leq l-s},
\end{equation}
where $ \{e_i\}_{s-l\leq i\leq s}$ and  $\{f_i\}_{=s\leq i\leq l-s}$ are as in \eqref{tuple1}.
Note that $E_{s,l}$ and $F_{s,l}$ are two transversal Lagrangian subspaces of $X_{s,l}$.
Therefore the symplectic form on $X_{s,l}$ defines the following natural identifications:
%$\sll_2$-modules
 $E_{s,l}
%^{\mathfrak{sp}}
\cong (F_{s,l}
%^{\mathfrak{sp}})
)^*$ and
$F_{s,l}
%^{\mathfrak{sp}}
\cong(E_{s,l}
%^{\mathfrak{sp}})
)^*$.
%($E_{s;l}^{\mathfrak{so}}\cong (F_{s;l}^{\mathfrak{so}})^*$ and $F_{s;l}^{\mathfrak{so}}\cong(E_{s;l}^{\mathfrak{so}})^*$).
Keeping in mind these identifications, an endomorphism $A\in\gl(V)$, which has the decomposition $A=A_{11}+A_{22}+A_{12}+A_{21}$ with respect to the splitting
\begin{equation}
\label{glXsl}
\gl(X_{s,l})=\gl(E_{s,l}
%^{\mathfrak{sp}}
)\oplus\gl(F_{s,l}%^{\mathfrak{sp}}
)\oplus\Hom(F_{s,l}%^{\mathfrak{sp}}
,E_{s,l}
%^{\mathfrak{sp}}
)\oplus \Hom(E_{s,l}
%^{\mathfrak{sp}}
,F_{s,l}%^{\mathfrak{sp}}
),
\end{equation}
 belongs to $\mathfrak{sp}\bigl(X_{s,l}\bigr)$ if and only if

\begin{equation}
\label{spcond}
A_{12}^*=A_{12}, \,A_{21}^*=A_{21}, \text{ and } A_{22}=-A_{11}^*.
\end{equation}
%It belongs to $\mathfrak{so}\Bigl(V_{s;l}^{\mathfrak{so}}\Bigr)$ if and only if $A_{12}^*=-A_{12}$, $A_{21}^*=-A_{21}$, and $A_{22}=-A_{11}^*$.
Therefore the map $F_{s,l}:A\mapsto A_{11}+A_{22}+A_{12}$ defines the following identifications
\begin{equation}
\label{idsp1}
\mathfrak{sp}(X_{s,l})\cong \gl(E_{s,l}
%^{\mathfrak{sp}}
)\oplus S^2(E_{s,l}
%^{\mathfrak{sp}}
)\oplus S^2 (F_{s,l}
%^{\mathfrak{sp}}
).
\end{equation}
Here $S^2(E_{s,l}
%^{\mathfrak{sp}}
)$
%and $\wedge^2(E_{s;l}
%^{\mathfrak{sp}}
%)$
is a subspace of
$\Hom(F_{s,l}
%^{\mathfrak{sp}}
, E_{s,l}
%^{\mathfrak{sp}}
)$ and  $S^2(F_{s,l}
%^{\mathfrak{sp}}
)$
%and $\wedge^2(F_{s;l}
%^{\mathfrak{sp}}
%)$
is a subspace of $\Hom(E_{s,l}
%^{\mathfrak{sp}}
, F_{s,l}
%^{\mathfrak{sp}}
)$. %$S^2(E_{rs}^\varepsilon)$  of $E_{rs}^\varepsilon$.
Keeping this in mind, we define $\mathfrak{s}(E_{s,l})
%{\mathfrak{sp},
%1
%}
$
%($\mathfrak{l}_{s;l}^{\mathfrak{so},1}$)
as the maximal $\sll_2$-submodule of
%$V_{r_2s_2}\otimes (V_{r_1s_1})^*$
$S^2(E_{s,l}
%^{\mathfrak{sp}}
)$
%($\wedge^2(E_{s;l}
%^{\mathfrak{so}}
%)$)
%considered as a subspace of  the  graded space
%$\Hom(F_{rs}^\varepsilon,E_{rs}^\varepsilon)\subset \gl(V_{rs}^\varepsilon)$
concentrated in the non-negative degree part of $S^2(E_{s,l}
%^{\mathfrak{sp}}
)$.
%($\wedge^2(E_{s;l}
%^{\mathfrak{so}}
%)$).
Similarly, let $\mathfrak{s}(F_{s,l})
%^
%{\mathfrak{sp},
%2
%}
$
%($\mathfrak{l}_{s,l}^{\mathfrak{so},2}$)
 be the maximal $\sll_2$-submodule of
%$V_{r_2s_2}\otimes (V_{r_1s_1})^*$
$S^2(F_{s,l}
%^{\mathfrak{sp}}
)$
%($\wedge^2(E_{s;l}
%^{\mathfrak{so}}
%)$)
concentrated in the non-negative degree part of $S^2(F_{s,l}
%^{\mathfrak{sp}}
)$.

Then from~\eqref{idsp1} it follows that
\begin{equation}\label{lrse1}
\mathfrak{l}
%^{\mathfrak sp}
(X_{s,l}
%^{\mathfrak{sp}}
)\cong \mathfrak n(E_{s,l})\oplus \mathfrak{s}(E_{s,l})\oplus \mathfrak{s}(F_{s,l})
 %^{
%\mathfrak{sp},
%1}
%\oplus \mathfrak{l}_{s,l}^{
%\mathfrak{sp},
%2}.
%\cong\mathbb K\oplus \mathfrak{l}_{s;l}^{\mathfrak{sp},1}\oplus %\mathfrak{l}_{s;l}^{\mathfrak{sp},2}\\
\end{equation}

Obviously, $\mathfrak{l}
%^{\mathfrak sp}
(X_{s,l}
%^{\mathfrak{sp}}
)$ contains an element of rank $1$ if and only if at least one of the spaces $\mathfrak{s}(E_{s,l})
%^{
%\mathfrak{sp},
%1}
$,
$\mathfrak{s}(F_{s,l})
%^{
%\mathfrak{sp},
%2}
$ or $F^{-1}\bigl(\mathfrak n(E_{s,l})\bigr)$ contains an element of rank $1$.
Now analyze each of this spaces.
%component  of \eqref{lrse1} on the presence of elements of rank $1$.
Start with  $\mathfrak{s}(E_{s,l})
%^{
%\mathfrak{sp},
%1
$.
By analogy with \eqref{Vrsrs}, $\sll_2$-module $S^2(E_{s,l}
%^{\mathfrak{sp}}
)$
%and $\wedge^2(E_{s;l}
%^{\mathfrak{so}}
%)$)
is decomposed into the irreducible $\sll_2$-modules in the following way (\cite{Fulton}):
\begin{equation}
%&~&
S^2(E_{s,l}
%^{\mathfrak{sp}}
)=\bigoplus_{i=0}^{
%\frac{l_1+l_2-|l_1-l_2|}{2}
[\frac{l}{2}]}\Pi_{2l-4i}.
\label{Vrseps}
%\\
%\Pi_{(s_1-r_1)+(s_2-r_2)+1}\oplus \Pi_{(s_1-r_1)+(s_2-r_2)-1}\bigoplus\Pi_{} ,$
%\end{equation}
%\begin{equation}
%&~&\label{Mrseps}
%\wedge^2(E_{s;l}
%%^{\mathfrak{so}}
%)=\bigoplus_{i=0}
%%\frac{l_1+l_2-|l_1-l_2|}{2}
%^{[\frac{l-1}{2}]}\Pi_{2l-2-4i}.
%\Pi_{(s_1-r_1)+(s_2-r_2)+1}\oplus \Pi_{(s_1-r_1)+(s_2-r_2)-1}\bigoplus\Pi_{} ,$
\end{equation}
Using the last decomposition, the following formula can be easily verified (see \cite[subsection 8.3.1]{flag2}):
\begin{equation}
\label{ngrseps}
\mathfrak{s}(E_{s,l})
%^{
%\mathfrak{sp},
%1}
=\bigoplus_{i=\max\{0,
%-r-[\frac{\varepsilon}{2}]
l-s\}}^{
%frac{l_1+l_2-|l_1-l_2|}{2}
[\frac{l}{2}]}\Pi_{2l-4i},
%\Pi_{(s_1-r_1)+(s_2-r_2)+1}\oplus \Pi_{(s_1-r_1)+(s_2-r_2)-1}\bigoplus\Pi_{} ,$
\end{equation}

 From this formula and the obvious analogues of Lemma \ref{larg1}
 %(that  will not be formulated here for shortness)
 it follows that $\mathfrak{l}_{s,l}^
%{\mathfrak{sp},
1
%}
$  contains element of rank $1$ if and only if $\mathfrak{s}(E_{s,l})
%^
%{\mathfrak{sp},
%1
%}
=S^2(E_{s,l}
%^{\mathfrak{sp}}
)$ or, equivalently, all elements of $S^2(E_{s,l}
%^{\mathfrak{sp}}
)$ have nonnegative degrees or ,equivalently, $l\geq s$, i.e. the  maximal degree of elements of $F_{s,l}$ does not exceed the minimal degree of elements of $E_{s,l}$. The latter condition in turn is equivalent to the fact that the corresponding skew Young diagram as in \eqref{youngodd}
%at most one column has two boxes.
%that
%the Consider the skew Young diagram consisting of two rows with $l_1+1$ and $l_2+1$ boxes, respectively, such that the columns of it are indexed by the set $I$ in the decreasing order from left to right and for any $i\in I$ the number of boxes in the column indexed by $i$ is equal to the dimension of the subspace of degree $i$ in $Y_1\oplus Y_2$. Then
%The latter condition is equivalent to the fact that this Young diagram
does not satisfy condition (1) of Theorem \ref{finitetype1}.

Further, the $\sll_2$-module $S^2(F_{s,l})$ has the decomposition to irreducible submodules similar
%^{\mathfrak{sp}}
to \eqref{Vrseps}. However, the grading on these irreducible submodules is different from the case of $S^2(E_{s,l})$, because it is induced from   $\Hom(E_{s,l}
%^{\mathfrak{sp}}
, F_{s,l}
%^{\mathfrak{sp}}
)$ instead of  $\Hom(F_{s,l}
%^{\mathfrak{sp}}
, E_{s,l}
%^{\mathfrak{sp}}
)$. One can easily show \cite[subsetion 8.3.1]{flag2} that
\begin{equation}
\label{ngrseps2}
\mathfrak{s}(F_{s,l})
%^{
%\mathfrak{sp},
%2}
=\begin{cases} \mathbb R&\text{if \text{$l$ is even and }}
%\varepsilon=2,\,r=-s-2
s=\frac{l}{2},\\
0&\text{otherwise}.
\end{cases}
\end{equation}
and it contains a rank $1$ element if and only if $s=l=0$, i.e. when the corresponding Young diagram consists of one column and two rows. The latter condition trivially does not satisfy condition (1) of Theorem \ref{finitetype1}. Note that the first option in \eqref{ngrseps2} is equivalent to the fact that the corresponding Young diagram is rectangular.

Finally, note that from \eqref{Y1Y2} applied to the case $Y_1=Y_2$ it follows that
\begin{equation}
\label{Y1=Y2}
\mathfrak n(E_{s,l})\cong\Pi_0\cong\mathbb R \,\Id_{E_{s,l}}
\end{equation}
 and, taking into account the last relation of \eqref{spcond}, it follows that  the space $(F_{s,l})^{-1}\bigl(\mathfrak n(E_{s,l})\bigr)$  has no element of rank 1 or, equivalently, elements $A$  in $\mathfrak{l}
%^{\mathfrak sp}
(X_{s,l}
%^{\mathfrak{sp}}
)$ of the form $A=A_{11}+A_{22}$ with respect to decomposition \eqref{glXsl} are never of rank $1$.

\begin{remark}
\label{Zrem}
As the matter of fact $(F_{s,l})^{-1}\bigl(\mathfrak n(E_{s,l})\bigr)$ is spanned by the endomorphism $Z$ such that $Z|_{E_{s,l}}=\Id|_{E_{s,l}}$ and $Z|_{F_{s,l}}=-\Id|_{F_{s,l}}$.
\end{remark}

\medskip
{\bf c) Algebras from item (4) of Lemma \ref{tensoroff}}
Note that $\mathfrak l
%^{\mathfrak{sp}}
(\mathcal L_m
%^{\mathfrak{sp}}
)\subset \mathfrak n(\mathcal L_m
%^{\mathfrak{sp}}
%, \mathcal L_m
%^{\mathfrak{sp}}
)\cong \mathbb R\, \Id_{\mathcal L_m}$, but $\Id\notin \mathfrak{sp}(\mathcal L_m
%^{\mathfrak{sp}}
)$.
%In the same way $\mathfrak l^{\mathfrak{so}}(\mathcal L_m^{\mathfrak{so}}
%)\subset \mathfrak n(\mathcal L_m^{\mathfrak{so}}, \mathcal L_m^{\mathfrak{so}})\cong \mathbb K %\Id$, but $\Id\notin \mathfrak{so}(\mathcal L_m^{\mathfrak{so}})$
Therefore
%\begin{equation}
% \label{splm}
$ \mathfrak l
%^{\mathfrak{sp}}
(\mathcal L_m
%^{\mathfrak{sp}}
)=0$.
%\quad \mathfrak l^{\mathfrak{so}}(\mathcal L_m^{\mathfrak{so}})=0.
% \end{equation}

Combining results of items a), b), and c) above  with Lemma \ref{tensoroff} we obtain the  Proposition \ref{propL}.
$\Box$

{\bf 2. Subspaces $\mathfrak a
%(
%\delta_{\rm{mod}}
%)
$, $\mathfrak z
%(
%\delta_{\rm{mod}}
%)
$, and $\mathfrak p
%(
%\delta_{\rm{mod}}
%)
$.}
%{\bf 2. Rank $1$ elements in $\varphi_{\rho}(\sll_2)$}
If we consider the embedding $\varphi_{
%\rho
\delta_{\rm{mod}}}(\sll_2)$ of $\sll_2$ into $\mathfrak{sp}(X)$ appearing in the decomposition \eqref{semidir}, then  from the basic representation theory of $\sll_2$ it follows immediately that $\varphi_{\rho}(\sll_2)$ has a rank $1$ element if and only if the skew Young diagram of $\delta_{\rm{mod}}$ does not satisfy condition (2) of Theorem  \ref{finitetype1}, i.e. it  consists of one row with two boxes.  Note that this and the ``only if'' part of Proposition \ref{propL} implies the ``only if'' part of Theorem \ref{finitetype1}.
However, this together with the ``if'' part of Proposition \ref{propL} is not enough to finish the proof of the ``if'' part of Theorem \ref{finitetype1}.

For this goal, we need a decomposition of $\mathfrak r\bigl(\mathfrak u^F(\delta_{\rm{mod}})\bigr)$, which is finer than \eqref{semidir}.  Let $\mathfrak a
%%(\delta_{\rm{mod}})
$ be the subalgebra of $\mathfrak r\bigl(\mathfrak u^F(\delta_{\rm{mod}})\bigr)$ consisting of those elements of $\mathfrak r\bigl(\mathfrak u^F(\delta_{\rm{mod}})\bigr)$ for which any irreducible $\sll_2$-submodule of the $\sll_2$-module $X$ is an invariant subspace.
%Assume that the skew Young tableaux $\Tab(\delta_{\rm{mod}})$ of $\delta_{\rm{mod}}$ consist of $k$ rows.
%For every $i$, $1\leq i\leq k$, let $\widetilde L_i$ be
%%First, for each row of the skew Young tableaux $\Tab(\delta_{\rm{mod}})$ consider
%the subspaces spanned by the vectors appearing in the boxes of the $i$th row of $\Tab(\delta_{\rm{mod}})$.
%%can be represented as the disjoint union of the maximal rectangular subtableaux. For every such %rectangular subtableaux consider the linear span of vectors from the boxes of this subtableaux. In %the notation of the sentence after formula \eqref{sympdecomp} these subspaces are either the %direct sum of $N(s,l)$ copies of $E_{s,l}$ or  the direct sum of $N(s,l)$ copies of $F_{s,l}$ or %at most one copy of $\mathcal L_m$, where spaces $E_{s,l}$, $F_{s,l}$, and $\mathcal L_m$ are as %in \eqref{EFrs} and \eqref{Lm}, respectively. We denote this subspaces by $\widetilde L_i$, $q\leq %i\leq \tilde q$ for some $\tilde q\in\mathbb N$.
%Define the algebra  $\mathfrak a
%%(\delta_{\rm{mod}})
%$ as the maximal subalgebra of $\mathfrak r\bigl(\mathfrak u^F(\delta_{\rm{mod}})\bigr)$ consisting of those elements of $\mathfrak r\bigl(\mathfrak u^F(\delta_{\rm{mod}})\bigr)$ for which every such  subspace $\widetilde L_i$ is an invariant subspace. As a matter of fact, one can define the subalgebra $\mathfrak a$ without referring to the subspaces $\widetilde L_i$ (which are not canonically defined if in the decomposition $\delta_{\rm mod})$ to the indecomposable conformally symplectic symple
The embedding  $\varphi_{\delta_{\rm{mod}}}(\sll_2)$ clearly belongs to $\mathfrak a
%(\delta_{\rm{mod}})
$,
%$$\varphi_{\delta_{\rm{mod}}}(\sll_2)\subset \mathfrak a.$$
which also implies that $\mathfrak a$ is an $\sll_2$-submodule of  $\mathfrak r\bigl(\mathfrak u^F(\delta_{\rm{mod}})\bigr)$.
%%$\mathfrak{sp}(X)$.
%Also, let
%$$\mathfrak z
%%(\delta_{\rm{mod}})
%=\mathfrak a
%%(\delta_{\rm{mod}})
%\cap \mathfrak l
%%^{\mathfrak{sp}}
%\left(X\right).$$
%By construction, $\mathfrak z$ is an $\sll_2$-submodule of $\mathfrak l
%%^{\mathfrak{sp}}
%\left(X\right)$.
Therefore, there exists the unique $\sll_2$-module $\mathfrak p$ which is complementary to $\mathfrak a$ in  $\mathfrak r\bigl(\mathfrak u^F(\delta_{\rm{mod}})\bigr)$,
%$\mathfrak l
%^{\mathfrak{sp}}
%\left(X\right)$,
\begin{equation}
\label{ap}
\mathfrak r\bigl(\mathfrak u^F(\delta_{\rm{mod}})\bigr)=\mathfrak a\oplus\mathfrak p.
\end{equation}
Note that by the construction $\mathfrak p$ is a subspace of $\mathfrak l(X)$.
Also, let
%$$
\begin{equation}
\label{z}
\mathfrak z
%%(\delta_{\rm{mod}})
=\mathfrak a
%%(\delta_{\rm{mod}})
\cap \mathfrak l
%%^{\mathfrak{sp}}
\left(X\right).
\end{equation}
%$$
Then
\begin{equation}
\label{az}
\mathfrak l
%^{\mathfrak{sp}}
\left(X\right)=\mathfrak z\oplus\mathfrak p, \quad  \mathfrak a=\varphi_{\delta_{\rm{mod}}}(\sll_2)\oplus\mathfrak z.
\end{equation}
Let us describe the introduced subspaces $\mathfrak a
%(
%\delta_{\rm{mod}}
%)
$, $\mathfrak z
%(
%\delta_{\rm{mod}}
%)
$, and $\mathfrak p
%(
%\delta_{\rm{mod}}
%)
$ in more detail.
For this fix, as before, a skew Young tableaux  $\Tab(\delta_{\rm{mod}})$ of $\delta_{\rm{mod}}$. Assume that this tableaux consists of $\tilde k$ rows (recall that $\tilde k=\rank D-1$)  and let $\widetilde L_i$, $1\leq i\leq \tilde k$, be
%%First, for each row of the skew Young tableaux $\Tab(\delta_{\rm{mod}})$ consider
the subspaces of $X$  spanned by the vectors appearing in the boxes of the $i$th row of $\Tab(\delta_{\rm{mod}})$. We have the following decomposition:
%, as in \eqref{sympdecomp} and \eqref{glsplit}:
\begin{eqnarray}
\label{sympdecomp1}
&~& X=\displaystyle{\bigoplus_{i=1}^{\tilde k} \widetilde L_i},\\
%\end{equation}
%\begin{equation}
&~&\label{glsplit1}
\gl(X)=
\displaystyle{\bigoplus_{i=1}^{\tilde k} \gl(\widetilde L_i)}\oplus
\displaystyle{\bigoplus_{
%1\leq
i\neq j
%\leq \widetilde k
}\Hom(\widetilde L_i, \widetilde L_j).}
%\end{equation}
\end{eqnarray}
Denote by $\pi_1$ and $\pi_2$ the projections from  $\gl(X)$ to $\displaystyle{\bigoplus_{i=1}^{\tilde k} \gl(\widetilde L_i)}$ and to $\displaystyle{\bigoplus_{
%1\leq
i\neq j
%\leq \widetilde k
}\Hom(\widetilde L_i, \widetilde L_j)}$, respectively, with respect to decomposition \eqref{glsplit1}.
Then by constructions

\begin{equation}
\label{pi12}
\mathfrak z=\pi_1\bigl(\mathfrak l(X)\bigr), \quad \mathfrak p=\pi_2\bigl(\mathfrak l(X)\bigr).
\end{equation}
%are the images of $\mathfrak l(X)$ under the projection from to the first and the second components of the decomposition \eqref{glsplit1}, respectively.
%\begin{equation}
%\label{az}
%\mathfrak l
%%^{\mathfrak{sp}}
%\left(X\right)=\mathfrak z\oplus\mathfrak p.
%\end{equation}
%%. This submodule will be denoted by $\mathfrak p$
%Then
%\begin{equation}
%\label{ap}
%\mathfrak r\bigl(\mathfrak u^F(\delta_{\rm{mod}})\bigr)=\mathfrak a\oplus\mathfrak p
%\end{equation}

Now prove that if the skew Young diagram of $\delta_{\rm{mod}}$ satisfies condition (1) of Theorem  \ref{finitetype1}  and the algebra $\mathfrak r\bigl(\mathfrak u^F(\delta_{\rm{mod}})\bigr)$ contains an element of rank $1$, then $\mathfrak a$  must contain an element of rank $1$. Indeed, if $\mathfrak r\bigl(\mathfrak u^F(\delta_{\rm{mod}})\bigr)$ contains an element of rank $1$ with nonzero $\mathfrak p$ component with respect to the decomposition \ref{ap} then, using the second relation of \eqref{pi12} and the fact that the space $\mathfrak p$ is an invariant subspace of the projection from $\mathfrak{gl}(X)$ to  $\mathfrak{gl}(\widetilde L_i,\widetilde L_j)$ with respect to the decomposition \eqref{glsplit1}, the subspace $\mathfrak p$ contains a rank $1$ element, which contradicts Proposition \ref{propL}.

Finally, assume that there exists an element  $A\in \mathfrak a$ of rank 1. Since each subspace $\widetilde L_i$ is an invariant subspace of $A$ then there exist $i_0$, $1\leq i_0\leq\tilde k$ such that $A|_{\widetilde L_{i_0}}$ has rank $1$. Assume that $A$ splits as $A_1+A_2$ with  respect to the second decomposition in \eqref{az} so that  $A_1\in\varphi_{\delta_{\rm{mod}}}(\sll_2)$ and $A_2\in \mathfrak z$. On one hand, since $\widetilde L_{i_0}$ is an irreducible $\sll_2$-module, the spectrum of the endomorphism $A_1|_{\widetilde L_{i_0}}$ constitute an arithmetic progression centered at (symmetric with respect to) the origin. On the other hand, $A_2|_{\widetilde L_{i_0}}$ is either a multiple of the identity by the formula of type \eqref{Y1=Y2}, if $\widetilde L_{i_0}$ is indexed by $\mathbb Z$ (or, equivalently, $\widetilde L_{i_0}$ is a (Lagrangian) subspace of the domain of the indecomposable symbol of type $\mathbb R\delta_{s,l}$) or equal to zero by item c) of the proof of Proposition \ref{propL}, if $\widetilde L_{i_0}$ is indexed by $\cfrac {1}{2}\mathbb Z_{\rm odd}$ (or, equivalently, $\widetilde L_{i_0}$ is the domain of the indecomposable symbol of type $\mathbb R\tau_m$) . Therefore  $A|_{\widetilde L_{i_0}}$ may have rank $1$ if an only if
$\dim\,\widetilde L_{i_0}\leq 2$.

If $\dim \widetilde L_{i_0}=1$ then the Young diagram of $\delta_{\rm{mod}}$ contains two rows with one box each, which contradicts condition  (1) of Theorem  \ref{finitetype1}. If $\dim \widetilde L_{i_0}=2$, then there are two possibilities:

\begin{enumerate}
\item
$\widetilde L_{i_0}$ is indexed by $\mathbb Z$. Then $\Diag(\delta_{\rm{mod}})$ contains 2 rows with 2 boxes which contradicts either Remark \ref{middleodd} or  condition  (1) of Theorem  \ref{finitetype1}.

\item
$\widetilde L_{i_0}$ is indexed by $\cfrac {1}{2}\mathbb Z_{\rm odd}$. Let us prove that in this case $\Diag(\delta_{\rm{mod}})$  must consist of one row with two boxes (in contradiction to  condition  (2) of Theorem  \ref{finitetype1}). Otherwise, the condition $A|_{\widetilde L_{i}}=0$ for all $i\neq i_0$, $1\leq i\leq \tilde k$ holds and it is not void. This implies that  $A_1|_{\widetilde L_{i}}=0$ for all $i\neq i_0$, $1\leq i\leq \tilde k$, which yields that $A_1=0$. Besides,  $A_2|_{\widetilde L_{i}}=0$ for all $i\neq i_0$, $1\leq i\leq \tilde k$ and by item c) of the part 1 of the current proof $A_2|_{\widetilde L_{i_0}}=0$. Therefore $A_2=0$ and hence $A=0$, contradicting the assumption that $A$ has rank $1$.
\end{enumerate}
This completes the proof of ``if'' part of Theorem  \ref{finitetype1}.
\end{proof}

\section{Reduction of calculation of algebra $\mathfrak u(\eta, \mathfrak u^F(\delta_{\rm mod}))$ to standard prolongation of $\mathfrak p$}
\setcounter{equation}{0}
\setcounter{theorem}{0}
\setcounter{lemma}{0}
\setcounter{proposition}{0}
\setcounter{definition}{0}
\setcounter{corollary}{0}
\setcounter{remark}{0}
\setcounter{example}{0}
\label{pstandartsec}

Next natural task is to describe the algebra $\mathfrak u(\eta, \mathfrak u^F(\delta_{\rm mod}))$ as explicit as possible.

First of all consider the case of rank $2$ distributions, which is important as a building block for the case of  distributions of arbitrary even rank.

\begin{theorem}\label{rank2thm}
A $(2,n)$-distribution of maximal class has a modified Jacobi symbol $\delta_{\mod}$  isomorphic to $\mathbb R\tau_{\frac{n-2}{2}}$, i.e. its Young diagram consists  of one rows with $2(n-3)$ boxes,  the algebra $\mathfrak u^F(\delta_{mod})$ is equal to an embedding of an irreducible representation of $\gl_2$ into $\mathfrak{csp}(\eta^{-1})$ and the algebra $\mathfrak u(\eta, \mathfrak u^F(\delta_{mod}))$ satisfies the following:

\begin{enumerate}
\item \cite{yamag} If $n=5$, then $\mathfrak u(\eta, \mathfrak u^F(\delta_{\rm mod}))$ is isomorphic to the split real form of the exceptional Lie algebra $G_2$ (as also expected from the classical Cartan work \cite{cart10});

\item \cite{doubzel1, doubzel2} If $n>5$, then the first prolongation $\mathfrak u^1\bigl(\eta, \mathfrak u^F(\delta_{\rm mod})\bigr)$ of the pair $\bigl(\eta, \mathfrak u^F(\delta_{\rm mod})\bigr)$ vanishes, i.e. $\mathfrak u(\eta, \mathfrak u^F(\delta_{\rm mod}))$ is isomorphic to the semidirect sum of the above mentioned embedding of an irreducible representation of $\gl_2$ into $\mathfrak{csp}(\eta^{-1})$ with $\eta$.
\end{enumerate}
\end{theorem}
It turns out that for more general Jacobi symbols with sufficiently long rows of their Young diagrams the subspace $\mathfrak p$
introduced in the previous section (see the decomposition~\eqref{ap}) plays the crucial role in the description of the algebra $\mathfrak u(\eta, \mathfrak u^F(\delta_{mod}))$. First, note that since $\eta^{-1}$ ($\cong X$) generates $\eta$, for every $i\geq 0$ any element of  $i$th Tanaka algebraic prolongation $\mathfrak u^i\bigl(\eta, \mathfrak u^F(\delta_{\rm mod})\bigr)$ of the pair $\bigl(\eta, \mathfrak u^F(\delta_{\rm mod})\bigr)$ is uniquely determined by its restriction to $X$.

Let $\phi \in \mathfrak u^1\bigl(\eta, \mathfrak u^F(\delta_{\rm mod})\bigr)$. Then by definition for any $v_1,v_2 \in X\cong\eta^{-1}$ we have
\begin{equation}
\label{prolong1}
%$$
\phi(v_1)v_2-\phi(v_2)v_1=\phi([v_1, v_2]).
%$$
\end{equation}
Fix a symplectic form $\sigma$, representing the conformal symplectic structure on $X$. Then there exists the unique $y\in \eta^{-2}$ such that $[v_1, v_2]=\sigma(v_1, v_2) y$ for any $v_1, v_2$ in $X$. Let $v=\phi(y)$. Note that $v\in X$. Using this and relation \eqref{prolong1}, we obtain another characterization of  $\mathfrak u^1\bigl(\eta, \mathfrak u^F(\delta_{\rm mod}\bigr))$: $\phi \in \Hom(X, \mathfrak u^F(\delta_{\rm mod})\oplus \Hom(\eta^{-2}, X)$ belongs to $\mathfrak u^1\bigl(\eta, \mathfrak u^F(\delta_{\rm mod})\bigr)$ if and only if the restriction of $\phi$ to $X$ satisfies the following condition: there exists $v\in X$ such that for any $v_1, v_2\in X$

\begin{equation}
\label{prolong1new}
%$$
\phi(v_1)v_2-\phi(v_2)v_1=\sigma(v_1, v_2)v.
%$$
\end{equation}
On the other hand, by definition $\phi \in \Hom(X, \mathfrak u^F(\delta_{\rm mod}))$  belongs to the standard prolongation $u^F(\delta_{\rm mod})^{(1)}$ if and only if for any $v_1, v_2\in X$
\begin{equation}
\label{prolong1standard}
%$$
\phi(v_1)v_2-\phi(v_2)v_1=0.
%\sigma(v_1, v_2)v.
%$$
\end{equation}

Relations \eqref{prolong1new} and \eqref{prolong1standard} imply that the standard prolongation $\bigl(u^F(\delta_{\rm mod})\bigr)^{(1)}$ can be considered as the subspace of the space  $\mathfrak u^1\bigl(\eta, \mathfrak u^F(\delta_{\rm mod})\bigr)$ corresponding to $v=0$ in \eqref{prolong1new}. The same statement with the same proof is true for the higher order prolongations, i.e.

\begin{equation}
\label{standsymp}
\bigl(u^F(\delta_{\rm mod})\bigr)^{(i)}\subset \mathfrak u^i\bigl(\eta, \mathfrak u^F(\delta_{\rm mod})\bigr), \quad i>0.
\end{equation}
\begin{theorem}\label{genpr_i0}

Assume that the skew Young diagram of the modified Jacobi symbols $\delta_{\rm mod}$ satisfies both conditions (1) and (2) of Theorem  \ref{finitetype1} and in addition the following two conditions hold:
\begin{enumerate}
\item[(3)] The number of boxes in every row indexed by $\mathbb Z$ is at least  $4$;

\item[(4)] The number of boxes in the row indexed by $\cfrac{1}{2}\mathbb Z_{\rm odd}$ (if it exists) is at least $6$.
\end{enumerate}
Then for any $i>0$ the $i$th  Tanaka algebraic prolongation $\mathfrak u^i\bigl(\eta, \mathfrak u^F(\delta_{\rm mod})\bigr)$ of the pair $\bigl(\eta, \mathfrak u^F(\delta_{\rm mod})\bigr)$ coincides with the
standard $i$th prolongation $\pg^{(i)}$ of the subspace $\pg$ of $\mathfrak{gl}(X)$ defined in \eqref{pi12}.
\end{theorem}
\begin{remark}
Note that condition (3) in Theorem \ref{genpr_i0} is equivalent to the fact that in the decomposition of $\delta_{\rm mod}$ into the indecomposable symbols all components of type $\mathbb R\delta_{s,l}$ satisfy $l\geq 3$ and condition (4) is equivalent to the fact that in the same decomposition the component $\mathbb R\tau_m$ , if it exists, satisfies $m>\cfrac{3}{2}$. Also, conditions (1) and (3) are equivalent to  condition (1) and the fact that in the decomposition of $\delta_{\rm mod}$ into the indecomposable symbols there is no symbol of the type $\mathbb R\delta_{1,2}$, i.e. the indecomposable symbol with rectangular Young diagram of size $2\times 3$.
\end{remark}

Before proving Theorem \ref{genpr_i0}, note that since $\mathfrak p\subset \mathfrak l(X)$, one has that $\mathfrak p^{(i)}\subset \bigl(\mathfrak l(X)\bigr)^{(i)}$. On the other hand, from \eqref{standsymp} it follows that $\bigl(\mathfrak l(X)\bigr)^{(i)}\subset \mathfrak u^i\bigl(\eta, \mathfrak l(X))\bigr)\subset \mathfrak u^i\bigl(\eta, \mathfrak u^F(\delta_{\rm mod})\bigr)$, which implies the following

\begin{corollary}
\label{corl(X)}
Under assumptions of Theorem \ref{genpr_i0} for any $i>0$ the $i$-th  Tanaka algebraic prolongation $\mathfrak u^i\bigl(\eta, \mathfrak u^F(\delta_{\rm mod})\bigr)$ of the pair $\bigl(\eta, \mathfrak u^F(\delta_{\rm mod})\bigr)$ coincides with the
standard $i$th prolongation $\bigl(\mathfrak l(X)\bigr)^{(i)}$ of the algebra $\mathfrak l(X)$ of the infinitesimal symmetries of the flat curve $\mathcal F_{\delta_{\rm mod}}$ with flag symbol $\delta_{\rm mod}$ considered as a parametrized curve.
%$\mathfrak{gl}(X)$ defined in \eqref{pi12}.
\end{corollary}

\textbf{Proof of Theorem \ref{genpr_i0}.}
First, let us prove our theorem for $i=1$, i.e. for the first prolongation.

 Our theorem will be proved if we will show  that for such $\phi$ we have that $\phi(v_1)\in \mathfrak{p}$ for any $v_1\in X$ and that $v=0$.

 We can always assume that the Young subdiagram of the indecomposable component of type $\mathbb R\tau_m$ in the diagram $\Diag(\delta_{\rm mod})$, if exists, is situated at the very bottom. Fix again a skew Young tableaux $\Tab(\delta_{\rm mod})$ and let $r=\left[\frac {1}{2}(\rank\overline{} D-1)\right]$.
For any $i$ such that $1\leq i\leq r$ let $E_i$ be the span of the vectors located in the boxes of the $(2i-1)$-st row of  $\Tab(\delta_{\rm mod})$ and $F_i$ be the span of the vectors located in the boxes of the $2i$-th rows of  $\Tab(\delta_{\rm mod})$. Finally, if rank $D$ is even, let $\mathcal L$ be the span of the vectors located in the last box of  $\Tab(\delta_{\rm mod})$.

Let us fix a basis in the subspace $\mathfrak z$ defined by \eqref{z}. For any $i$, $1\leq i\leq r$, let $Z_{i}\in\mathfrak{csp}(X)$ such that $Z_i|_{E_i}=\Id|_{E_i}$, $Z_i|_{F_i}=-\Id|_{F_i}$, and $Z_i$ equal to zero on all vectors of $\Tab(\delta_{\rm mod})$ that are not located in the boxes of the $(2i-1)$-st and $2i$-th row of $\Tab(\delta_{\rm mod})$. Then using Remark \ref{Zrem} and item c) of the proof of Proposition \ref{propL} one can show that
\begin{equation}
\label{Zspan}
\mathfrak z=\Span \{Z_i\}_{i=1}^r.
\end{equation}
%Let $(\phi,v)$ be any element in $(\s_{k,l})^{(1m)}$. Let us note
%that the space $\ag+\pg$ is in fact a subalgebra in $\gl(V)$
%preserving the subspace $V_e\subset V$. The restriction of this
%subalgebra to $V_e$ is 4-dimensional and is generated by the
%elements $X_{V_e}$, $H_{V_e}$, $Y_{V_e}$ and $(Z_1)_{V_e}$. In
%particular, it is equal to the image of the irreducible embedding of
%$\gl(2,\R)$ into $\gl(V_e)$.
%
%Let $(\phi,v)\in (\s_{k,l})^{(1m)}$. Let us prove that $v=0$ and
%$\phi$ takes values in $\pg$. Consider the equation~\eqref{genprol}
%in the following cases.
%
Now assume that  $v_1,v_2\in E_i$. Let $\overline{E_i}$ be the span of all vectors
located in the boxes of  $\Tab(\delta_{\rm mod})$ which are not in the $(2i-1)$st row and let $\Pi_i$ be the projection from $X$ to $E_i$ parallel to $\overline{E_i}$. Let $\Pi_i\circ \mathfrak u^F(\delta_{\rm mod})=\{\Pi_i\circ A: A\in \mathfrak u^F(\delta_{\rm mod}\bigr)\}$.
Consider the subspace $\Bigl(\Pi_i\circ \mathfrak u^F(\delta_{\rm mod})\Bigr)|_{E_i}$ of $\mathfrak {gl}(E_i)$  such that it consists of the restrictions of elements of $\Pi_i\circ \mathfrak u^F(\delta_{\rm mod})$ to $E_i$. Then from the second identity of \ref{az} and Remark \ref{Zrem} it follows that
%The restriction of this
%subalgebra to $V_e$
the space $\Bigl(\Pi_i\circ \mathfrak u^F(\delta_{\rm mod})\Bigr)|_{E_i}$
is $4$-dimensional and it is generated by the restrictions of elements of $\varphi_{\delta_{\rm mod}}$ to $E_i$ and $\Id|_{E_i}$, i.e.
%elements $X_{V_e}$, $H_{V_e}$, $Y_{V_e}$ and $(Z_1)_{V_e}$. In
%particular,
it is equal to the image of the irreducible embedding of
$\gl_2(\mathbb R)$ into $\gl(E_i)$.

 As the restriction of the symplectic
form to $E_i$ vanishes, we get $\phi(v_1)v_2=\phi(v_2)v_1$. Considering this equation $\mod \overline{E_i}$ we get exactly
the equation for the standard prolongation of $\Bigl(\Pi_i\circ \mathfrak u^F(\delta_{\rm mod})\Bigr)|_{E_i}\cong \gl_2(\mathbb R)$. According to the result of Kobayashi--Nagano~\cite{kobnag}, the first
prolongation of the irreducible embedding of $\gl_2(\R)$ is non-zero only if
the dimension of the representation space does not exceed~$3$, which contradicts assumption (3)
of our theorem.
%In our case the
%lowest possible dimension of $V_e$ is achieved when $k=2,l=1$ and is equal
%to~$4$.
Thus, we see that $\phi(v_1)v_2=0\,\,\mod
\overline E_i$ for all $v_1,v_2\in E_i$. In particular, this implies that
\begin{equation}
\label{phiEi}
\phi(E_i) \in \pg+\R(Z_i+\Id)+\sum_{j\neq i} Z_j.
\end{equation}
Absolutely analogous analysis of the cases $v_1,v_2\in F_i$  implies that
\begin{equation}
%&~&
\phi(F_i)\in\pg+\R(Z_i-\Id)+\sum_{j\neq i} Z_j. \label{phiFi}
%\\
\end{equation}

In the case when $v_1,v_2\in \mathcal L$, using similar arguments, one can get
from condition (4) of the theorem under consideration and item (2) of Theorem \ref{rank2thm}  that

\begin{equation}
%&~&
\phi(\mathcal L)\in \pg+\sum_{j=1}^r Z_j \label{phiL}
%\end{eqnarray}
\end{equation}
%\end{proof}
>From relations \eqref{phiEi}-\eqref{phiFi} it follows that if $x\in X$ decomposes as $x=\sum_{j=1}^r (x_{E_j}+x_{F_j})+x_\mathcal L$, where $x_{E_j}\in E_j$, $x_{F_j}\in F_j$, and $x_{\mathcal L}\in \mathcal L$, then there exist functionals  $\alpha_i\in X^*$, $1\leq i\leq r$, such that

\begin{equation}
\label{phipr}
\phi(x)=\phi'(x)+\sum_{i=1}^r\bigl(\alpha_i(x_{E_i})(Z_i-\Id)+ \alpha_i(x_{F_i})(Z_i+\Id)\bigr)+\sum_{i=1}^r\alpha_i\Bigl(\sum_{j\neq i}(x_{E_j}+x_{F_j})+x_{\mathcal L}\bigr)Z_i,
\end{equation}
where $\phi'$ takes values in $\pg$.
Now assume that $v_1=v_{E_i}\in E_i$ and $v_2=v_{F_i}\in F_i$.
Then by \eqref{prolong1new} we get:
\begin{equation}
\label{vevf}
\phi(v_{E_i})v_{F_i} - \phi(v_{F_i})v_{E_i} = \sigma(v_{E_i},v_{F_i}) v.
\end{equation}
 Also let $\overline{E_i\oplus F_i}=\overline {E_i}\cap\overline {F_i}$.
Consider equation \eqref{vevf} modulo $\overline{F_i}$.
\begin{enumerate}
\item
\emph{Assume that the subdiagram of $\Diag(\delta_{\mod})$ consisting of the $(2i-1)$st and $2i$th row
%of it
is not rectangular.
%for any $1\leq i\leq r$
} Then by \eqref{ngrseps2}, the second line there, the restriction of an element of $\pg$
on $E_i$ takes values in $\overline{E_i\oplus F_i}$. Taking into account this fact, the fact that an element of $\pg$ sends $F_i$ to $\overline {F_i}$ , and  relation \eqref{phipr}, we get from  \eqref{vevf} that
$$
-2\alpha_i(v_{E_i})v_{F_i}=\sigma(v_{E_i},v_{F_i})v \,\,\mod \overline{F_i}, \quad\text{for any }v_{E_i}\in E_i,
v_{F_i}\in F_i.
$$
As the dimensions of $E_i$ and $F_i$ are at least 4 by assumption (3) of the theorem, for any $v_{E_i}\in E_i$ we can find
a non-zero vector $v_{F_i}\in F_i$ such that $\sigma(v_{E_i},v_{F_i})=0$. Hence, we see
that
\begin{equation}
\label{aE}
\alpha_i|_{E_i}=0.
\end{equation}
In particular, we also see that $v\in \overline {F_i}$.

\item
\emph{Now assume that the subdiagram of $\Diag(\delta_{\mod})$ consisting of the $(2i-1)$st and $2i$th row
%of it
is rectangular.
%for any $1\leq i\leq r$
} The difference with the previous case is that
%n by \eqref{ngrseps2} (the second row there)
the restriction of an element of $\pg$
on $E_i$ takes values in $\overline{E_i}$ but not in $\overline{E_i\oplus F_i}$ in general, because the considered case corresponds to the first line of \eqref{ngrseps2}.
Assume that the considered rectangular diagram has size $2\times  (2s+1)$, the boxes of the upper row of the corresponding subtableaux are filled by vectors $e_{s}, e_{s-1},\ldots, e_{-s}$ from the left to the right, and the boxes of the lower row of the corresponding subtableaux are filled by vectors $f_{s}, f_{s-1},\ldots, f_{-s}$ from the left to the right. Recall that by constructions
\begin{equation}
\label{sigmaeq}
\sigma(e_j,e_k)=\sigma(f_j,f_k)=0, \quad \sigma(e_j,f_{-j})=(-1)^{s-j}
\end{equation}
for all integers  $j,k$ such that $-s\leq j,k\leq s$.
Then from the first line of \eqref{ngrseps2} it follows that there exists $\beta_i\in F_i^*$ such that
\begin{equation}
\label{betadef}
\forall -s\leq k\leq s, v_{F_i}\in F_i, \quad \phi'(v_{F_i})e_k=\beta_i(v_{F_i})f_k\quad \mod \overline{E_i\oplus F_i}.
\end{equation}

Taking into account this fact, the fact that an element of $\pg$ sends $F_i$ to $\overline {F_i}$, and the relation \eqref{phipr},  we get after substitution of $v_{E_i}=e_k$ into \eqref{vevf} that:
\begin{equation}
\label{rect1}
-\beta_i(v_{F_i})f_k-2\alpha_i(e_k)v_{F_i}=\sigma(e_k,v_{F_i})v \,\,\mod \overline{F_i}, \quad\text{for any }
v_{F_i}\in F_i, -s\leq k\leq s.
\end{equation}

Taking  $v_{F_i}=f_j$ with $j\neq -k$ and using \eqref{sigmaeq} we obtain
\begin{equation}
\label{rect2}
-\beta_i(f_j)f_k-2\alpha_i(e_k)f_j=0 \,\,\mod \overline{F_i}.
\end{equation}

Since $s>1$ by assumption (3)  of Theorem \ref{genpr_i0},  we can always take $j\neq -k$ such that in addition $j\neq k$. Using this $j$ in \eqref{rect2} we get  that $\alpha(e_k)=0$ . As the latter relation is valid for arbitrary $k$ such that $-s\leq k\leq s$, the relation \eqref{aE} holds in the considered case as well. Moreover, combining \eqref{aE} and \eqref{rect2} one also gets that
$\beta_i=0$, which together with \eqref{betadef} implies that
\begin{equation}
\label{beta}
{\rm Im}\bigl(\phi'(v_{F_i})|_{E_i}\bigr)=0 \quad \mod \overline{E_i\oplus F_i}.
\end{equation}
%As the dimensions of $E_i$ and $F_i$ are at least 4 by assumption (3) of the theorem, for any $v_{E_i}\in E_i$ we can find
%a non-zero vector $v_{F_i}\in F_i$ such that $\sigma(v_{E_i},v_{F_i})=0$. Hence, we see
%that
%\begin{equation}
%\label{aE}
%\alpha_i|_{E_i}=0.
%\end{equation}
%In particular, we also see that $v\in \overline {F_i}$.
\end{enumerate}
Let us now fix an arbitrary $v_{E_i}\in E_i$. From~\eqref{phipr} and \eqref{aE} we have
$\phi(v_{E_i})=\phi'(v_{E_i})\,\,\mod \overline{E_i\oplus F_i}$.  Let us prove that $\phi'(v_{E_i})=0\,\,\mod \overline{E_i\oplus F_i} $. Indeed, by \eqref{phipr}, \eqref{vevf}, and \eqref{beta}, when the subdiagram of $\Diag(\delta_{\mod})$ consisting of the $(2i-1)$st and $2i$th row
%of it
is rectangular,  we have
\begin{equation}
\label{rank2}
\phi'(v_{E_i})v_{F_i} = \sigma(v_{E_i},v_{F_i})v + 2\alpha_i(v_{F_i})v_{E_i}\,\, \,\,\mod\overline{E_i\oplus F_i}, \quad\text{for any }v_{F_i}\in F_i.
\end{equation}
The splitting
\begin{equation}
\label{splitEFL}
X=\begin{cases}
\displaystyle{\left(\bigoplus_{i=1}^r (E_i\oplus F_i)\right)}& \text{if }\rank D=2r+1,\\
\displaystyle{\left(\bigoplus_{i=1}^r (E_i\oplus F_i)\right)\oplus \mathcal L}& \text{if } \rank D=2r+2
\end{cases}
\end{equation}
 induces the natural splitting of $\gl(X)$ as in \eqref{glsplit}. Let $\mathfrak P_i$ be the projection from $\gl(X)$ to $\gl(F_i, E_i)$ with respect to the latter splitting.
Equation \eqref{rank2} and the fact that $\phi'(v_{E_i})v_{F_i}$ takes values in $\overline{F_i} $ imply that the rank of the linear map $\mathfrak P_i \bigl(\phi'(v_{E_i})\bigr)$ is not greater than $2$. On one hand,  $\mathfrak P_i \bigl(\phi'(v_{E_i})\bigr)$ must belongs to $\mathfrak s (E_i)\subset S^2(E_i)$ considered as a subspace of $\Hom(F_i, E_i)$ as in item b) of the proof of Proposition \ref{propL}.

On the other hand,  from the condition (1) of Theorem \ref{finitetype1} (which is also assumed in the formulation of the present theorem and which is  equivalent to conditions (1) or (3) of Theorem
\ref{finitetype} depending on $\rank D$) and decomposition \eqref{ngrseps} it follows that the subspace  $\mathfrak s (E_i)$ does not contain nonzero elements of rank not greater than $2$ . Below (Corollary to Lemma~\ref{fulton}) we will give the proof of this using the description  of the latter decomposition via geometry of rational normal curve. This and relation \eqref{rank2} imply that   vectors $v$ and $v_{E_i}$ are collinear $\mod \overline{E_i\oplus F_i}$.
As $v_{E_i}$ can be arbitrary and the dimension of the subspace $E_i$ is at least $4$, this implies that
\begin{eqnarray}
&~&\label{vo}v=0\, \mod \overline{E_i\oplus F_i},\\
&~&\label{aF}
\alpha_i|_{F_i}=0.
\end{eqnarray}

Since  relation
%\eqref{aE},
\eqref{vo},
%and \eqref{aF}
holds for an arbitrary $i$,$1\leq i\leq r$, we have
\begin{equation}
\label{voL}
v\in \mathcal L.
%,\\
%&~&
%\label{aL}
%\alpha_i|_{\overline {\mathcal L}}=0,
\end{equation}

Now assume that $v_1, v_2\in\mathcal L$ and let $\overline{\mathcal L}$ denotes the span of all vectors in $\Tab(\delta_{\rm{mod}})$ located in the boxes which are not in the last row of
$\Tab(\delta_{\rm{mod}})$. Consider equation \eqref{prolong1new} modulo $\overline{\mathcal L}$. Using  the definition of the space $\mathfrak p$, the fact that $Z_i|_{\mathcal L}=0$, and the relation\eqref{phipr}, one gets from \eqref{prolong1} that
\begin{equation}
\label{overL}
\sigma(v_1, v_2)v \in \overline{\mathcal L}.
\end{equation}

Since $\sigma|_{\mathcal L}\neq 0$ (actually $\sigma|_{\mathcal L}$ is a symplectic form on $\mathcal L$), we get from \eqref{overL} that $v\in \overline{\mathcal L}$, which together with \eqref{voL} implies that
\begin{equation}
\label{vzero}
v=0.
\end{equation}

Further,  let $v_1=v_{E_{i_1}}\in E_{i_1}$ and $v_2=v_{E_{i_2}}\in E_{i_2}$. where $i_1\neq i_2$. Then from \eqref{prolong1} and the fact that $\sigma(v_{E_{i_1}}, v_{E_{i_2}})=0$ it follows that $\phi(v_{E_{i_1}})v_{E_{i_2}}=\phi(v_{E_{i_2}})v_{E_{i_1}}$. Considering the latter equation modulo $\overline{E_{i_1}}$  and using \eqref{phipr} one gets that
\begin{equation}
\label{EE}
\phi'(v_{E_{i_1}})v_{E_{i_2}}=\alpha_{i_1}(v_{E_{i_2}})v_{E_{i_1}}\,\,\mod \overline{E_{i_1}}
\end{equation}
As before, the splitting
\eqref{splitEFL} induces the natural splitting of $\gl(X)$ as in \eqref{glsplit}. Let $\mathfrak P_{i_1, i_2}$ be the projection from $\gl(X)$ to $\gl(E_{i_2}, E_{i_1})$ with respect to the latter splitting.
Equation \eqref{EE} implies that the rank of the linear map $\mathfrak P_{i_1, i_2} \bigl(\phi'(v_{E_{i_1}})\bigr)$ is not greater than $1$.

On one hand,  $\mathfrak P_{i_1, i_2} \bigl(\phi'(v_{E_{i_1}})\bigr)$ must belongs to $\mathfrak n (E_{i_2}, E_{i_1})$.
On the other hand,  from our assumptions it follows that $\mathfrak n (E_{i_2}, E_{i_1})$ does not contain elements of rank $1$.
%the condition (1) of Theorem \ref{finitetype1} (which is also assumed in the formulation of the present theorem and which is  equivalent to conditions (1) or (3) of Theorem
%\ref{finitetype} depending on $\rank D$) and decomposition \eqref{ngrseps} it follows that the subspace  $\mathfrak s (E_i)$ does not contain nonzero elements of rank not greater than $2$ . Below (Corollary to Lemma~\ref{fulton}) we will give the proof of this using the description  of the latter decomposition via geometry of rational normal curve.
This implies that  $\mathfrak P_{i_1, i_2} \bigl(\phi'(v_{E_{i_1}})\bigr)=0$ and thus by \eqref{EE}  $\alpha_{i_1}(v_{E_{i_2}})v_{E_{i_1}}=0$.
%vectors $v$ and $V_{E_i}$ are collinear.
As $v_{E_{i_1}}\in E_{i_1}$ can be arbitrary  this implies that
\begin{equation}
\label{a12}
%v=0\, \mod \overline{E_i\oplus F_i},\\
%&~&\label{aF}
\alpha_{i_1}|_{E_{i_2}}=0, \quad i_1\neq i_2.
\end{equation}
In completely analogous way (taking $v_1\in E_{i_1}$ and $v_2 \in F_{i_2}$ or $v_1\in E_i$ and $v_2\in \mathcal L$) one can get that $\alpha_{i_1}|_{F_{i_2}}=0$, $i_1\neq i_2$ and $\alpha_i|_{\mathcal L}=0$. Combining this with relations \eqref{aE}, \eqref{aF}, and \eqref{a12} we get that
\begin{equation}
\label{afin}
\alpha_i=0, \quad 1\leq i\leq r.
\end{equation}

Therefore, by \eqref{phipr} $\phi=\phi'$, i.e. $\phi$ takes values in $\mathfrak p$. This together with \eqref{vzero} completes the proof of the Theorem \ref{genpr_i0} in the case of the first prolongation.

 Now we are going to establish the induction step for proving our theorem for prolongations of higher order. For this first let us consider the following more abstract situation:

  Let $Y$ be a vector space endowed with a skew-symmetric (possibly degenerate) form $\omega$. Similarly to symplectic case (i.e. when the form is degenerate) consider the graded nilpotent Lie algebra $\tilde\eta=\tilde\eta^{-1}\oplus\tilde\eta^{-2}$ , where $\tilde\eta^{-1}=Y$, $\dim \tilde\eta^{-2}=1$, $\tilde\eta^{-2}$ belongs to the center of $\tilde \eta$, and if a vector $z$ generates $\eta^{-2}$ then $[y_1, y_2]=\omega(y_1, y_2) z$ for any $y_1, y_2 \in \eta_{-1}$.

Let $W$ be a subspace in $\mathfrak r\bigl(\mg^0(\eta)\bigr)$. As in Remark \ref{Tanakastr}, $W$ can be considered as a subspace of $gl(Y)$. Let $W^{(1)}$ be the standard prolongation of $W$ and $W^{(1\mathfrak{sp})}$ be the first prolongation of the pair $(\eta,W)$ (in our previous notation $W^{(1\mathfrak{sp})}=\mathfrak u^1(\eta, W)$. Note that both $W^{(1)}$ and $W^{(1\mathfrak{sp})}$ can be considered as subspaces of $\Hom(Y, W)$ that in turn can be naturally identified with the subspace of $\gl(W\times Y)$ as follows:  that assigns to a map $\Phi\in \Hom(Y, W)$ the map  $\widetilde\Phi \in \gl(Y')$ as follows:
\begin{equation}
\label{idtriang}
\widetilde\Phi(w,y):=(\Phi(y),0).
\end{equation}

%\begin{lemma}
 Further, assume that the space $Y$ decomposes as follows
 \begin{equation}
 \label{splitY}
 Y=\bigoplus_{i=1}^{\tilde k} Y_i,
 \end{equation}
 where subspaces $Y_i$ satisfy the following properties:
 \begin{enumerate}
 \item [(Y1)]
 For any integer $j$ such that $1\leq j\leq \left[\frac{\tilde k}{2}\right]$ there exist vectors $v_1 \in Y_{2j-1}$ and $v_2 \in Y_{2j}$ such that $\omega(v_1, v_2) v\neq 0$;
 \item[(Y2)] In the case of odd $\tilde k$ the restriction of the form $\omega$ to the subspace $Y_{\tilde k}$ is nonzero identically zero.
 %\item[(Y1)]
% All subspaces $Y_i$, except the subspace $Y_{\tilde k}$ in the case of odd $\tilde k$,  are isotropic with respect to the form $\omega$, i.e. $\omega|_Y_i\equiv 0$;
% \item[(Y2)] For all $j$, $1\leq j\leq \left[cfrac{\tilde k}{2}\right]$ the restriction of the form $\omega$ to the subspace $Y_{2j-1}\oplus Y_{2j}$ is nonzero identically zero and , if $\tilde k$ is odd, the restriction of the form $\omega$ to the subspace $Y_{\tilde k}$ is nonzero identically zero.
 \end{enumerate}
% \end{lemma}
 As before, splitting \eqref{splitY} of $Y$ induces the splitting $\mathfrak {gl}(Y)=\displaystyle{\bigoplus_{i_1, i_2=1}^k}\Hom(Y_{i_1}, Y_{i_2})$. Let $\pi_{i_1,  i_2}:\mathfrak {gl}(Y)\rightarrow \Hom(Y_{i_1}, Y_{i_2})$ the canonical projection with respect to this splitting.

 \begin{lemma}
 \label{inductlem}
 With all notations as above, assume that the space $W$ satisfies the following two properties:
 \begin{enumerate}
 \item[(W1)] $W\subset \displaystyle{\bigoplus_{i_1\neq i_2}}\Hom(Y_{i_1}, Y_{i_2})$;
 \item[(W2)] For every  $i_1, i_2\in \{1,\ldots, \tilde k\}$ which are distinct (i.e. $i_1\neq i_2$) the subspace $\pi_{i_1, i_2}(W)$ of  $\Hom(Y_{i_1}, Y_{i_2})$ does not contain elements of rank $1$.
  \end{enumerate}
 Then the following two statements hold:

 \begin{enumerate}
 \item $W^{(1\mathfrak{sp})}=W^{(1)}$
 \item Consider the space $$Y'=\left(\bigoplus_{i_1\neq i_2} \pi_{i_1, i_2}(W)\right)$$
 %W\times Y
  endowed with the symplectic form $\omega'$ which is the pullback of the form $\omega$ with respect to the canonical projection from $Y'$ to $Y$ (note that  $W\times Y\subset Y'$ by condition (W1)). Take the following splitting of $Y'$:
     \begin{equation}
     \label{splitY'}
     Y'=\bigoplus_{i=1}^{\tilde k} Y_i',
 \end{equation}
 where
 \begin{equation}
     \label{splitiY'}
     Y_i'=\left(\bigoplus_{\{j: j\neq i, 1\leq j\leq \tilde k\}}\pi_{j, i}(W)\right)\times Y_i.
     \end{equation}
 \end{enumerate}
 Then the space $Y'$ satisfies conditions (Y1) and (Y2) and the space $W'=W^{(1\mathfrak{sp})}=W^{(1)}$, considered as the subspace of $\gl(Y')$
(under the natural identification as in \eqref{idtriang}),
 %that assigns to a map $\Phi\in \Hom(Y, W)$ the map  $\widetilde\Phi \in \gl(Y')$ as follows: $\widetilde\Phi(w,y):=(\Phi(y),0)$)
satisfies conditions (W1) and (W2) with respect to the splitting \eqref{splitY'}.
\end{lemma}
\begin{proof} Take  $\Phi\in W^{(1\mathfrak{sp})}$.
Similarly to \eqref{prolong1new}, there exists a vector $v\in Y$ such that
\begin{equation}
\label{omega}
\Phi(v_1)v_2-\Phi(v_2)v_1=\omega(v_1, v_2)v, \quad \forall v_1,v_2 \in Y
\end{equation}
The first part of the lemma will be proved if we will show that $v=0$.

Similarly to the previous considerations, let $\overline {Y_i}=\displaystyle{\bigoplus_{j\neq i}} Y_j$ and $\overline {Y_{i_1}\oplus Y_{i_2}}=\displaystyle{\bigoplus_{j\notin \{i_1, i_2\}}} Y_j=\overline {Y_{i_1}}\cap \overline {Y_{i_1}}$.
Fix integer $j$ such that  $1\leq j\leq \left[\frac{\tilde k}{2}\right]$ and take $v_1 \in Y_{2j-1}$, $v_2 \in Y_{2j}$. By property (W1) $\Phi(v_2)v_1\in \overline{Y_{2j-1}}$. Therefore, from \eqref{omega} it follows that
$\Phi(v_1)v_2 =\omega(v_1, v_2) v\,\,\mod \overline{Y_{2j-1}}$.
Hence, the linear map $\Phi(v_1)v_2$ has rank not greater than $1$.
>From this and assumption (W2) it follows that $\pi_{2j, 2j-1}\bigl(\Phi(v_1))=0$, i.e.
$\omega(v_1, v_2) v\in \overline{Y_{2j-1}}$ for any $v_1 \in Y_{2j-1}$, $v_2 \in Y_{2j}$. By assumptions (Y1) there exists  $v_1 \in Y_{2j-1}$ and $v_2 \in Y_{2j}$ such that $\omega(v_1, v_2) v\neq 0$, which implies that
\begin{equation}
\label{vY1}
v\in \overline{Y_{2j-1}}.
\end{equation}
By the same arguments, exchanging the role of $v_1$ and $v_2$, one gets that
\begin{equation}
\label{vY2}
v\in \overline{Y_{2j}}.
\end{equation}
As $j\in\{1,\ldots, \left[\frac{\tilde k}{2}\right]\}$ is arbitrary, we get from \eqref{vY1} and \eqref{vY2} that $v=0$ if $\tilde k$ is even or

\begin{equation}
\label{vYk}
v\in Y_{\tilde k} \text{ if  } \tilde k \text{ is odd.}
\end{equation}

In the latter case take $v_1$ and $v_2$  in $Y_{\tilde k}$. Then by assumption (W1)
$\Phi(v_1)v_2-\Phi(v_2)v_1\in \overline Y_{\tilde k}$. Hence, using \eqref{omega}, we get that
$\omega(v_1, v_2) v\in \overline{Y_{\tilde k}}$. By assumption (W2) there exist vectors $v_1$ and $v_2$  in $Y_{\tilde k}$ such that $\omega(v_1, v_2)\neq 0$, which implies that $v\in \overline{Y_{\tilde k}}$. Combining this with \eqref{vYk} we get that $v=0$ for odd $\tilde k$ as well. This proves the first part of the lemma.

 %Let $(\phi,v)\in W^{(1m)}$. Consider~\eqref{genprol} for $v_e\in E$, $v_f\in
%F$. As $\phi(v_f)v_e = 0$, we get $\phi(v_e)v_f = \sigma(v_e,v_f)v$. So, the
%element $\phi(v_e)$ is a map of rank $\le 1$. Hence, by assumption of the
%lemma, we get $\phi(v_e)=0$ for all $v_e\in V_e$. Since $F$ has a trivial
%intersection with $\ker\sigma$, for any non-zero $v_f\in F$ there exists such
%$v_e\in E$ that $\sigma(v_e,v_f)\ne 0$. Hence, we also get $v=0$. This proves
%that $W^{(1m)}=W^{(1)}$.

Now let us prove the second part of the lemma.
First it is clear from definition of the form $\omega'$ that the splitting  \eqref{splitY'} satisfies conditions (Y1) and (Y2).

Let us prove that the space $W'$ satisfies condition (W1) with respect to the splitting \eqref{splitY'}.
%i.e. that
%\begin{equation}
%\label{W1check}
%W'\subset \displaystyle{\bigoplus}_{i_1\neq i_2}\Hom(Y_{i_1}', Y_{i_2}')$
%\end{equation}
Let $\pi_{i_1,i_2}':\gl(Y')\mapsto \Hom(Y_{i_1}', Y_{i_2}')$ be the canonical projection with respect to the splitting of $\gl(Y')$ induced by the splitting \eqref{splitY'} of the space $Y'$.
We actually have to prove that $\pi_{i, i}'(\Psi)=0$ for any $\Psi \in W'$ and any integer  $i$, $1\leq i\leq \tilde k$.

If $\pi_i: Y\mapsto Y_i$ denote the canonical projection with respect to the splitting \eqref{splitY},  then directly from definition \eqref{splitiY'} of subspaces $Y_i'$ for any $\Psi \in \Hom(Y, W)\subset \gl(Y')$ we have the following relation:
\begin{equation}
\label{pirel}
\pi_{i_1, i_2}'(\Psi)(y)=\pi_{i_2}\circ\bigl(\Psi(\pi_{i_1} y)\bigr), \quad \forall y\in Y.
\end{equation}
In the righthand side of the last equation $\circ$ means the composition of two linear maps:
$\Psi(\pi_{i_1} y)\in W\subset \gl(Y)$ and $\pi_{i_2}\in \Hom (Y,Y_{i_2})$.

If $\Psi\in W'$ ($=W^{(1)}$), then $\Psi(y)z=\Psi(z) y$ for any $y, z \in Y$. Therefore, by \eqref{pirel}

$$(\pi_{i, i}\Psi)(y)z=\pi_i\bigl(\Psi (\pi_i y) z\bigr)=\pi_i\bigl(\Psi (z) \pi_i y\bigr).$$
Since $\Psi(z)\in W$ and $W$ satisfies condition (W1), then $\Psi (z) \pi_i y\in \overline{Y_i}$, which implies that $\pi_i\bigl(\Psi (z) \pi_i y\bigr)=0$. So, $\pi_{i, i}\Psi)=0$, hence the space $W'$ satisfies condition (W1).

Finally, let us prove that  the space $W'$ satisfies condition (W1) with respect to the splitting \eqref{splitY'}. Assume by contradiction that there exists $\Psi\in W'$ and indices $i_1 \neq i_2$, $1\leq i_1, i_2\leq \tilde k$ such that the map $\pi_{i_1, i_2}'(\Psi)$ has rank 1. It means hat there exist $w_0\in \gl(Y)$ and $\alpha \in Y_{i_1}^*$, $\alpha\neq 0$ such that

\begin{equation}
\label{rank1w}
\pi_{i_1, i_2}'(\Psi)(y)=\alpha(\pi_{i_1} y)w_0.
\end{equation}

Take $y_1, y_2\in Y_{i_1}$. Then from \eqref{pirel} it follows that
\begin{equation}
\label{rank1w1}
\pi_{i_1, i_2}'(\Psi)(y_1)y_2=\pi_{i_2}\Psi(y_1)y_2=\alpha(y_1)w_0(y_2).
\end{equation}
In the same way,
%\begin{equation}
%\label{rank1w1}
$$\pi_{i_1, i_2}'(\Psi)(y_2)y_1=\alpha(y_2)w_0(y_1).$$
%end{equation}
Since $\Psi\in W'=W^{(1)}$,  $$\pi_{i_1, i_2}'(\Psi)(y_1)y_2=\pi_{i_2}\Psi(y_1)y_2=\pi_{i_2}\Psi(y_2)y_1=\pi_{i_1, i_2}'(\Psi)(y_2)y_1.$$
Therefore,
%\begin{equation}
%\label{alpha12}
$\alpha(y_1)w_0(y_2)=\alpha(y_2)w_0(y_1)$
%\end{equation}

Since by our assumptions $\alpha\neq 0$, there exists $y_2\in Y_{i_1}$ such that $\alpha(y_2)\neq 0$, which implies that

\begin{equation}
\label{alpha12}
w_0(y)=\cfrac{\alpha(y)}{\alpha(y_2)}w_0(y_2), \quad \forall y\in Y_{i_1}
\end{equation}
So, the map  $w_0$ has rank $1$. Finally, \eqref{rank1w1} implies that $w_0(y_2) \in Y_{i_2}$, which yields that $w_0\in \pi_{i_1, i_2}(W)$. This means that the space $\pi_{i_1, i_2}(W)$ contains elements of rank $1$, contradicting assumption (W2) on $W$.
%Finally note that from \eqref{rank1w} and \eqref{pirel} it follows that  $w_0(y)=
%We have already proved above that
%$\phi(v_e) = 0$ for any $v_e\in E$. Hence, $W^{(1m)}$ lies in
%$\Hom(F',E')\subset\End(V')$. It is sufficient to show that $W'$ does not have
%any elements of rank $1$. Suppose, there is such an element. Then it can be
%represented as $v_f\mapsto \alpha(v_f)w_0$ for some non-zero $\alpha\in F^*$,
%$w_0\in W$. Again, considering equation~\eqref{genprol} for two vectors
%$v_1,v_2\in F$, we get:
%\[
%\alpha(v_1)w_0(v_2) = \alpha(v_2)w_0(v_1), \quad\text{for all }v_1,v_2\in F.
%\]
%We can always choose such $v_0\in F$ that $\alpha(v_0)\ne 0$. Then we get
%$w_0(v) = \alpha(v)w_0(v_0)/\alpha(v_0)$ for any $v\in F$. Hence, $w_0$ has
%rank $1$ as well, and this contradicts the assumption of the lemma.
\end{proof}

Now let us explain how to use Lemma \ref{inductlem} as an induction step in proving of our Theorem \ref{genpr_i0} for higher order prolongations. Assume as before that $\eta=\eta^{-1}\oplus\eta^{-2}$ be the Heisenberg algebra, the algebra $\mg^0\subset\mathfrak{csp}(\eta^{-1})$ and for a given $i\in\mathbb N$
$$Y=\eta^{-1}\oplus\mg^0\oplus\bigoplus_{j=1}^{i-1}\mathfrak u^j(\eta,\mg^0).$$
Let $\omega$ be the slew-symmetric form on $Y$, which is the pullback of a symplectic form on $\eta^{-1}$ (defined up to a multiplication by a constant).
Take $W=\mathfrak u^i(\eta,\mg^0)$.
Note that $W$ can be considered as a subspace of $\Hom\bigl(\eta^{-1}, \mathfrak u^{i-1}(\eta,\mg^0)\bigr)$, which in turn can be identified with a subspace of $\gl(Y)$.
Directly from the definitions of prolongations one can show that under this identification

\begin{equation}
\label{1sphigh}
W^{(1\mathfrak{sp})}\cong u^{i+1}(\eta,\mg^0).
\end{equation}
Here again $u^{i+1}(\eta,\mg^0)$ is identified with a subspace in $\gl\bigl(Y\oplus u^i(\eta,\mg^0)\bigr)$.

To finish the proof of Theorem \ref{genpr_i0} we proceed by induction.
We proved previously that $u^1(\eta,\mg^0)=\mathfrak p^{1}$. Also, the algebra $\mathfrak p$ satisfies conditions of Lemma \ref{inductlem}. Hence, $\mathfrak p^{(1\mathfrak{sp})}=\mathfrak p^{(1)}$. Moreover, using the fact that $(\mathfrak p^{(i)})^{(1)}=\mathfrak p^{(i+1)}$ and applying the same lemma inductively we get that  $\mathfrak p^{(i+1)}=(\mathfrak p^{(i)})^{(1\mathfrak {sp})}$.
Assuming by induction that $u^i(\eta,\mg^0)=\mathfrak p^{(i)}$ for some $i\in\mathbb N$ and using \eqref{1sphigh}, we get that
$$u^{i+1}(\eta,\mg^0)=(u^i(\eta,\mg^0))^{(1\mathfrak {sp})}= (\mathfrak p^{(i)})^{(1\mathfrak {sp})}=\mathfrak p^{(i+1)},$$
which completes the induction step and therefore the proof of the theorem. $\Box$
%\end{proof}
%\end{proof}
\section{Description of algebra $\mathfrak u(\eta, \mathfrak u^F(\delta_{\rm mod}))$ via geometry of certain projective varieties}
\setcounter{equation}{0}
\setcounter{theorem}{0}
\setcounter{lemma}{0}
\setcounter{proposition}{0}
\setcounter{definition}{0}
\setcounter{corollary}{0}
\setcounter{remark}{0}
\setcounter{example}{0}
\label{secantsec}
%\begin{remark}
%\label{paramrem}
%(Geometric interpretation of the ideal $\mathfrak l
%%^{\mathfrak{sp}}
%(X)$)

 In this section we make an attempt given a flag symbol satisfying conditions of Theorem \ref{genpr_i0} to describe the algebra  $\mathfrak u(\eta, \mathfrak u^F(\delta_{\rm mod}))$ in the  more explicitly way  using the fact that if a subspace in $\mathfrak{sp}(X)$  can be identified with a space of quadratic polynomials vanishing on a given projective variety $\mathcal E$ in $\mathbb PX$, then the standard prolongations of this subspace can be identified with the space of certain polynomials vanishing on the secant variety of $\mathcal E$ of certain order.
 For general Jacobi symbols in this way we are able only to show that $\mathfrak u^i(\eta, \mathfrak u^F(\delta_{\rm mod}))$ is a subspace in a space of this kind, obtaining an upper bound for the dimension of $\mathfrak u^i(\eta, \mathfrak u^F(\delta_{\rm mod}))$.
  %The goal of this section is given a flag symbol satisfying conditions of Theorem \ref{genpr_i0} to show that the $i$th algebraic prolongation $\mathfrak u^i(\eta, \mathfrak u^F(\delta_{\rm mod}))$ is a subspaces of certain spaces of polynomials vanishing on a certain projective variety and
However, for a large class of Jacobi symbols the spaces $\mathfrak u^i(\eta, \mathfrak u^F(\delta_{\rm mod}))$ can be identified with certain spaces of polynomials of this kind (and not only with certain subspaces of them).  In this way we get the description of the prolongation algebra $\mathfrak u(\eta, \mathfrak u^F(\delta_{\rm mod}))$ for all Jacobi symbols appearing in rank $3$ distributions, for which this algebra is finite dimensional,  and also for most of such Jacobi symbols appearing in rank $4$ distributions (see Remark \ref{234} below). In particular, in the case of rank $3$ distributions  with Jacobi symbol having nonrectangular skew Young diagram this description is given in terms of tangential developable and secant varieties of a rational normal curve.
%We shall
%use the fact that it coincide with the standard $i$th prolongation of the algebra $\mathfak l(X)$ or the subspace $\mathfrak p$ and some well-known algebraic-geometric description of standard prolongations of certain spaces.
%In the case when the  Jacobi symbol is such that in its decomposition onto symplectically indecomposable symbols there is no symbols of type $\delta_{s,2s}$, i.e. with the rectangular Young diagram consisting of two rows, $\mathfrak u^i(\eta, \mathfrak u^F(\delta_{\rm mod}))$  can be identified  for any $i>0$
%In the case of rank 3 distributions of maximal class with $6$-dimensional square this gives an exact description of  $\mathfrak u^i(\eta, \mathfrak u^F(\delta_{\rm mod}))$ for any $i>0$.

To begin with, to any  $A\in \Hom(X,Y)$ one can naturally assign a bilinear form $\widetilde A$ on $X\times Y^*$ as follows: $\widetilde A(v, p):=p(Av), \,\,v\in X, p\in X^*$.
Given a subspace $W$ in  $\Hom(X,Y)$ denote by $\widetilde W$ the corresponding subspace in the space of bilinear forms on $X\times Y^*$.
Let $r$ and $s$ are dimensions of $X$ and $Y$ respectively and  $x=(x_1,\ldots x_r)$ and $y=(y_1,\ldots, y_s)$ are coordinates on $X$ and $Y^*$ w.r.t. some bases in those vector spaces.
Then any bilinear form is identified with a polynomial in $x$ and $y$, which is linear in $x$ and linear in $y$.
The following lemma is well known and immediately follows from definitions:
\begin{lemma}
\label{partial}
The $i$th prolongation of a subspace $W$ of $\Hom(X,Y)$ can be naturally identified with all polynomials $F(x,y)$ of degree $i+1$ in $x$ and linear in $y$ such that all their partial derivatives in $x$ of order $i$ (i.e. polynomials $\frac{\partial^i F}{\partial x_1^{\alpha_1}\dots
\partial x_r^{\alpha_r}}$ for all multi-indices $\alpha=(\alpha_1,\dots,\alpha_r)$
with $|\alpha|=i$) belong to the space $\widetilde W$.
\end{lemma}

Now assume that the space $X$ is endowed with a symplectic form $\sigma$. The form $\sigma$ defines the identification between $X$ and $X^*$ given by the map $v\mapsto \sigma(v,\cdot)$,
$x\in X$. Under this identification given an endomorphism $A$ on $X$ the form $\widetilde A$ is a bilinear form on $X\times X$. An endomorphism $A$ of $X$ belongs to $\mathfrak{sp}(X)$ if and only if the corresponding bilinear form $\widetilde A$ is symmetric. In this case $\widetilde A$ is completely determined by the corresponding quadratic form
\begin{equation}
\label{tildeAsymp}
v\mapsto \widetilde A(v,v)=\sigma (v, Av).
\end{equation}
Therefore a subspace $W$ of $\mathfrak{sp}(X)$ can be identified with a subspace of quadratic forms on $X$ which by analogy with the previous considerations will be denoted by $\widetilde W$. Then in this situation Lemma \ref{partial} can be reformulated as follows:

\begin{lemma}
\label{partialsymp}
The $k$th prolongation of a subspace $W$ of $\mathfrak{sp}(X)$ can be naturally identified with the space of all polynomials of degree $k+2$ in $X$ such that all their partial derivatives of order $k$ belong to the corresponding subspace $\widetilde W$ in the space of quadratic forms.
\end{lemma}

Given a projective variety $\mathcal E$ denote by $\mathcal I(\mathcal E)$ the ideal of homogeneous polynomials vanishing on $\mathcal E$. We shall also denote by $\mathcal I_k(\mathcal E)$ the subspace
of all polynomials of degree $k$ in $\mathcal I(\mathcal E)$.
%Given a projective variety $\mathcal C$ d
Define
%its $k$
\emph{the $k$th
secant variety $\S^k\mathcal E$ of $\mathcal E$} as the algebraic closure of the union of
$k$-planes in the projective space passing through $k+1$ points from $\mathcal E$. By definition
we set $\S^0(\mathcal E)=\mathcal E$.

\begin{lemma}
\label{secantlem}
Assume that a subspace $W$ of $\mathfrak {sp}(X)$ satisfies the following property: there exists a projective variety $\mathcal E$ in $X$ such that
the corresponding subspace $\widetilde W$ belongs to $\mathcal I_2(\mathcal E)$. Then the  $k$th standard prolongation of $W$ can be identified with a subspace of  $\mathcal I_{k+2}\bigl(S^k\mathcal E\bigr)$. Moreover, if $\widetilde W=\mathcal I_2(\mathcal E)$, then $W^{(k)}\cong \mathcal I_{k+2}\bigl(S^k\mathcal E\bigr)$.
\end{lemma}
\begin{proof}
By Lemma \ref{partialsymp} the $k$th prolongation of $W$ can be
identified with polynomials $F$ of degree $k+2$ such
that all their partial derivatives of order $k$ belong to
%$\frac{\partial^n F}{\partial x_1^{\alpha_1}\dots
%\partial x_r^{\alpha_r}}\in
$\mathcal I_2(\mathcal E)$.
%for all multi-indices $\alpha=(\alpha_1,\dots,\alpha_r)$
%with $|\alpha|=n$.
It is clear that any such polynomial and all its
partial derivatives of degree $\le k$ lie in $\mathcal I(\mathcal E)$.

Let us prove that any such polynomial $F$ vanishes identically at
$\S^{k}(\mathcal E)$. It is sufficient to prove that it vanishes at any
secant $k$-plane of $\mathcal E$. Indeed, let $p_0,\dots,p_k\in \mathcal E$ be the set
of $k+1$ points and let
\[
W\colon \P^{k}\to \P^{r-1}, \quad [y_0:y_1:\cdots:y_k]\mapsto
y_0p_0+y_1p_1+\dots y_np_n;
\]
be the embedding of the corresponding secant $k$-plane, where $r=\dim X$. Then $F\circ
W$ is a polynomial of degree $k+2$ on $\P^k$ that vanishes at basis
points $q_i\in \P^k$, $i=0,\dots,k$, and, in addition, all its
derivatives of degree $\le k$ also satisfy this property. This
immediately implies that $F\circ W=0$. Hence, $F$ vanishes
identically at $\S^{k}(\mathcal E)$, i.e. it belongs to $\mathcal I_{k+2}\bigl(S^k\mathcal E\bigr)$.

Finally, if $\widetilde W=\mathcal I_2(\mathcal E)$ by make the arguments in the reversed order one gets that if a polynomials $F$ of degree $k+2$ vanishes on $S^k\mathcal E$ then all its derivatives of order $k$ vanishes on $\mathcal E$ and hence $F$ belongs to $W^{(k)}$.

\end{proof}
%\end{proof}

The question now is what varieties can be taken as $\mathcal E$ for subspaces $\mathfrak p$ or $\mathfrak l(X)$? Throughout the rest of the section we assume that the symbol $\delta_{\mod}$ satisfies all assumptions of Theorem \ref{genpr_i0}.

As in the beginning of the proof of Theorem \ref{genpr_i0}
assume that in the diagram $\Diag(\delta_{\rm mod})$ the Young subdiagram of the indecomposable component of type $\mathbb R\tau_m$, if exists, is situated at the very bottom. Fix again a skew Young tableaux $\Tab(\delta_{\rm mod})$ and let $r=\left[\frac {1}{2}(\rank\overline{} D-1)\right]$.
For any $i$ such that $1\leq i\leq r$ let $E_i$ be the span of the vectors located in the boxes of the $(2i-1)$st row of  $\Tab(\delta_{\rm mod})$ and $F_i$ be the span of the vectors located in the boxes of the $2i$th rows of  $\Tab(\delta_{\rm mod})$. Finally, if rank $D$ is even, let $\mathcal L$ be the span of the vectors located in the last box of  $\Tab(\delta_{\rm mod})$.
%Fix a decomposition of the symbol $\delta_{\mod}$ into symplectically indecomposable symbols (note that such decomposition os not unique in general, if it contains more than one copy of some symplectically indecomposable symbol). Consider the following three cases separately:

Let $j_0$ be the minimal index in $I$ (with respect to the natural order) such that the subspace $X_{j_0}$ is isotropic. Note that in fact $j_0$ is the minimal positive index in $I$, i.e. $j_0=1$ if $I=\mathbb Z$ and $j_0=\cfrac{1}{2}$ if $I=\cfrac{1}{2}\mathbb Z_{{\rm odd}}$  or $I=\cfrac{1}{2}\mathbb Z$.
Let $\mathcal A$ be a subset $\{1,\ldots, r\}$ of all $i$ such that the subdiagram of $\Diag(\delta_{\mod})$ consisting of the $(2i-1)$st and $2i$th row is not rectangular. Assume that the restriction of the symbol $\delta_{\mod}$ to $E_i\oplus F_i$ is isomorphic to the symbol $\mathbb R\delta_{s,l}$ and fix
the skew Young tableaux for this restriction  as in \eqref{youngodd1}. Given $i\in \mathcal A$ let $\mathfrak L_i$ be the linear span of $X_{j_0}$ and the vector $e_0$ from the skew Young tableaux \eqref{youngodd1}
corresponding to the restriction of the symbol $\delta_{\mod}$ to $E_i\oplus F_i$. If $i\notin \mathcal A$ we set $\mathfrak L_i=X_{j_0}$.

Let $\mathcal E_i(\delta_{\mod})$ be the algebraic closure of the  curve  $x\mapsto e^x \mathfrak L_i(\delta_{mod}), \, x\in \delta_{\mod}$.
$\mathcal E_i(\delta_{\mod})$ can be considered as algebraic varieties in the projective space $\mathbb P X$ (by taking the closure of the union of all lines belonging to spaces $e^x \mathfrak L_i(\delta_{mod})$ for all $x\in \delta_{\mod}$).
%Let $i_0$ be the minimal index in $I$ (with respect to the natural order) such that the subspace $X_{i_0}$ is isotropic. Note that in fact $i_0$ is the minimal positivie index in $I$, i.e. $i_0=1$ if $I=\mathbb Z$ and $i_0=\cfrac{1}{2}$ if $I=\cfrac{1}{2}\mathbb Z_{\odd}$  or $I=\cfrac{1}{2}\mathbb Z$.

\begin{lemma}
\label{Fi_0}
If $A\in \mathfrak p$, then the corresponding quadratic form $\widetilde A$ vanishes on $\mathcal E_i(\delta_{\mod})$ for every $i$ such that $1\leq i\leq r$, if $\rank D$ is odd, and $1\leq i\leq r+1$, if $\rank D$ is even.  In other words,
\begin{equation}
\label{L(X)I2}
\widetilde{\mathfrak p}\subset \mathcal I_2\bigl(\mathcal E_i(\delta_{\mod})\bigr).
\end{equation}
\end{lemma}

\begin{proof} If $A\in \mathfrak p$ then \eqref{lXcond}  implies that $A \bigl(e^{x}(X_{j})\bigr)\subset e^{x}(X_{j})\,\, \forall  x\in \delta_{\mod}$. Since by the definition of $j_0$ all subspaces $e^{x}(X_{j_0})$ are isotropic, from this and \eqref{tildeAsymp} it follows that the quadratic form  $\widetilde A$ vanishes on $e^{x}(X_{j_0})$. This completes the proof in case $i\notin\mathcal A$.

Now assume that $i\in \mathcal A$.
Since $e^tx \mathfrak L_i\subset \Bigl(e^{x}(X_{j_0})\Bigr)^\angle$, then $\tilde A (v_1, v_2)=0$ for any $v_1\in e^{x}(X_{j_0})$ and $v_2\in e^tx \mathfrak L_i$.
 Finally, in case by \eqref{ngrseps2}, the second line there, $A e_0$ belongs to $\overline{E_i\oplus F_i}\subset E_i^\angle$, where $e_0$ is as in \eqref{youngodd1}, considered as the skew Young tableaux for the restriction of the symbol $\delta_{\mod}$ to $E_i\oplus F_i$. This implies that the quadratic form  $\widetilde A$ vanishes on $e^{x}\mathfrak L_i$ and completes the proof of the lemma.
\end{proof}

%\label{ioosculation}
%\begin{remark}

As a direct consequence of Corollary \ref{corl(X)} and Lemmas \ref{secantlem} and \ref{Fi_0}, we get the following

\begin{theorem}
\label{thmmainsec8}
Assume that the flag symbol $\delta_{\mod}$ satisfies all assumptions of Theorem \ref{genpr_i0}.
Then for every $i$ such that $1\leq i\leq r$, if $\rank D$ is odd, and $1\leq i\leq r+1$, if $\rank D$ is even and for any positive integer $k$ the $k$th  Tanaka algebraic prolongation $\mathfrak u^k\bigl(\eta, \mathfrak u^F(\delta_{\rm mod})\bigr)$ of the pair $\bigl(\eta, \mathfrak u^F(\delta_{\rm mod})\bigr)$  can be identified with a subspace of  the space $\mathcal I_{k+2}\bigl (S^k \mathcal E_i(\delta_{\mod})\bigr)$ of all polynomials of degree $k+2$  on $X$ that vanishes on the $k$th secant variety of the variety  $\mathcal E_i(\delta_{\mod})$.
%Moreover, if $\widetilde W=\mathcal I_2(\mathcal E)$, then $W^{(k)}\cong I_{k+2}\bigl(S^k\mathcal E\bigr)$.
\end{theorem}

Note that  under assumptions of this theorem on the flag symbol $\delta_{\mod}$ certain osculating subspace of the curve $x\mapsto e^x X_{j_0}$, $x\in \delta_{\mod}$,
is equal to the whole ambient vector space $X$ at any point.
Besides, by construction $X_{j_0}\subseteq \mathfrak L_i$ for any $i$.
Therefore, $S^K\mathcal E_i(\delta_{\mod})=\mathbb PX$ for sufficiently large $K$ and by Theorem \ref{thmmainsec8} the spaces  $\mathfrak u^K\bigl(\eta, \mathfrak u^F(\delta_{\rm mod})\bigr)$ vanishes for such $K$. This gives more constructive proof of the finiteness type result for distributions with Jacobi symbols satisfying assumptions of Theorem  \ref{thmmainsec8} without referring to the Spencer criterium. In addition, this gives some estimate from above on the size of the spaces $\mathfrak u^k\bigl(\eta, \mathfrak u^F(\delta_{\rm mod})\bigr)$.

The space $\mathcal I_2\bigl(\mathcal E_i(\delta_{\mod})\bigr)$ is usually larger than the space $\mathfrak p$, so Theorem \ref{thmmainsec8} gives only an estimation from above for the size of   $\mathfrak u^k\bigl(\eta, \mathfrak u^F(\delta_{\rm mod})\bigr)$.
If the quadratic form $\widetilde A$ belongs to $\mathcal I_2\bigl(\mathcal E_i(\delta_{\mod})\bigr)$, then by polarization the corresponding bilinear form $\widetilde A$ satisfies $\widetilde A(v_1, v_2)=0$ for every $v_1, v_2 \in e^{x}(\mathfrak L_i)$
and $x\in \delta_{\mod}$. Hence, if $A\in \mathfrak{sp}(X)$ is related to the form $\widetilde A$ by \eqref{tildeAsymp}, then $\sigma(v_1, A v_2)=0$ for every $v_1, v_2 \in e^{x}(\mathfrak L_i)$, which is equivalent to the following inclusion
\begin{equation}
\label{I2cond}
A \bigl(e^{x}\mathfrak L_i\bigr)\subset \Bigl(e^{x}\mathfrak L_i\Bigr)^\angle,\quad \forall x\in\delta_{\mod}
\end{equation}
%\end{remark}

If $i\in \mathcal A$, then $\Bigl(e^{x}(\mathfrak L_i)\Bigr)^\angle$ is equal to the linear span of all element of nonnegative degree in $\overline {E_i\oplus F_i}$ and vector $e_1$,  while if $i\notin\mathcal A$ then $\Bigl(e^{x}(\mathfrak L_i)\Bigr)^\angle$ is equal to the linear span of element of nonnegative degrees in $X$. So, condition \eqref{I2cond} does not ensure that any $A\in \mathcal I_2\bigl(\mathcal E_i(\delta_{\mod})\bigr)$ is concentrated  in nonnegative degrees. It is guarantied only in the case $r=0$ (corresponding to rank 2 distributions of maximal class), because in this case the subspaces $X_{j_0}$ is Lagrangian, but this case is not interesting since
 here all subspaces involved such as $\mathfrak p, \mathfrak l(X)$,  and $\mathcal I_2\bigl(\mathcal E_i(\delta_{\mod})\bigr)$ are equal to zero, as follows from above arguments and part c) of the proof of Lemma \ref{tensoroff}.

Now we give a class of flag symbols $\delta_{\mod}$ for which the spaces $\mathfrak u^i(\eta, \mathfrak u^F(\delta_{\rm mod}))$ can be identified with spaces of polynomials of this kind ( and not only with subspaces of them). Continuing the numeration of assumptions Theorems \ref{finitetype1} and \ref{genpr_i0} assume that

\begin{enumerate}
\item[(5)]
In  a decomposition of the symbol $\delta_{\mod}$ into indecomposable components there is no indecomposable symbols of the type $\mathbb R\delta_{s, 2s}$, i.e. symbols with rectangular Young diagram consisting of two rows. Equivalently, the introduced above set $\mathcal A$ coincides with $\{1,\ldots, r\}$.

\item[(6)] For any  two indecomposable symbols $\mathbb R\delta_{s_1, l_1}$ and $\mathbb R\delta_{s_2,l_2}$ in the decomposition of the symbol $\delta_{\mod}$ into indecomposable components we have
 \begin{equation}
 \label{cond6}
 s_1+s_2>\max\{l_1,l_2\}.
\end{equation}
\end{enumerate}

\begin{remark}
\label{5,6}
Note that assumption (6) formally implies assumption (5) if one counts  the same indecomposable component twice
%$\mathbb R\delta_{s_1, l_1}$ and $\mathbb R\delta_{s_2,l_2}$
. However, we would like to emphasis assumption (5) as a separate one.
\end{remark}

The key point of assumption (5) is that under this assumption
\begin{equation}
\label{cond5key}
\mathfrak s(F_i)=0, \quad \forall 1\leq i\leq r,
\end{equation}
where, as in the part b) of the proof of Lemma \ref{tensoroff}, $\mathfrak s(F_i)$
is  the maximal $\sll_2$-submodule of
%$V_{r_2s_2}\otimes (V_{r_1s_1})^*$
$S^2(F_{s,l}
%^{\mathfrak{sp}}
)$
%($\wedge^2(E_{s;l}
%^{\mathfrak{so}}
%)$)
concentrated in the non-negative degree part of $S^2(F_{s,l}
%^{\mathfrak{sp}}
)$, considered as the subspace of $\Hom(E_i, F_i)$ under the identification of $E_i^*$ with $F_i$ via the symplectic form (see also \eqref{ngrseps2}, second line).

The key point of assumption (6) is that
\begin{equation}
\label{cond6key}
\mathfrak n(E_{i_1}, F_{i_2})=0, \quad \forall 1\leq i_1\neq i_2\leq r,
\end{equation}

where as in the part b) of the proof of Lemma \ref{tensoroff}, $\mathfrak n(E_{i_1}, F_{i_2})$ is  the maximal $\sll_2$-submodule of
%$V_{r_2s_2}\otimes (V_{r_1s_1})^*$
$\Hom (E_{i_1}, F_{i_2})
%^{\mathfrak{sp}}
)$
%($\wedge^2(E_{s;l}
%^{\mathfrak{so}}
%)$)
concentrated in the non-negative degree part of $\Hom (E_{i_1}, F_{i_2})
%^{\mathfrak{sp}}
)$. Identity \eqref{cond6key} can be obtain easily from \eqref{Y1Y2}.
%^{\mathfrak{sp}}

Further, let $\mathfrak L$ be equal to the linear span of all spaces $\mathfrak L_i$ with $1\leq i\leq r$, if $\rank D$ is odd, and $1\leq i\leq r+1$, if $\rank D$ is even.

\begin{remark}
\label{lagrrem}
 Note that under assumption (5) the space $\mathfrak L$ is Lagrangian.
\end{remark}
 Now let $\mathcal E(\delta_{\mod})$ be the algebraic closure of the  curve  $x\mapsto e^x \mathfrak L(\delta_{mod}), \, x\in \delta_{\mod}$. Again $\mathcal E(\delta_{\mod})$ can be considered as algebraic varieties in the projective space $\mathbb P X$ (by taking the closure of the union of all lines belonging to spaces $e^x \mathfrak L(\delta_{mod})$ for all $x\in \delta_{\mod}$).

\begin{lemma}
\label{Fi_0exact}
Assume that Jacobi symbols satisfies conditions of Theorem \ref{finitetype1} and in addition satisfies assumptions (5) and (6) above.
If $A\in \mathfrak p$, then the corresponding quadratic form $\widetilde A$ vanishes on $\mathcal E(\delta_{\mod})$.  In other words,
\begin{equation}
\label{L(X)I2exact}
\widetilde{\mathfrak p}\subset \mathcal I_2\bigl(\mathcal E(\delta_{\mod})\bigr).
\end{equation}
\end{lemma}

\begin{proof} We proceed exactly as in the proof of Lemma \ref{Fi_0}. The only difference is in the final step: here we need to show that $\sigma(v, Av)=0$ for any $v\in \mathfrak L$ which has  degree 0. Indeed, from \eqref{cond5key} and \eqref{cond6key} it follows that  $Av\in \mathfrak L\cap \displaystyle{\oplus}_{i=1}^r E_i$. Then by Remark \ref{lagrrem} $\sigma(Av,v)=0$.
\end{proof}

Further, using the same arguments as ones leading to relation \eqref{I2cond} we get that $\widetilde A\in  \mathcal I_2\bigl(\mathcal E(\delta_{\mod})$ if and only if

\begin{equation}
\label{I2condexact}
A \bigl(e^{x}\mathfrak L\bigr)\subset \Bigl(e^{x}\mathfrak L\Bigr)^\angle=e^{x}\mathfrak L,\quad \forall x\in\delta_{\mod}.
\end{equation}

The last equality follows again from the fact that $\mathfrak L$ is Lagrangian.
The next important observation is that if the symbol $\delta_{\mod}$ satisfies assumption (1) of Theorem \ref{finitetype1} then the whole flat curve $\mathcal F_{\delta_{mod}}$ can be uniquely recovered from the curve
$x\mapsto e^{x}\mathfrak L$,  $x\in \delta_{\mod}$
by osculation and taking the skew-orthogonal complement. This together with \eqref{I2condexact} and Remark \ref{symparamrem} implies that $A\in \mathfrak l(X)$. In other words,

\begin{equation}
\label{I2above}
\mathcal I_2\bigl(\mathcal E(\delta_{\mod})\subset \mathfrak l(X).
\end{equation}
Combining inclusions \eqref{L(X)I2exact} and \eqref{I2above} with Theorem \ref{genpr_i0}, Corollary \ref{corl(X)}, and the last statement of Lemma \ref{secantlem}, we get the following

\begin{theorem}
\label{thmmainsec8.5}
Assume that the flag symbol $\delta_{\mod}$ satisfies all assumptions of Theorem \ref{genpr_i0} and in addition assumptions (5) and (6) above.
Then for any positive integer $k$ the $k$th  Tanaka algebraic prolongation $\mathfrak u^k\bigl(\eta, \mathfrak u^F(\delta_{\rm mod})\bigr)$ of the pair $\bigl(\eta, \mathfrak u^F(\delta_{\rm mod})\bigr)$  can be identified with the space $\mathcal I_{k+2}\bigl (S^k \mathcal E(\delta_{\mod})\bigr)$ of all polynomials of degree $k+2$  on $X$ that vanishes on the $k$th secant variety of the variety  $\mathcal E(\delta_{\mod})$.
%Moreover, if $\widetilde W=\mathcal I_2(\mathcal E)$, then $W^{(k)}\cong I_{k+2}\bigl(S^k\mathcal E\bigr)$.
\end{theorem}

\begin{remark}
\label{234}
Note that if $r\leq 1$, i.e. in the case of Jacobi symbols of distributions of rank $2, 3$, or $4$ of maximal class condition 6 is void (if one considers distinct indecomposable symbols in the decomposition of $\delta_{\mod}$) and therefore can be omitted in Theorem \ref{thmmainsec8.5}. Consider each of these ranks separately:
\begin{enumerate}
\item
In the case of rank $2$ distributions we already have Theorem \ref{rank2thm}, so Theorem \ref{thmmainsec8.5} does not add anything.

\item
In the case of rank $3$ distributions by the first sentence of this remark and Corollary \ref{finitetyperank3} all assumptions in Theorem \ref{thmmainsec8.5} can be replaced by the assumption that the symplectically flat distribution is of maximal class, has $6$-dimensional square, and the corresponding skew Young diagram is not rectangular.

\item
In the case of rank $4$ all assumptions in Theorem \ref{thmmainsec8.5} can be replaced by the assumption that the symplectically flat distribution is of maximal class, the skew Young diagram of $\mathbb Z$-graded component of its symbol is nonrectangular and has at least $3$ columns with $2$ boxes and the skew diagram of $\cfrac{1}{2}\mathbb Z_{{\rm odd}}$-graded component is a row with at least $6$ boxes. Note that the minimal dimension of the ambient space $M$, when such situation happens is $11$ and it corresponds to the Jacobi symbol $\mathbb R \left(\delta_{3,2}\oplus\tau_{\frac{5}{2}}\right)$.
\end{enumerate}
\end{remark}

 Now, following the same line, we show that in the case of Jacobi symbols of rank $3$ distributions of maximal class with $6$-dimensional square and non-rectangular skew Young diagram  one can obtain more simple
 %not only estimate but the exact
 description of spaces $\mathfrak u^k\bigl(\eta, \mathfrak u^F(\delta_{\rm mod})\bigr)$ in terms of a space of certain polynomials vanishing on some projective variety in a projective space of smaller dimension.

 %First assume that $\delta=\delta_{\mod}=\delta_{s,l}$ with $s< l<2s$, i.e. such that the Jacobi symbol satisfies the conditions of finiteness of type of Theorem \ref{finitetype} and the corresponding Young diagram is not rectangular.
 In the considered case $r=1$, $X=E_1\oplus F_1$, and $\delta=\delta_{\mod}=\mathbb R\delta_{s,l}$ for some positive integer $s$ and  $s< l<2s$. We also fix the skew Young tableaux as in \ref{youngodd1}
 %We will use the notation of item b) of the proof of Lemma \ref{tensoroff} with $X=X_{s,l}$ and $E_{s,l}$, $F_{s,l}$ as in \eqref{EFrs},  $ \{e_i\}_{s-l\leq i\leq s}$ and  $\{f_i\}_{=s\leq i\leq l-s}$ are as in \eqref{tuple1}.
 Recall that $F_1$ is an invariant subspace of  $\delta$. Let $\mathfrak C$ be the (regular) rational normal curve in $\mathcal P F_1$, which is the closure of the orbit of the line  represented by the vector  $f_{l-s}$ under the one-parametric group $e^{\delta}$. The $j$th tangential developable $\mathcal T^j\mathfrak C$ is by definition  (the algebraic closure of) the union of the $j$ osculating subspaces $\mathfrak C^{(i)}$ of the curve $\mathfrak C$.
 Obviously $\mathcal T^0\mathfrak C=\mathfrak C$.
 Note that in our case $i_0=1$. Also by constructions we have the following relation

 \begin{equation}
 \label{tangentrel}
 \mathcal T^{l-s-1}\mathfrak C=
{\mathcal F}_{\delta}^{1}
 \cap F_1,
 \end{equation}
where, as before, ${\mathcal F}_{\delta}$ denotes the flat curve with the flag symbol $\delta$ (see Definition \ref{flatcurve}).
\begin{lemma}
\label{fulton}
In the case $\delta_{\mod}=\mathbb R\delta_{s,l}$ with $s< l<2s$ the subspace $\mathfrak p$ can be identified with the space $\mathcal I_2(\mathcal T^{l-s-1}\mathfrak C)$ of all quadratic forms on $F_1$ that vanishes on the $(l-s-1)$st tangential developable $\mathcal T^{l-s-1}\mathfrak C$ of the rational normal curve $\mathfrak C$.
\end{lemma}

\begin{proof} Following part b) of the proof of Lemma \ref{tensoroff}, we consider $S^2(E_1)$
%^{\mathfrak{sp}}
as a subspace of   $\Hom(F_1, E_1)$ after identification of $F_1^*$ with $E_1$ via the symplectic form $\sigma$ and   let $\mathfrak{s}(E_1)
%{\mathfrak{sp},
%1
%}
$
%($\mathfrak{l}_{s;l}^{\mathfrak{so},1}$)
be  the maximal $\sll_2$-submodule of
%$V_{r_2s_2}\otimes (V_{r_1s_1})^*$
$S^2(E_1
%^{\mathfrak{sp}}
)$
%($\wedge^2(E_{s;l}
%^{\mathfrak{so}}
%)$)
%considered as a subspace of  the  graded space
%$\Hom(F_{rs}^\varepsilon,E_{rs}^\varepsilon)\subset \gl(V_{rs}^\varepsilon)$
concentrated in the non-negative degree part of $S^2(E_1
%^{\mathfrak{sp}}
)$. From the same part b) of the proof of Lemma \ref{tensoroff} the subspace $\pg$ coincides with the subspace $\mathfrak s(E_1)$. Then, using the same arguments as in the proof of Lemma \ref{L(X)I2} and the relation \eqref{tangentrel}, we get that $\widetilde{\pg}\subset \mathcal I_2(\mathcal T^{l-s-1}\mathfrak C)$.

On the other hand, by the same arguments as in the proof of inclusion \eqref{I2cond} we get that the space $I_2(\mathcal T^{l-s-1}\mathfrak C)$ can be identified with the maximal $\sll_2$-submodule of
%$V_{r_2s_2}\otimes (V_{r_1s_1})^*$
$S^2(E_1
%^{\mathfrak{sp}}
)$
%($\wedge^2(E_{s;l}
%^{\mathfrak{so}}
%)$)
%considered as a subspace of  the  graded space
%$\Hom(F_{rs}^\varepsilon,E_{rs}^\varepsilon)\subset \gl(V_{rs}^\varepsilon)$
concentrated in the degrees greater or equal to $-1$.
%part of $S^2(E_{s,l}
%^{\mathfrak{sp}}
%)$.
 However, from decomposition \eqref{Vrseps} of $S^2(E_1
%^{\mathfrak{sp}}
)$ into irreducible $\sll_2$ modules it follows that the part of this decomposition concentrated in nonnegative degree is equal to the part of this decomposition concentrated in degree greater or equal than $-1$. This completes the proof of the lemma.
\end{proof}

As a direct consequence of Theorem \ref{genpr_i0}, the last sentence of Lemma \ref{secantlem} and Lemma \ref{fulton}, we get the following

\begin{theorem}
\label{thmmainsec81}
Assume that the flag symbol $\delta_{\mod}=\mathbb R\delta_{s,l}$ with $s< l<2s$.
Then the $k$th  Tanaka algebraic prolongation $\mathfrak u^k\bigl(\eta, \mathfrak u^F(\delta_{\rm mod})\bigr)$ of the pair $\bigl(\eta, \mathfrak u^F(\delta_{\rm mod})\bigr)$  can be identified with the space $\mathcal I_{k+2}\bigl(S^k(\mathcal T^{l-s-1}\mathfrak C)\bigr)$ of all polynomials of degree $k+2$  on $F_{s,l}$ that vanishes on the $k$th secant variety of the $l-s-1$st tangential developable $T^{l-s-1}\mathfrak C$ of the regular rational normal curve $\mathfrak C$ on $\mathbb P F_{s,l}$.
%Moreover, if $\widetilde W=\mathcal I_2(\mathcal E)$, then $W^{(k)}\cong I_{k+2}\bigl(S^k\mathcal E\bigr)$.
\end{theorem}

\begin{remark} In order to prove Theorem \ref{thmmainsec81} we can also proceed as follows: prove that $\mathfrak p\subset\mathcal I_2(\mathcal T^{l-s-1}\mathfrak C)$, as in the first paragraph of the proof of Lemma \ref{fulton}. Then observe that the space  $\mathcal I_2(\mathcal T^{l-s-1}\mathfrak C)$ can be naturally embedded into the space  $\mathcal I_2\bigl(\mathcal E(\delta_{\mod})\bigr)$ by extension of a quadratic form from $F_1$ to $X$ such that $E_1$ is the kernel of this extension. Then from \eqref{I2above} we have that $\mathcal I_2(\mathcal T^{l-s-1}\mathfrak C)\subset \mathfrak l(X)$  and Theorem \ref{thmmainsec81} follows from Theorem \ref{genpr_i0}, Corollary \ref{corl(X)}, and the last statement of Lemma \ref{secantlem}.
\end{remark}

Let us consider the case $l=s+1$ in more detail. First note that  $\dim \mathbb P F_{s,s+1}=s+1$.
Let $(x_1, \ldots x_{s+2})$ are coordinates in $F_{s, s+1}$ with respect to the basis $(f_1, f_0, \ldots ,f_{-s})$.
We have a well-known
description of $k$-secant varieties of the rational normal
curve $\mathfrak C$:
%as in \eqref{ratnormcurve} with $r=s+1$:
\begin{lemma}[\cite{Harris}]\label{harris}
%Let $\V=\S^k\mqthfrak C$ be the $k$-secant variety of $\mathfrak C$, where $k\ge0$.
The ideal $\mathcal I(\S^k\mathfrak C)$ is generated (as an ideal) by $I_{k+2}(\S^k\mathfrak C)$.
The space $I_{k+2}(\S^k\mathfrak C)$, in its turn, is generated (as a linear
space) by all rank $k+2$ minors of the matrix:
\[
\begin{pmatrix}
x_1 & x_2 & x_3 & \dots & x_{s+2-\alpha} \\
x_2 & x_3 & x_4 & \dots & x_{s+3-\alpha} \\
\vdots \\
x_{\alpha+1} & x_{\alpha+2} & x_{\alpha+3} & \dots & x_{s+2}
\end{pmatrix},
\]
where $\alpha$ is an arbitrary integer between $k+2$ and $s-k-1$.
\end{lemma}
In particular, we see that in the considered case  the space $\pg^{(k)}$,
$k\ge0$ coincides with the space $I_{k+2}(\S^k \mathfrak C)$ and is described
explicitly by Lemma~\ref{harris}.

Finally, in order to complete the case of rank $3$ distributions of maximal class with $6$-dimensional square consider the case of Jacobi symbol with rectangular Young diagram:

\begin{theorem}
\label{rank3thm}
If a $(3,n)$- distribution of maximal class with $6$-dimensional square has Jacobi symbol
with rectangular diagram , then the following statements hold:

\begin{enumerate}
\item
The dimension $n$ of the ambient manifold $M$ is even and the (modified) Jacobi symbol is isomorphic to $\mathbb R\delta_{s, 2s}$ with $s=\cfrac{n-4}{2}$.

\item The algebra $\mathfrak u^F(\delta_{mod})$ is isomorphic to the direct product $\sll_2\times \gl_2$ acting on $X\cong E_1\otimes \mathbb R^2$ componentwise such that $\sll_2$  acts irreducibly on $E_1$ and the action of $\gl_2$ on $\mathbb R^2$ is standard. %The algebra $\mathfrak u(\eta, \mathfrak u^F(\delta_{mod}))$ satisfies the following:

\item \cite{yamag} If $n=6$, then $\mathfrak u(\eta, \mathfrak u^F(\mathbb R\delta_{1,2})$ is isomorphic to the algebra $\mathfrak{so}(4,3)$ (as also expected from \cite{Bryant, BryantPhd});

\item  If $n$ is even and greater than $6$, then the first prolongation $\mathfrak u^1\bigl(\eta, \mathfrak u^F(\mathbb R\delta_{\frac{n-4}{2}, n-4})\bigr)$ of the pair $\bigl(\eta, \mathfrak u^F(\mathbb R\delta_{\frac{n-4}{2}, n-4})\bigr)$ vanishes, i.e. $\mathfrak u(\eta, \mathfrak u^F(\mathbb R\delta_{\frac{n-4}{2}, n-4}))$ is isomorphic to the semidirect sum of $\sll_2\times \gl_2$ and  $X\cong E_1\otimes \mathbb R^2$ (with the action of $\sll_2\times \gl_2$ on $E_1\otimes \mathbb R^2$ described above).
\end{enumerate}
\end{theorem}
\begin{proof}
For item (1) as was already discussed in section 4, paragraph before Remark \ref{6dim} a $(3,n)$-distribution of maximal class has Jacobi symbol isomorphic to $\mathbb R\delta_{s, n-4}$ with nonnegative integer $s$ satisfying  $\frac{n-4}{2}\leq s\leq n-4$. This symbol has rectangular diagram if $s=\frac{n-4}{2}$, which implies that $n=2s+4$, i.e. $n$ is even.

Item (2) follows from part b) of the proof of Lemma \ref{tensoroff} applied for the case of rectangular Young diagram: take equations \eqref{Y1=Y2} and \eqref{ngrseps2}, first line there) and plug then into formulas \eqref{semidir} and \eqref{speq}.

For item (3)  we refer to \cite{yamag}, where all cases of semisimple universal algebraic prolongations are considered, including our case here.

 Let us prove item (4). The flag symbol in the considered situation satisfies all assumptions of Theorem \ref{genpr_i0}: we are in the case of $\mathbb Z$ grading and  each row of the corresponding Young diagram consist of at least $5$ boxes. Therefore to prove item (4) it is enough to show that
 \begin{equation}
 \label{p10}
 \pg^{(1)}=0.
 \end{equation}
Note that from part b) of the proof of Lemma \ref{tensoroff} it follows that $\pg$ is $2$-dimensional and, if the Young tableaux of our symbol is as in \eqref{youngodd1} then $\pg$ is spanned by endomorphisms $p_1$ and $p_2$ of $X$ defined as follows:
$$p_1(f_i)=e_i, p_2(e_i)=f_i, \quad \forall -s\leq i\leq s;  p_1|_{E_1}=0, p_2|_{F_1}=0.$$
 In the terminology used in item (2) of the Theorem $p_1$ and $p_2$ are the  endomorphisms induced by the action of $\left(0, \begin{pmatrix} 0& 0\\1&0\end{pmatrix}\right)$ and $\left(0, \begin{pmatrix} 0& 1\\0&0\end{pmatrix}\right)$ (considered as element of  $\sll_2\times \gl_2$) on $X\cong E_1\otimes \mathbb R^2$. Keeping this in mind, the relation \eqref{p10} can be checked directly from the definition of the first standard prolongation and this check is left to the reader.
\end{proof}

\end{document}